\documentclass[11pt]{amsart}
\usepackage[dvipsnames]{xcolor}
\usepackage[utf8]{inputenc}
\usepackage{amsmath}
\usepackage{mathtools}
\usepackage{amsthm}
\usepackage{amsfonts}
\usepackage{amssymb}

\usepackage[margin=1in]{geometry}
\usepackage{tikz}

\usepackage{hyperref}
\usepackage[capitalize]{cleveref}
\usepackage[shortlabels,inline]{enumitem}

\numberwithin{equation}{section}

\DeclarePairedDelimiterX{\norm}[1]{\lVert}{\rVert}{#1}

\newtheorem{theorem}{Theorem}[section]
\newtheorem*{theorem*}{Theorem}
\newtheorem{lemma}[theorem]{Lemma}
\newtheorem{proposition}[theorem]{Proposition}
\newtheorem*{proposition*}{Proposition}

\newtheorem{corollary}[theorem]{Corollary}
\theoremstyle{definition}
\newtheorem{example}[theorem]{Example}

\newtheorem{definition}[theorem]{Definition}

\newcommand{\R}{\mathbb{R}}
\newcommand{\Q}{\mathbb{Q}}
\newcommand{\N}{\mathbb{N}}
\newcommand{\Z}{\mathbb{Z}}

\newcommand{\eps}{\varepsilon}
\newcommand{\Len}{\mathrm{L}}
\newcommand{\isom}{\cong}

\newcommand{\isomto}{\xrightarrow{\isom}}
\DeclareMathOperator{\Int}{Int}

\newcommand\mydots{\makebox[3em][c]{.\hfil.\hfil.}}

\DeclareMathOperator{\Cost}{Cost}
\DeclareMathOperator{\supp}{supp}
\DeclareMathOperator{\asdim}{asdim}
\DeclareMathOperator{\im}{\operatorname{im}}
\DeclareMathOperator{\diam}{diam}
\DeclarePairedDelimiter{\abs}{\lvert}{\rvert}
\DeclareMathOperator{\perimeter}{perimeter}
\DeclareMathOperator{\CAT}{CAT}
 
\begin{document}

\begin{abstract}
  Motivated by persistent homology and topological data analysis, we consider formal sums on a metric space with a distinguished subset. These formal sums, which we call persistence diagrams, have a canonical 1-parameter family of metrics called Wasserstein distances. We study the topological and metric properties of these spaces. Some of our results are new even in the  case of persistence diagrams on the half-plane. 
  Under mild conditions, no persistence diagram has a compact neighborhood.
  If the underlying metric space is $\sigma$-compact then so is the space of persistence diagrams. However, under mild conditions, the space of persistence diagrams is not hemicompact and the space of functions from this space to a topological space is not metrizable.
  Spaces of persistence diagrams inherit completeness and separability from the underlying metric space.
  Some spaces of persistence diagrams inherit being path connected, being a length space, and being a geodesic space, but others do not.
  We give criteria for a set of persistence diagrams to be totally bounded
and relatively compact.
 We also study the curvature and dimension of spaces of persistence diagrams and their embeddability into a Hilbert space.
As an important technical step, which is of independent interest,
 we give necessary and sufficient conditions for the existence of optimal matchings of persistence diagrams.
\end{abstract}

\title{Topological and metric properties of spaces of generalized persistence diagrams}
\author{Peter Bubenik}
\address{Department of Mathematics, University of Florida, Gainesville, USA}
\email{peter.bubenik@ufl.edu}
\author{Iryna Hartsock}
\address{Department of Machine Learning, H. Lee Moffitt Cancer Center \& Research Institute, Tampa, USA}
\email{iryna.hartsock@moffitt.org}


\maketitle

\setcounter{tocdepth}{1}
\tableofcontents

\section{Introduction}

We study the topological and metric properties of spaces of formal sums on metric pairs.
A metric pair, $(X,d,A)$, consists of a metric space $(X,d)$ and a subset $A \subset X$.
A formal sum on $(X,d,A)$ is an equivalence class of a formal sum on $X$ modulo formal sums on $A$.

Such sums arise in persistent homology. Indeed,
barcodes~\cite{Collins:2004} are formal sums on the metric pair $(\Int(\R),d,\emptyset)$, where $\Int(\R)$ denotes the set of intervals in $\R$ and $d$ is either the Hausdorff distance or the length of the symmetric difference, and
persistence diagrams~\cite{cseh:stability,csehm:lipschitz} are formal sums on the metric pair $(\overline{\R}^2_{\leq},d,\overline{\Delta})$, where $\overline{\R}^2_{\leq}$ denotes the set of ordered pairs $(x,y)$, where $x,y \in [-\infty,\infty]$ and $x \leq y$, $d$ is a metric induced by a norm on $\R^2$, and $\overline{\Delta}$ denotes the subset of ordered pairs with $x=y$.
Note that to accommodate these basic examples, we need to relax the usual notion of metric to allow infinite distances and remove the requirement that $d(x,y)=0$ implies that $x=y$.
Similarly, we have formal sums on the metric pairs $(\Int(\Q),d,\emptyset)$, $(\R^2_{\leq},d,\Delta)$,
$(\Q^2_{\leq},d,\Delta\cap\Q)$, and
$(\overline{\Q}^2_{\leq},d,\overline{\Delta}\cap\overline{\Q})$.
Additional cases arise in the study of persistence modules indexed by posets~\cite{Kim:2023}.
In this setting, an interval-decomposable persistence module can be described by a formal sum on the set of intervals in the poset
\cite{Skryzalin:2017,Botnan:2018,Cochoy:2020,Bjerkevik:2021,MR4334502}.
More generally, in any essentially small Krull-Schmidt category, every object can be described by a formal sum on a set of objects having local endomorphism rings. In these examples, the set $A$ consists of the zero object.
In measure theory, measures on a metric space that are sums of Dirac measures may be described by a formal sum on this metric space.
After rescaling, empirical measures on a metric space are examples of such measures.
In these examples, the set $A$ is given by a distinguished subset of the metric space.


The set of formal sums on a metric pair, which we denote $D(X,A)$, has a canonical and universal family of metrics $W_p$, $p \in [1,\infty]$, called Wasserstein distances~\cite{bubenik2019universality}.
We call $(D(X,A),W_p)$, a space of persistence diagrams.
These metrics canonically extend to countable formal sums~\cite{bubenik2020virtual}.

In topological data analysis, persistence diagrams summarize the shape of data and are used for statistical analysis and machine learning~\cite{Rabadan:2019,dey_wang_2022,Carlsson:2022}.
The topological and metric properties of the space of persistence diagrams both facilitate and restrict this analysis,
motivating our work.

\subsection*{Our contributions}

The following definitions are used in stating our results.
We say that the subset $A \subset X$ is \emph{distance minimizing} (\cref{def:distance-minimizing}) if for all $x \in X$, there exists an $a \in A$ such that $d(x,A) = d(x,a)$.
Unless it is explicitly stated, the subset $A$ does not need to be distance minimizing.
We say that $A$ is \emph{isolated} (\cref{def:isolated}) if there exists a $\delta > 0$ such that for all $x \in X$, $d(x,A) < \delta$ implies that $x \in A$.
A countable formal sum is \emph{essentially $p$-finite} (\cref{def:D_p}) if for all $\delta > 0$ its restriction to the subset of $X$ whose distance to $A$ is at least $\delta$ has finite support, and its restriction to the subset within distance $\delta$ of $A$ has finite p-Wasserstein distance to 
zero.
Our main objects of study are the metric space $(D(X,A),W_p)$, its subset $(D^n(X,A),W_p)$, consisting of formal sums of cardinality at most $n$, and its extension $(\overline{D}_p(X,A),W_p)$, consisting of countable formal sums that are essentially $p$-finite.

We start with proving some basic metric properties (\cref{sec:basic-metric}) and topological properties (\cref{sec:basic-topological}) of these spaces of persistence diagrams.
The Wasserstein distances are given by an infimum over matchings between persistence diagrams (\cref{def:wasserstein}). The existence of an optimal matching realizing this distance is a central question in optimal transportation theory.

\begin{theorem*}[\cref{thm:optimal-matching}]
Given a metric pair $(X,d,A)$, there exists an optimal matching between any two persistence diagrams if and only if $A$ is distance minimizing.
\end{theorem*}


Having a path-connected space is important for stochastic algorithms such as stochastic gradient descent and the Metropolis-Hastings algorithm.

\begin{theorem*}[\cref{thm:path-connected,ex:circles}]
  Assume $(X,d)$ is path connected and let $p \in [1, \infty]$. Then so are $(D(X,A),W_p)$ and $(D^n(X,A),W_p)$. However, $(\overline{D}_p(X,A),W_p)$ need not be path connected.
\end{theorem*}

In length spaces, distances may be approximated by paths.

\begin{theorem*}[\cref{thm:length,prop:length}]
  Assume that $(X,d)$ is a length space. Then, for $p \in  [1, \infty]$, so are $(D(X,A),W_p)$, $(\overline{D}_p(X,A),W_p)$, and $(D^n(X,A),W_1)$. However, for $p \neq 1$, $(D^n(X,A),W_p)$ need not be a length space.
\end{theorem*}

The existence of geodesics is important for statistics, in order to find means, for example.

\begin{theorem*}[\cref{thm:geodesic,prop:X/A_geodesic,prop:unique_geodesic}]
  Assume that $(X,d)$ is a geodesic space and that $A$ is distance minimizing. Then, for $p \in  [1, \infty]$,  $(D(X,A),W_p)$, $(\overline{D}_p(X,A), W_p)$, and $(D^n(X,A),W_1)$ are also geodesic spaces. However, for $p \neq 1$, $(D^n(X,A),W_p)$ need not be a geodesic space.
  Furthermore, for $p \in  [1, \infty]$, if either of $(D(X,A),W_p)$ or $(\overline{D}_p(X,A),W_p)$ have a pair of persistence diagrams with distinct optimal matchings then they do not have unique geodesics.
\end{theorem*}

In order to apply probability theory to persistence diagrams, we would like our spaces to be Polish~\cite{Billingsley:1999}.
Note that $X/A$ includes in $(D(X,A),W_p)$ and that we denote the induced (subspace) metric on $X/A$ by $d_p$ (\cref{def:dp,lem:isomorphism}).  

\begin{theorem*}[\cref{prop:separability}]
  Let $p \in  [1, \infty]$. If $(X,d)$ is separable then so is $(\overline{D}_p(X,A),W_p)$.
  Conversely, if $(\overline{D}_p(X,A),W_p)$ is separable then so is $(X/A,d_p)$. 
\end{theorem*}

\begin{theorem*}[\cref{thm:completeness_inf,thm:completeness_p,prop:completion,prop:complete}]
   Let $p \in  [1, \infty]$. If $(X,d)$ is complete then so is $(\overline{D}_p(X,A),W_p)$.
  Furthermore, $(\overline{D}_p(X,A),W_p)$ is the Cauchy completion of $(D(X,A),W_p)$.
  Conversely, if $(\overline{D}_p(X,A),W_p)$ is complete then so is $(X/A,d_{p})$.
\end{theorem*}

In order to apply functional analysis and machine learning to persistence diagrams, a suitable notion of compactness is often required.

\begin{proposition*}[\cref{prop:D^n(X)_is_cmpt}]
  If $(X,d)$ is compact and $p \in  [1, \infty]$ then so is $(D^n(X,A),W_p)$.
\end{proposition*}

\begin{theorem*}[\cref{thm:not_loc_compactness,thm:not_loc_compactness_D_p}]
  Assume that $A$ is not isolated and let $p \in  [1, \infty]$.
  Then every persistence diagram in $(D(X,A),W_p)$ does not have a compact neighborhood.
  Similarly, every persistence diagram in $(\overline{D}_p(X,A),W_p)$ does not have a compact neighborhood.
  Therefore, $(D(X,A),W_p)$ and $(\overline{D}_p(X,A),W_p)$ are not locally compact and any compact set in these spaces has empty interior.
\end{theorem*}

\begin{proposition*}[\cref{prop:sigma-compact}]
  If $(X, d)$ is $\sigma$-compact and $p \in  [1, \infty]$ then $(D(X,A), W_p)$ is $\sigma$-compact.
\end{proposition*}

\begin{theorem*}[\cref{thm:not_hemicompact}]
  If $A$ is not isolated  and $p \in  [1, \infty]$ then $(D(X,A),W_p)$ and $(\overline{D}_p(X,A),W_p)$ are not hemicompact (i.e. in both spaces there is no sequence of compact subsets such that every compact subset of the space lies inside some compact set in the sequence).
\end{theorem*}

In topological data analysis,
one considers continuous functions from spaces of persistence diagrams to a given topological space. For a quantitative analysis of such maps, we would like this function space to be metrizable.

\begin{theorem*}[\cref{thm:not-metrizable}]
  Let $(X,d,A)$ be a metric pair with $A$ not isolated, let $Y$ be a topological space that contains a nontrivial path, and let $p \in  [1, \infty]$.
  Then the function spaces
  $C((D(X,A), W_p), Y)$ and
  $C((\overline{D}_p(X,A), W_p), Y)$
  are not metrizable. 
\end{theorem*}

We remark that the previous three results seem to be new even for the example $(X,d,A) = (\R^2_{\leq},d,\Delta)$.

In various settings,
it is important to characterize relatively compact sets so that one may know when a sequence has a convergent subsequence (e.g.\ Arzel\'a-Ascoli Theorem, Prokhorov's Theorem).
We characterize relatively compact sets and totally bounded sets of persistence diagrams. 
Since $(X,d)$ need not be complete, these notions are distinct.
For example, we obtain a characterization of relatively compact sets of barcodes with rational endpoints and of persistence diagrams consisting of dots with rational coordinates.

\begin{theorem*}[\cref{thm:criterion_tot_bd}]
 Let $p \in  [1, \infty]$. A subset $S \subset (\overline{D}_p(X,A),W_p)$ is totally bounded if and only if it is
  uniformly upper finite (\cref{def:unif_upper_finite}),
  upper totally bounded (\cref{def:upper_tot_bd}),
  and uniformly $p$-vanishing (\cref{def:lower_uniform}).
\end{theorem*}

\begin{theorem*}[\cref{thm:criterion_rel_compct}]
  Let $p \in  [1, \infty]$. A subset $S \subset (\overline{D}_p(X,A),W_p)$ is relatively compact if and only if it is
  uniformly upper finite, 
  upper relatively compact (\cref{def:upper_compact}),
  and uniformly $p$-vanishing. 
\end{theorem*}

One may apply geometric methods to a metric space if it has curvature bounded above or below (i.e. is an Alexandrov space)~\cite{burago2001course,Bridson:1999}.

\begin{theorem*}[\cref{cor:cat}]
  Assume that $(X,d)$ is separated and geodesic,  $A$ is distance minimizing, and there exists persistence diagrams with distinct optimal matchings. Let $p \in  [1, \infty]$.
  Then $(\overline{D}_p(X,A),W_p)$ is not a $\CAT(k)$ space for any $k$.
\end{theorem*}

\begin{theorem*}[\cref{prop:curvature}]
  If $(X,d,A)$ is a geodesic Alexandrov space with curvature bounded from below by zero, $A$ is distance minimizing, and $p \in  [1, \infty]$ then $(\overline{D}_2(X,A),W_2)$ is a geodesic Alexandrov space with curvature bounded from below by zero.
\end{theorem*}

We show that our spaces of persistence diagrams are often infinite dimensional.

\begin{theorem*}[\cref{thm:top_dimension}]
  Under mild conditions, $(D(X,A),W_p)$ has infinite Lebesgue covering dimension for $p \in  [1, \infty]$.
\end{theorem*}

\begin{theorem*}[\cref{cor:infinite-hausdorff-dimension}]
 If $(X,d,A)$ is a geodesic Alexandrov space with non-negative curvature  and $A$ is distance minimizing and not isolated then $(\overline{D}_2(X,A),W_2)$ has infinite Hausdorff dimension.
\end{theorem*}

Since kernel methods in statistical learning and machine learning use a mapping to a Hilbert space~\cite{Scholkopf:2001,Steinwart:2008}, it is important to know to what extent such maps can preserve a given metric structure.
Since our spaces of persistence diagrams often have distinct geodesics, they cannot be isometrically embedded into a Hilbert space.
However, if a metric space has only finite distances and has finite asymptotic dimension then it coarsely embeds into a Hilbert space.

\begin{proposition*}[\cref{prop:coarse_emb_D^n_pointed_case}]
  Assume that $(X,d)$ is proper, that all distances are finite, and that $(X,d)$ has finite asymptotic dimension. Let $x_0 \in X$ and $p \in  [1, \infty]$. Then $(D^n(X,x_0),W_p)$ can be coarsely embedded into a Hilbert space.
\end{proposition*}

However, if we do not bound the cardinality of our persistence diagrams then we often have infinite asymptotic dimension. 

\begin{proposition*}[\cref{prop:asdim_is_infty}]
  Assume that $(X,d)$ is a geodesic space, $(X/A,d_1)$ is unbounded and proper, and that $A$ is distance minimizing. Let $p \in  [1, \infty]$.
  Then  $(D(X,A), W_p) $ has infinite asymptotic dimension.
\end{proposition*}

\subsection*{Related work}

Many of our results are indebted to previous work studying the topological and metric properties of formal sums on the metric pair $(\overline{\R}^2_{\leq},d,\overline{\Delta})$ and its subspaces,
and our proofs often follow these previous arguments.
%
Mileyko, Mukherjee and Harer~\cite{mileyko2011probability} considered
$(\overline{D}_p(\R^2_{\leq},\Delta),W_p)$
for $p \in [1,\infty)$.
They showed that it
is complete and separable and gave a characterization of its totally bounded sets, which are also its relatively compact sets.
Blumberg et al~\cite{blumberg2014robust} considered
$(\overline{D}_{\infty}(\overline{\R}^2_{\leq},\overline{\Delta}),W_{\infty})$.
They showed that it is complete and separable and that it is the Cauchy completion of $(D(\overline{\R}^2_{\leq},\overline{\Delta}),W_{\infty})$.
Turner et al~\cite{turner2014frechet} considered
$(\overline{D}_2(\overline{\R}^2_{\leq},\overline{\Delta}),W_2)$.
They showed that it has optimal matchings, that it is not a $\CAT(k)$ space for $k>0$, and that it is a non-negatively curved Alexandrov space.
Perea, Munch and Khasawneh~\cite{perea2019approximating} considered
$(\overline{D}_{\infty}([0,\infty)^2_{\leq},\Delta),W_{\infty})$.
They gave a characterization of its relatively compact sets and showed that all of its persistence diagrams do not have a compact neighborhood.
They also showed that the function space
$C((\overline{D}_{\infty}([0,\infty)^2_{\leq},\Delta),W_{\infty}),\R)$ is not metrizable.
K.\ Turner~\cite{turner2013medians} considered
$(\overline{D}_p(\overline{\R}^2_{\leq},\overline{\Delta}),W_p)$, for $p \in [1,\infty]$.
She proved this space is geodesic and is not $\CAT(k)$ for any $k > 0$.
Chowdhury~\cite{Chowdhury:2019a} studied geodesics in $(\overline{D}_p(\R^2_{\leq},\Delta),W_p)$ and whether or not they could be uniquely characterized as convex combinations.

Bubenik and Wagner~\cite{bubenik2020embeddings} showed that $(\overline{D}_{\infty}(\R^2_{\leq},\Delta),W_{\infty})$ does not coarsely embed into Hilbert space.
This result was extended to
$(D(\R^2_{\leq},\Delta),W_{\infty})$ by
Mitra and Virk~\cite{mitra2019space} and to
$(D(\R^2_{\leq},\Delta),W_p)$ for $p \in (2,\infty)$ by
A.\ Wagner~\cite{wagner2019nonembeddability}. 
Bubenik and Vergili~\cite{bubenik2018topological}
studied the topological properties of the space of $\R$-indexed persistence modules equipped with the interleaving distance.
The study of persistence diagrams on metric pairs was initiated by Bubenik and Elchesen~\cite{bubenik2019universality}.

Che, Galaz-Garc\'ia, Guijarro, and Mebrillo Solis ~\cite{che2021metric} were the first to study the metric properties of formal sums on metric pairs.
  They consider metric spaces $(X,d)$ in which $d(x,y)< \infty$ for all $x,y \in X$ and for which $d(x,y)>0$ if $x \neq y$ and $p \in [1,\infty)$.
  They show that if $(X,d)$ is complete or separable then so is $(\overline{D}_p(X,A),W_p)$.
  If $X$ is proper, a stronger condition than $A$ being distance minimizing 
  (see Lemma \ref{lem:A_dist_minimizing} and Example \ref{ex:non-proper_and_A_minimizing}),
  they show that if $X$ is geodesic then so is $(\overline{D}_p(X,A),W_p)$ and if $X$ is an Alexandrov space with curvature bounded below by zero then so is $(\overline{D}_2(X,A),W_2)$.
  Assuming that $X$ is proper, they give a criterion for compact subsets of $(\overline{D}_p(X,A),W_p)$. Also, they show that $(\overline{D}_2(\R^2,\Delta),W_2)$ has infinite topological dimension, Hausdorff dimension, and asymptotic dimension. 
  In addition, they study the sequential continuity of the mapping from metric pairs to persistence diagrams and the space of directions of $(\overline{D}_2(X,A),W_2)$ in the case that $X$ is an Alexandrov space with non-negative curvature.


\section{Persistence diagrams and Wasserstein distance}

In this section we introduce necessary background and notation.
	
\subsection{Commutative monoids and persistence diagrams.} \label{sec:monoids}
	
We start by defining our main objects of study: certain free commutative monoids whose elements we call persistence diagrams.

Let $\Z_+$ denote the set of non-negative integers.
It is a free commutative monoid under addition. Also let $\N= \Z_+ \setminus \{0\}$.
	Let $X$ be a set.
	Let $\Z_+^X$ denote the set of functions $f:X \to \Z_+$. 
	Then $\Z_+^X$ is a commutative monoid under pointwise addition with neutral element the zero function $0$. Note that $\Z_+^X= \prod_{x \in X} \Z_+$, the product in the category of commutative monoids.
	For $f \in \Z_+^X$ the \emph{support} of $f$ is given by $\supp(f) = \{x \in X \, | \, f(x) > 0\}$.
	There is a canonical inclusion of $X$ in $\Z_+^X$ given by mapping $x \in X$ to its indicator function, $1_x$, given by $1_x(y) = 1$, if $x=y$, and $1_x(y)=0$, otherwise. We denote $1_x$ by $x$. Let $f \in \Z^{X}_+$. It is often convenient to write $f$ as an indexed sum, $f=\sum_{i \in I} x_i$. Define the cardinality of $f$ to be the cardinality of the set $I$. The following cases, where $I$ is ordered, are particularly useful, $f=\sum_{i=1}^n x_i$ and $f=\sum_{i=1}^{\infty} x_i$.
	For $n \in \N$, let $D^n(X)$ be the subset of $\Z_+^X$ consisting of elements of cardinality at most $n$.
	Similarly, let $D(X)$ be the submonoid $\Z_+^X$ of elements of finite cardinality and let $\overline{D}(X)$ be the submonoid $\Z_+^X$ of elements of countable cardinality. Note that $D(X)=\bigoplus_{x \in X} \Z_+$ which is the coproduct in the category of commutative monoids.
	We have inclusions $X \isomto D^1(X) \subset D(X) \subset \overline{D}(X)$.
	
	Let $A \subset X$. Then $f:A \to \Z_+$ has a canonical extension to a function $f:X \to \Z_+$ with the same support. Thus we may identify $\Z_+^A$ as a submonoid of $\Z_+^X$.
\begin{definition}
		Let $f,g:X \to \Z_+$. Define $f \sim g$ if $f|_{X \setminus A} = g|_{X \setminus A}$.
		Then $f \sim g$ and $h \sim k$ implies that $f+h \sim g+k$.
		That is, the relation is a congruence.
		The quotient monoid, $\Z_+^X/{\sim}$, is also denoted $\Z_+^X/\Z_+^A$ and we will usually denote it by $\Z_+^{X,A}$.
		We call the elements of $\Z_+^{X,A}$ \emph{persistence diagrams}  and the element $0$ the \emph{empty diagram}.
	\end{definition}
	There is an isomorphism $\rho: \Z_+^{X,A} \to \Z_+^{X \setminus A}$ given by $\rho([f]) = f|_{X \setminus A}$ and $\rho^{-1}(g) = [\overline{g}]$, where $\overline{g}$ is any extension of $g$ to $X$.
\begin{definition}	For $\alpha \in \Z_+^{X,A}$, we write $\hat{\alpha}=\rho(\alpha) $ and define the cardinality of $\alpha$ to be the cardinality of $\hat{\alpha}$, denoted $\abs{\alpha}$.
\end{definition}
	Note that $D(X) \cap \Z_+^A = D(A)$ and $\overline{D}(X) \cap \Z_+^A = \overline{D}(A)$.
\begin{definition}
	We define $D(X,A)$ to be the quotient monoid $D(X)/D(A)$ and $\overline{D}(X,A)$ to be the quotient monoid $\overline{D}(X)/\overline{D}(A)$.
	The elements of $D(X,A)$ and $\overline{D}(X,A)$ are \emph{finite persistence diagrams} and \emph{countable persistence diagrams} respectively.
\end{definition}
	Notice that the previously defined map $\rho$ gives isomorphisms $D(X,A) \isomto D(X \setminus A)$ and $\overline{D}(X,A) \isomto \overline{D}(X \setminus A)$.
	We have submonoids $D(X,A) \subset \overline{D}(X,A) \subset \Z_+^{X,A}$.
	For $n \in \N$, let $D^n(X,A) = \rho^{-1}D^n(X \setminus A)$.
	Then $D^n(X,A) \subset D(X,A)$ is the subset of persistence diagrams of cardinality at most $n$.
	For $\alpha \in \overline{D}(X,A)$, we have $\hat{\alpha}
        \in \overline{D}(X \setminus A)$, 
        with $\hat{\alpha} = \sum_{i \in I} x_i$, where $x_i \in X \setminus A$ for all $i \in I$, and $I$ is countable.
        Similarly, for $\alpha \in D(X,A)$, $\hat{\alpha} \in D(X \setminus A)$ and for $\alpha \in D^n(X,A)$, $\hat{\alpha} \in D^n(X \setminus A)$.
        In these two cases, we also have  $\hat{\alpha} = \sum_{i \in I} x_i$, where $x_i \in X \setminus A$ for all $i \in I$, but $I$ is finite in the first case, and has cardinality at most $n$ in the second case.
        If we say \emph{support} of $\alpha$ then we mean the support of $\hat{\alpha}$.
For a persistence diagram $\alpha$, call elements of the support of $\alpha$, \emph{dots} of the persistence diagram.

\begin{example}

Let $\Int(\R)$ denote the set of intervals in $\R$.
Then we have
$D(\Int(\R),\emptyset)$ and $\overline{D}(\Int(\R),\emptyset)$ (where $\emptyset \in \Int(\R)$)
whose underlying sets are the sets
of finite and countable \emph{barcodes}, respectively.
Let $\overline{\R}=[-\infty, \infty]$ and consider $\overline{\R}^2_{\leq} = \{(x,y) \in \overline{\R}^2 \ | \ x \leq y\}$ and $\overline{\Delta} = \{(x,y) \in \overline{\R}^2 \ | \ x = y\}$.
Then we have
$D(\overline{\R}^2_{\leq}, \overline{\Delta})$,
whose underlying set is the set
of \emph{finite classical persistence diagrams}, and
$\overline{D}(\overline{\R}^2_{\leq}, \overline{\Delta})$,
whose underlying set is the set
of \emph{countable classical persistence diagrams}.
If we restrict to $\R^2_{\leq} = \{(x,y) \in \R^2 \ | \ x \leq y\}$ and $\Delta = \{(x,y) \in \R^2 \ | \ x=y\}$,
then we have $D(\R^2_{\leq},\Delta)$
and $\overline{D}(\R^2_{\leq},\Delta)$
whose underlying sets consist of classical persistence diagrams
each of whose dots lies in $\R^2_{\leq}$.
For each of these, we may restrict to rational numbers and or to non-negative numbers.
\end{example}
	
Consider the quotient set $X/A$ and write $A$ for the subset $\{A\} \subset X/A$.
Then $\Z_+^{X/A,A} \isom \Z_+^{X,A}$,
$\overline{D}(X/A,A) \isom \overline{D}(X,A)$,
$D(X/A,A) \isom D(X,A)$, and $D^n(X/A,A) \isom D^n(X,A)$.

\begin{definition} \label{def:matching}
Given $\alpha,\beta \in \overline{D}(X,A)$ define a \emph{matching} of
$\alpha$ and $\beta$ to be an element
$\sigma \in \overline{D}(X \times X , A \times A)$ such that
$(\pi_1)_*(\sigma) = \alpha$ and $(\pi_2)_*(\sigma) = \beta$ where
  $(\pi_{1})_*, (\pi_{2})_*: \overline{D}(X\times X, A \times A) \to \overline{D}(X,A)$
  are induced by the projection maps $\pi_{1}, \pi_2: X\times X \to X$.  That is, for all $x \in X \setminus A$,
    $\hat{\alpha}(x) = \sum_{y \in X}\hat{\sigma}(x,y)$, and
    for all $y \in X \setminus A$,
    $\hat{\beta}(y) = \sum_{x \in X} \hat{\sigma}(x,y)$. 
  
%
Call $\sigma$ a \emph{finite matching} if $\sigma \in D(X \times X, A \times A)$.
\end{definition}	

\begin{lemma}
  \label{lem:matching}
Let $\alpha,\beta \in \overline{D}(X,A)$,
    where $\hat{\alpha} = \sum_{i \in I} x_i$ and $\hat{\beta} = \sum_{j \in J} y_j$.
    Then $\sigma \in \overline{D}(X \times X, A \times A)$
    is a matching of $\alpha$ and $\beta$ if and only if
    there exists a subset $K \subset I$ and an injection $\varphi:K \to J$ such that
    \begin{equation} \label{eq:matching}
      \hat{\sigma} = \sum_{k \in K} (x_k,y_{\varphi(k)}) + \sum_{i \in I \setminus K} (x_i,z_i) + \sum_{j \in J \setminus \varphi(K)} (w_j,y_j),
    \end{equation}
    for some $\{z_i\}_{i \in I \setminus K} \subset A$ and $\{w_j\}_{j \in J \setminus \varphi(K)} \subset A$.
\end{lemma}

\begin{proof}
  Assume that $\sigma$ is matching of $\alpha$ and $\beta$.
    Consider $\hat{\sigma}: (X \times X) \setminus (A \times A) \to \Z_+$.
  Let $\hat{\sigma}_0 = \hat{\sigma}|_{(X \setminus A) \times (X \setminus A)}$,
  $\hat{\sigma}_1 = \hat{\sigma}|_{(X \setminus A) \times A}$, and
  $\hat{\sigma}_2 = \hat{\sigma}|_{A \times (X \setminus A)}$.
  Then $\hat{\sigma} = \hat{\sigma}_0 + \hat{\sigma}_1 + \hat{\sigma}_2$.
  Let $\hat{\alpha}_0 = (\pi_1)_* \hat{\sigma}_0$,
  $\hat{\beta}_0 = (\pi_2)_* \hat{\sigma}_0$,
  $\hat{\alpha}_1 = (\pi_1)_* \hat{\sigma}_1$, and
  $\hat{\beta}_2 = (\pi_2)_* \hat{\sigma}_2$.
  Then $\hat{\alpha} = \hat{\alpha}_0 + \hat{\alpha}_1$,
  $\hat{\beta} = \hat{\beta}_0 + \hat{\beta}_2$, and
  $\abs{\hat{\sigma}_0} = \abs{\hat{\alpha}_0} = \abs{\hat{\beta}_0}$.
  Therefore, there exists $K \subset I$ and injection $\varphi:K \to J$ such that $\hat{\sigma}_0 = \sum_{k \in K} (x_k,y_{\varphi(k)})$.
  Furthermore, $\hat{\alpha}_1 = \sum_{i \in I \setminus K} x_i$ and
  $\hat{\beta}_2 = \sum_{j \in J \setminus \varphi(K)} y_j$.
  Hence, $\hat{\sigma}_1 = \sum_{i \in I \setminus K} (x_i,z_i)$,
  for some $\{z_i\}_{i \in I \setminus K} \subset A$, and
  $\hat{\sigma}_2 = \sum_{j \in J \setminus \varphi(K)} (w_j,y_j)$,
  for some $\{w_j\}_{j \in J \setminus \varphi(K)} \subset A$.
  Therefore, $\hat{\sigma}$ has the desired form \eqref{eq:matching}.

For the converse, assume that $\hat{\sigma}$ has the form \eqref{eq:matching}. Then it is easy to see that $\sigma$ is a matching of $\alpha$ and $\beta$.
\end{proof}

\subsection{Wasserstein distance} \label{sec:wasserstein} 

In the case that the set $X$ has a metric, we will define a metric on the corresponding set of countable persistence diagrams.
To accommodate examples arising in our main applications, such as barcodes and classical persistence diagrams, we use a more general notion of metric than is standard.

\begin{definition} \label{def:metric}
  Let $X$ be a set and $d:X \times X \to [0,\infty]$ such that for all $x \in X$, $d(x,x)=0$, for all $x,y \in X$, $d(x,y)=d(y,x)$, and for all $x,y,z \in X$, $d(x,z) \leq d(x,y) + d(y,z)$. The pair $(X,d)$ is often called an \emph{extended pseudometric space}. For brevity, we will call $d$ a \emph{metric} and $(X,d)$ a \emph{metric space}. If in addition for all $x,y \in X$, $d(x,y)< \infty$ then we say that the metric is \emph{non-extended} and if $d(x,y)=0$ implies that $x=y$ then we say that the metric is \emph{separated}.
\end{definition}

A \emph{metric pair} $(X,d,A)$ consists of a metric space $(X,d)$ and a distinguished subset $A \subset X$. 
For simplicity we will assume throughout that $A$ is closed. For $x_0 \in X$, the special case $(X, d, \{x_0\})$ is called a pointed metric space and denoted $(X, d, x_0)$.

\begin{definition} \label{def:wasserstein}
Let $(X,d,A)$ be a metric pair.
Let $\alpha, \beta \in \overline{D}(X,A)$ and let $\sigma \in \overline{D}(X \times X, A \times A)$ be a matching of $\alpha$ and $\beta$.
Then $\hat{\sigma} = \sum_{i \in I}(x_i,y_i)$, where $I$ is countable.
Let $p \in [1,\infty]$. Define the \emph{$p$-cost} of $\sigma$ to be given by
\begin{equation*}
	\Cost_p[d](\sigma) = \norm{(d(x_i,y_i))_{i \in I}}_p \in [0,\infty].
\end{equation*}
Define the \emph{$p$-Wasserstein distance} from $\alpha$ to $\beta$ to be given by
\begin{equation*}
	W_p[d](\alpha,\beta) = \inf \Cost_p[d](\sigma),
\end{equation*}
where the infimum is taken over all matchings of $\alpha$ and $\beta$. We will usually omit $d$ from this notation.
If $W_p(\alpha,\beta) = \Cost_p(\sigma)$ for some $\sigma$ then call $\sigma$ an \emph{optimal matching}.
\end{definition}

Note that each $p$-Wasserstein distance satisfies a universality property \cite{bubenik2019universality}. 

  Let $1 \leq q \leq p \leq \infty$.
  Since for any sequence $a = (a_i)$, $\norm{a}_p \leq \norm{a}_q$, we have the following.

\begin{lemma}
 Let $1 \leq q \leq p \leq \infty$. 
 For any matching $\sigma$ of $\alpha,\beta \in \overline{D}(X,A)$, $\Cost_p[d](\sigma) \leq \Cost_q[d](\sigma)$, and hence $W_p[d](\alpha,\beta) \leq W_q[d](\alpha,\beta)$.
  In particular, for any $1 \leq p \leq \infty$ and any matching $\sigma$ of $\alpha$ and $\beta$ with $\hat{\sigma} = \sum_{i \in I}(x_i,y_i)$, for all $i \in I$, $d(x_i,y_i) \leq \Cost_p[d](\sigma)$.
  
\end{lemma}

From Definition~\ref{def:wasserstein}, we have the following optimal matching for finite persistence diagrams on pointed metric spaces.

\begin{lemma} \label{lem:optimal-matching-finite-pointed-metric-space}
 Let $(X,d,x_0)$ be a pointed metric space.
  Let $\alpha,\beta \in D(X,x_0)$, where $\hat{\alpha} = \sum_{i=1}^m x_i$ and $\hat{\beta} = \sum_{j=1}^n y_j$. Then
\begin{equation*}
  W_p(\alpha,\beta) = \min_{\sigma \in S_{m+n}} \norm{ (d(x_i,y_{\sigma(i)}))_{i=1}^{m+n}}_p,
\end{equation*}
where $S_{m+n}$ denotes the permutation group on $m+n$ letters and $x_{m+1} = \cdots = x_{m+n} = x_0 = y_{n+1} = \cdots = y_{m+n}$.
\end{lemma}

\begin{proof}

  Let $x_{m+1} = \cdots = x_{m+n} = x_0 = y_{n+1} = \cdots = y_{m+n}$.
  By Lemma~\ref{lem:matching}, any matching $\sigma$ of $\alpha$ and $\beta$ has the form $\sigma = \sum_{i=1}^{m+n} (x_i,y_{\tau(i)})$ for some $\tau \in S_{m+n}$.
\end{proof}
        
\begin{example}
  Let $1 \leq p \leq \infty$. Consider the metric pair $(\R^2_{\leq}, l_q, \Delta)$ where $l_q(x,y)=\|x-y\|_q$ and $1 \leq q \leq \infty$. Then $(D(\R^2_{\leq}, \Delta), W_p)$
is the space of finite classical persistence diagrams consisting of only finite dots together with the $p$-Wasserstein distance with underlying distance on $\R^2$ given by the $q$-norm~\cite{csehm:lipschitz}.
Similarly, $(D(\overline{\R}^2_{\leq}, \overline{\Delta}), W_p)$
is the space of finite classical persistence diagrams together with the $p$-Wasserstein distance with underlying distance on $\R^2$ given by the $q$-norm.
Let $\Int(\R)$ denote the set of intervals in $\R$.
For the metric pair $(\Int(\R),d,\emptyset)$, where $d$ is either the Hausdorff distance or the length of the symmetric difference, we have the metric space $(D(\Int(\R),\emptyset),W_p)$
of finite barcodes~\cite{Collins:2004}.
We also have $(\overline{D}(\R^2_{\leq},\Delta),W_p)$, $(\overline{D}(\overline{\R}^2_{\leq},\overline{\Delta}),W_p)$, and $(\overline{D}(\Int(\R),\emptyset),W_p)$.
\end{example}

\section{Basic properties of spaces of persistence diagrams}

In this section we define spaces of countable persistence diagrams with certain finiteness properties and then discuss the elementary metric and topological properties of our spaces of persistence diagrams.

\subsection{Persistence diagrams with essentially finite persistence}

We start by defining spaces of countable persistence diagrams that satisfy certain finite conditions.

For $x \in X$, let $d(x,A) = \inf_{a \in A} d(x,a)$.

\begin{definition}
	Let $\delta \in (0,\infty]$. The \emph{$\delta$-offset} of $A$ is $A^{\delta} = \{x \in X \ | \ d(x,A) < \delta\}$. Note that the $\delta$-offset is sometimes called the $\delta$-neighborhood.
\end{definition}

\begin{definition}
		Let $\alpha \in \overline{D}(X,A)$.
		The \emph{$\delta$-upper part} of $\alpha$ is $u_{\delta}(\alpha) = \hat{\alpha}|_{X \setminus A^{\delta}}$. The \emph{$\delta$-lower part} of $\alpha$ is $\ell_{\delta}(\alpha) = \hat{\alpha}|_{A^{\delta} \setminus A}$. Note that $\hat{\alpha}=u_{\delta}(\alpha)+\ell_{\delta}(\alpha)$.
\end{definition}

%
We define the following sets of \emph{essentially $p$-finite persistence diagrams.}

\begin{definition}\label{def:D_p}
  Let $(X, d, A)$ be a metric pair. For $p \in [1, \infty)$, define
  \begin{align*}
    \overline{D}_p(X,A)=\big\{\alpha \in \overline{D}(X,A) \ \big| \
    \abs{u_{\infty}(\alpha)} < \infty \text{ and } W_p(\ell_{\infty}(\alpha), 0) < \infty \big\}. 
  \end{align*}
  For $p=\infty$, let
  \[\overline{D}_{\infty}(X,A)=\big\{\alpha \in \overline{D}(X,A) \
    \big| \ \text{for any } \ \delta>0 \text{ we have }
    |u_{\delta}(\alpha)| < \infty \big\}.\]
\end{definition}
	
\begin{example}
  For $p \in [1,\infty)$,
  $(\overline{D}_p (\R^2_{\leq}, \Delta), W_p)$ is the space of
  classical persistence diagrams with finite $p$-persistence~\cite{mileyko2011probability} and
  $(\overline{D}_{\infty}(\Int(\R),\emptyset),W_{\infty})$ is the space of barcodes introduced by Blumberg et al~\cite{blumberg2014robust}.
  Similarly we have
  $(\overline{D}_{\infty} (\R^2_{\leq}, \Delta), W_{\infty})$,
  $(\overline{D}_p(\Int(\R),\emptyset),W_p)$ for $p \in [1,\infty)$,
  and
  $(\overline{D}_p (\overline{\R}^2_{\leq}, \overline{\Delta}), W_p)$ for $p \in [1,\infty]$.

\end{example}

\begin{lemma} \label{lem:overlineDp}
 
  Let $p \in [1,\infty)$.
  Let $\alpha \in \overline{D}(X,A)$.
  The following are equivalent.
  \begin{enumerate}
  \item \label{it:Wp1} $\alpha \in \overline{D}_p(X,A)$.
  \item \label{it:Wp2} There is an $\eps > 0$ such that $\abs{u_{\eps}(\alpha)} < \infty$ and $W_p(\ell_{\eps}(\alpha),0) < \infty$.
  \item \label{it:Wp3} For all $\eps > 0$, $\abs{u_{\eps}(\alpha)} < \infty$ and $W_p(\ell_{\eps}(\alpha),0) < \infty$.
  \end{enumerate}
  
\end{lemma}

\begin{proof}
  \eqref{it:Wp3} $\Rightarrow$ \eqref{it:Wp2} is immediate.

 \eqref{it:Wp2} $\Rightarrow$ \eqref{it:Wp1}. First, $\abs{u_{\infty}(\alpha)} \leq \abs{u_{\eps}(\alpha)} < \infty$.
  Second, $W_p(\ell_{\infty}(\alpha),0)^p = W_p(\ell_{\eps}(\alpha),0)^p + W_p(\ell_{\infty}(u_{\eps}(\alpha)),0)^p$. The first term is finite by assumption and since $\abs{u_{\eps}(\alpha)} < \infty$, so is the second term.

  \eqref{it:Wp1} $\Rightarrow$ \eqref{it:Wp3}. Let $\eps > 0$. Suppose $\abs{u_{\eps}(\alpha)} = \infty$. Since $\abs{u_{\infty}(\alpha)} < \infty$, there is an infinite number of dots in $\alpha$ with $\eps \leq d(x_i,0)< \infty$ and thus $W_p(\ell_{\infty}(\alpha),0) = \infty$. Hence $\abs{u_{\eps}(\alpha)} < \infty$.
  Finally, $W_p(\ell_{\eps}(\alpha),0) \leq W_p(\ell_{\infty}(\alpha),0) < \infty$.
\end{proof}

\subsection{Basic metric properties}
\label{sec:basic-metric}

Next we consider elementary metric properties of our spaces of persistence diagrams.

\begin{definition}[{\cite[Lemmas 3.13, 3.17]{bubenik2019universality}}]
 \label{def:dp}
Let $(X,d,A)$ be a metric pair. For $p \in [1, \infty]$, we have a metric $d_p$
    on $X$ and $X/A$
    defined by
    \[d_p(x, y)= \min \big(d(x, y), \big\|\big(d(x, A), d(y, A)\big)\big\|_p\big).
    \]
  \end{definition}

  Note that $(X/A, d_p, A)$ is a pointed metric space. 
    Furthermore, the quotient map $(X, d) \to (X/A, d_{p})$ is $1$-Lipschitz and if $1\leq q \leq p \leq \infty$ then $d_p \leq d_q$ and the identity maps $(X, d_q) \to (X, d_p)$ and $(X/A, d_q) \to (X/A, d_p)$ are $1$-Lipschitz.
    Also if $A=\{x_0\}$ then $X/A \cong X$ and $d_1=d$.
     Furthermore, since all $p$-norms on $\R^2$ are equivalent, $d_p$ and $d_q$ are Lipschitz equivalent. We have the following.
  
    \begin{lemma}
      \label{lem:equiv_metrics}
  	Let $(X,d)$ be a metric space and let $A \subset X$. Then for all $p , q \in [1, \infty]$ the quotient metrics $d_p$ and $d_{q}$ on $X/A$ are Lipschitz equivalent. It follows that for all $p \in [1, \infty]$, $d_p$ metrizes the quotient topology on $X/A$.
  \end{lemma}

  \begin{lemma}[\cite{bubenik2019universality,bubenik2020virtual}]
    \label{lem:isomorphism}
   	Let $(X,d, A)$ be a metric pair and $1 \leq q \leq p \leq \infty$. 
   	\begin{enumerate}[label=(\alph*)]
   		\item $(d_q)_p=d_p$.
   		\item $W_p[d_q]=W_p[d_p]=W_p[d]$.
   		\item We have the following isometric isomorphisms: $(\overline{D}_p(X,A), W_p) \cong (\overline{D}_p(X/A, A), W_p)$, $(D(X,A), W_p) \cong (D(X/A, A), W_p)$ and $(D^n(X,A), W_p) \cong (D^n(X/A, A), W_p)$.
   		\item \label{it:isomorphism-d} The canonical map $(X,d) \to (D(X,A), W_p)$ is $1$-Lipschitz and induces the isometric isomorphism $ (X/A, d_p) \isomto (D^1(X, A), W_p)$.
                \item  \label{it:subadditivie} $W_p[d]$ is \emph{$q$-subadditive}: $W_p[d](\alpha'+\alpha'',\beta'+\beta'') \leq \norm{(W_p[d](\alpha',\beta'),W_p[d](\alpha'',\beta''))}_q$.
   	\end{enumerate}
   \end{lemma}

   \begin{lemma} \label{lem:DpDq}
     Let $1 \leq q \leq p \leq \infty$.
  $\overline{D}_{q}(X, A)$ is a submonoid of $\overline{D}_{p}(X,A)$ and, therefore, $\bigcup_{q=1}^{\infty} \overline{D}_{q}(X,A) \subset \overline{D}_{\infty}(X,A)$. 
  \end{lemma}

  \begin{proof}
  For $1 \leq q \leq p \leq \infty$, $\|\cdot\|_p \leq \|\cdot \|_q$ and hence $W_p \leq W_q$. It follows that $\overline{D}_q(X, A)$ is a submonoid of $\overline{D}_p(X,A)$. Therefore, $\bigcup_{q=1}^{\infty} \overline{D}_q(X,A) \subset \overline{D}_{\infty}(X,A)$. 
  \end{proof}

The next example shows that the converse inclusion does not always hold.

\begin{example}
	For classical persistence diagrams, consider $\hat{\alpha}=\sum_{n=1}^{\infty} \sum_{k=1}^{2^n} (0, \frac{1}{n})$, i.e. $\hat{\alpha}$ consists of dots $(0, \frac{1}{n})$ with multiplicity $2^n$. It is easy to check that $\alpha \in \overline{D}_{\infty}(X,A)$ but $\alpha \notin \overline{D}_q(X,A)$ for any $ \in [1, \infty)$.
\end{example}

    
    \begin{lemma}\label{lem:W_p_ineq}
    	Let $\alpha, \beta \in D(X, A)$, $|\beta| < |\alpha|$ and $p \in [1, \infty]$. Let $\hat{\alpha}=\sum_{i =1}^{|\alpha|} x_i$ where for every $i$ $d(x_i, A) \geq d(x_{i+1}, A)$. Then $W_p(\alpha, \beta) \geq W_p(\sum_{i=|\beta|+1}^{|\alpha|} x_i, 0)$.
    \end{lemma}
    
    \begin{proof}
    For each matching $\sigma$ for $\alpha$ and $\beta$, we can write $\sigma=\sigma'+\sigma''$, where $\sigma'$ matches $\alpha'$ and $\beta$, $\sigma''$ matches $\alpha''$ and $0$, $\alpha=\alpha'+\alpha''$, and  $|\sigma'|=|\beta|$. Therefore, 
    \[\Cost_p (\sigma) \geq \Cost_p (\sigma'')=W_p(\alpha'', 0) \geq  W_p(\sum_{i=|\beta|+1}^{|\alpha|} x_i, 0). \qedhere\]
    \end{proof}

    \begin{lemma}\label{lem:u_bounded}
    	Let $ \alpha \in \overline{D}_{\infty}(X, A)$ and $p \in [1, \infty]$. Then $|u_{\delta}(\alpha)| < \infty$ for every $\delta >0$.
    \end{lemma}

    \begin{proof}
     The result follows from Definition \ref{def:D_p} for $p=\infty$ and from \cref{lem:overlineDp} for $p \in [1, \infty)$.
    \end{proof}


   \begin{lemma}\label{lem:dist_between_lower_part_and_0_sum_case}
   	Let $p \in [1, \infty]$, $ \alpha \in \overline{D}_p(X, A)$ and $\hat{\alpha}=\sum_{i=1}^{\infty} x_i $. Then for every $\eps >0$ there is $n \in \N$ such that $W_p(\sum_{i= n}^{\infty} x_i, 0) < \eps$.
   \end{lemma}

  \begin{proof}
  	For $p=\infty$, the proof follows from  Lemma \ref{lem:u_bounded}. Assume $p \in [1, \infty)$. Since $|u_{\infty}(\alpha)| < \infty$ and $W_p(\ell_{\infty}(\alpha), 0) < \infty$ it follows that there is an $m \in \N$ such that $W_p(\sum_{i=m}^{\infty} x_i, 0)^p =
  	\sum_{i=m}^{\infty} d(x_i, A)^p < \infty$. Then for any $\eps > 0$ there is an $n \in \N$ such that \[W_p(\sum_{i=n}^{\infty} x_i, 0)^p=\sum_{i = n}^{\infty} d(x_i, A)^p < \eps.  \qedhere\]
  \end{proof}

\begin{corollary}\label{cor:seq_convergent}
	Let $p \in [1, \infty]$, $ \alpha \in \overline{D}_p(X, A)$ and $\hat{\alpha}=\sum_{i=1}^{\infty} x_i $. Let $\alpha_n=\sum_{i=1}^n x_i$. Then $(\alpha_n) \to \alpha$.
\end{corollary}

\begin{proof}
	Let $\eps >0$. By Lemma \ref{lem:dist_between_lower_part_and_0_sum_case}, there is an $n$ such that $W_p(\alpha, \alpha_n) \leq W_p(\sum_{i={n+1}}^{\infty} x_i, 0) < \eps.$
\end{proof}
    
    \begin{lemma}\label{lem:dist_between_lower_part_and_0}
    	Let $ \alpha \in \overline{D}_p(X, A)$ and $p \in [1, \infty]$. Then 
$\lim_{\delta \to 0} W_p\big(\ell_{\delta}(\alpha), 0\big) = 0$.
    \end{lemma}
    
    \begin{proof}
      For $p=\infty$, the result follows from
\cref{def:wasserstein}.
      Let $\hat{\alpha}=\sum_{i=1}^{\infty} x_i $, $p \in [1, \infty)$ and $\eps >0$. By Lemma \ref{lem:dist_between_lower_part_and_0_sum_case}, there is an $n$ such that $W_p(\sum_{i= n}^{\infty} x_i, 0) < \eps$. Let $\delta=\min_{1\leq i\leq n} (d(x_i, A))$. Then \[W_p(\ell_{\delta}(\alpha), 0) \leq W_p(\sum_{i= n}^{\infty} x_i, 0) < \eps. \qedhere\]
    \end{proof}

    \begin{lemma}\label{lem:lower_part_is_zero}
    	Let $\alpha \in \overline{D}_p(X,A)$. If there is $\delta >0$ such that $\ell_{\delta}(\alpha)=0$ then $\alpha \in D^n(X,A)$ for some $n \in \N$.
    \end{lemma}
    
    \begin{proof}
    	Suppose there exists $\delta >0$ such that $\ell_{\delta}(\alpha)=0$. Then $\hat{\alpha}=u_{\delta}(\alpha)$ and Lemma \ref{lem:u_bounded} concludes the proof.
    \end{proof}

Consider a pointed metric space $(X,d,x_0)$. 
Let $\sim$ be an equivalence relation on $X$.
Recall that a quotient metric on $X/{\sim}$ is defined by
\begin{equation} \label{eq:quotient-metric}
  \overline{d}([x], [x'])=\inf\big\{ d(p^1, q^1) + d(p^2, q^2)+ \dots +d(p^n, q^n)\big\},
\end{equation}
where the infimum is taken over all finite sequences $(p^1, \dots, p^n)$ and $(q^1, \dots, q^n)$ with $[x]=[p_1]$, $[q^i]=[p^{i+1}]$ for all $1 \leq i \leq n-1$ and $[x']=[q^n]$. 
Now let $n \geq 1$ and $p \in [1, \infty]$.
The product $X^n$ has the \emph{$p$-product metric} $d_p^n((x_1,\ldots,x_n),(y_1,\ldots,y_n)) = \norm{(d(x_i,y_i))_{i=1}^n}_p$ \cite{bubenik2019universality}.
Furthermore the symmetric group $S_n$ acts on $X^n$ by permuting the indices and we have the quotient metric space $(X^n/S_n,\overline{d_p^n})$.
Using the invertibility of the action of the symmetric group and the triangle inequality, the quotient metric~\eqref{eq:quotient-metric} simplifies to
  \begin{equation} \label{eq:quotient-metric-symmetric}
    \overline{d_p^n}([x],[x']) = \min_{\tau \in S_n} \norm{(d(x_i,x'_{\tau^{-1}(i)}))_{i=1}^n}_p.
  \end{equation}


\begin{lemma}\label{lem:D^n_isom_embedding}
  The map $\phi: (D^n(X,x_0),W_p) \to (X^{2n}/S_{2n},\overline{d_p^{2n}})$
  sending $\sum_{i=1}^{m} x_i$, where $0 \leq m \leq n$, to $[x_1,\ldots, x_{m}, \underbrace{x_0,\ldots,x_0}_{2n-m \text{ copies}}]$ is well defined and is an isometric embedding.
\end{lemma}

\begin{proof}
  It is easy to see that $\phi$ is a well defined injective map.
  It remains to show that it $\phi$ is an isometry.
  Let $\alpha=\sum_{i=1}^{n_1} x_i \in D^n(X, x_0)$ and $\beta=\sum_{i=1}^{n_2} y_i \in D^n(X, x_0)$.
  Let $x = (x_1,\ldots,x_{n_1},x_0,\ldots,x_0)$ and $x' = (y_1,\ldots,y_{n_2},x_0,\ldots,x_0)$.
    Comparing \eqref{eq:quotient-metric-symmetric}
    and Lemma~\ref{lem:optimal-matching-finite-pointed-metric-space}, we see that $\phi$ is an isometry.

 \end{proof}

\subsection{Basic topological properties}
\label{sec:basic-topological}

Finally we consider some elementary topological properties of our spaces of persistence diagrams.
Recall that we assume that $A$ is a closed subset of $X$.
Assume that $p \in [1,\infty]$.
    
    \begin{proposition}\label{prop:D^n_closed} $D^n(X, A)$ is closed in $(D^m(X, A), W_p)$ for integers $1 \leq n < m \leq \infty$, where $D^{\infty}(X,A) = \overline{D}(X,A)$.
    \end{proposition}
    
    \begin{proof}
    	We will show that the complement of $D^n(X,A)$ is open. Let $\alpha \in D^m(X,A) \setminus D^n(X,A)$. By Lemma \ref{lem:W_p_ineq}, $\alpha$ is in the interior of $D^m(X,A) \setminus D^n(X,A)$.
    \end{proof}
    
\begin{lemma}\label{lem:S_is_cl_and_op}
  Let $(X, d, A)$ be a metric pair and let $S$ be a subset of $X$
  containing $A$.
   If $S$ is closed in $X$ then $D(S, A)$ is closed in $(D(X, A), W_p)$.
   If $S$ is open in $X$ and for some $\delta >0$, $A^{\delta} \subset S$
   then $D(S, A)$ is open in $(D(X, A), W_p)$.
\end{lemma}
    
\begin{proof}
   Suppose $S$ is closed in $X$.
   Then $X \setminus S$ is open in $X$.
   Let $\alpha  \in D(X,A) \setminus D(S, A)$, where $\hat{\alpha}= x_1+\dots + x_n$.
   Then there exists $1 \leq i \leq n$ such that $x_i \in X \setminus S$.
   Since $X \setminus S$ is open, there is an $\eps > 0$ such that $B_{\eps}(x_i) \subset X \setminus S$.
 Suppose $\beta \in B_{\eps}(\alpha)$ and $\hat{\beta}=y_1+\dots+y_m$.
 Then $W_{\infty}(\alpha,\beta) \leq W_p(\alpha, \beta) < \eps$.
 Therefore, there exists $1 \leq j \leq m$ such that $d(x_i, y_j) < \eps$.
 Thus $y_j \in B_{\eps}(x_{i}) \subset X\setminus S$.
 Hence, $\beta \in D(X,A) \setminus D(S, A)$ and hence $B_{\eps}(\alpha) \subset D(X, A) \setminus D(S, A)$. So $\alpha$ is in the interior of $D(X,A) \setminus D(S, A)$ and thus $D(S,A)$ is closed in $D(X,A)$.
    	
 Suppose $S$ is open in $X$.
 Let $\alpha  \in D(S, A)$ and $\hat{\alpha}=x_1+\dots+x_n$.
 Since $S$ is open in $X$, for any $1 \leq i \leq n$, there is an $\eps_i >0$ such that $B_{\eps_i}(x_i) \subset S$.
 Let $\eps=\min\{\delta, \eps_1, \dots, \eps_n\}$.
 Suppose $\beta \in B_{\eps}(\alpha)$ and $\hat{\beta}=y_1+\dots+y_m$.
 Then for each $1 \leq j \leq m$, either $d(y_j, A) < \eps$ or $d(y_j, x_i) < \eps$ for some $x_i \in \alpha$.
 This implies that either $y_j \in B_{\eps}(A) \subset S$ or  $y_j \in B_{\eps}(x_i) \subset S$.
 So $y_j \in S$ for all $1 \leq j \leq m$.
 Therefore, $\beta \in D(S, A)$ and $B_{\eps}(\alpha) \subset D(S, A)$.
 Hence $D(S, A)$ is open in $D(X, A)$.
    \end{proof}

\begin{definition} \label{def:isolated}
		Let $(X, d, A)$ be a metric pair. Say that $A$ is \emph{isolated} if there is $\delta>0$ such that $A^{\delta}=A$.
\end{definition}

    \begin{proposition}\label{prop:D_equal_D_p}
      $A$ is isolated if and only if $\overline{D}_p(X,A)=D(X, A)$.
    \end{proposition}
    
    \begin{proof} 
    	Suppose $A^{\delta}=A$ for some $\delta >0$. Let $\alpha \in \overline{D}_p(X,A)$. Then by assumption $\ell_{\delta}(\alpha)=0$. By Lemma \ref{lem:lower_part_is_zero}, $\alpha \in D(X,A)$. 
    	
    	Conversely, suppose $\overline{D}_p(X,A)=D(X, A)$ and for all $\delta>0$, $A^{\delta} \neq A$. Then there is a sequence $(x_n)$ of distinct elements of $X$ such that $0< d(x_n, A) < \frac{1}{2^n}$. Let $\hat{\alpha}= \sum_{i=1}^{\infty} x_i$. Note that $\alpha \in \overline{D}_p(X,A)$. However, $\alpha \notin D(X,A)$ which gives us a contradiction. 
      \end{proof}

\begin{lemma}
  Suppose $(\alpha_n)$ is a sequence in $(\overline{D}_p(X,A),W_p)$ such that $\alpha_n \to \alpha$ for some $\alpha \in \overline{D}_p(X,A)$. Let $x \in \supp(\alpha)$.
  Then there exists a sequence $(x_n)$ in $(X,d)$ such that $x_n \in \supp(\alpha_n)$ and $x_n \to x$.
\end{lemma}

\begin{proof}
  Let $\eps > 0$. Fix $n$. Then there exists $x_n \in \supp(\alpha_n)$ such that $d(x_n,x) \leq W_p(\alpha_n,\alpha) + \eps$.
  Since $\alpha_n \to \alpha$, $x_n \to x$.
\end{proof}



A \emph{strong deformation retraction} of a topological space $X$ to a subset $A \subset X$ is a continuous map $H:X \times I \to X$ such that
  \begin{enumerate*}
  \item for all $x \in X$, $H(x,0) = x$,
  \item for all $x \in X$, $H(x,1) \in A$, and
  \item \label{it:sdr3} for all $a \in A$ and for all $t \in I$, $H(a,t) = a$.
  \end{enumerate*}

\begin{proposition} \label{prop:contractible-bottleneck}
  Let $(X, d, x_0)$ be a pointed metric space.
  If there is a strong deformation retraction of $(X,d)$ to $x_0$
  then there is
  a strong deformation retract of $(\overline{D}_{\infty}(X, x_0), W_{\infty})$ to $0$,
  which restricts to
  a strong deformation retract of $(D(X, x_0), W_{\infty})$ to $0$,
  and for all $n$,
  a strong deformation retract of $(D^n(X, x_0), W_{\infty})$ to $0$.
\end{proposition}

For the unit interval, we use the standard metric $d(s,t) = \abs{s-t}$ and for the product of two metric spaces $(X,d_X)$ and $(Y,d_Y)$ we will use the product metric $d((x,y),(x',y')) = \max(d_X(x,x'),d_Y(y,y'))$.
The proof of this proposition will use the following lemma.

\begin{lemma} \label{lem:contractible}
  Let $(X,d,x_0)$ be a pointed metric space.
    Given a continuous map $H:(X,d) \times I \to (X,d)$ such that for all $t$, $H(x_0,t)=x_0$ and $\eps > 0$, there exists a $\delta > 0$ such that whenever $x \in X$ and $s \in I$ with $d(x,x_0) < \delta$ it follows that $d(H(x,s),x_0) < \eps$.
\end{lemma}

\begin{proof}
   Consider $(x_0, t) \in X \times I$. By the continuity of $H$, there exists a $\delta_t>0$ such that for all $(x,s) \in X \times I$ with $d(x,x_0) < \delta_t$ and $\abs{s-t}< \delta_t$ it follows that $d(H(x,s),x_0) = d(H(x,s),H(x_0,t)) < \eps$.
    The open balls $B_{\delta_t}((x_0,t))$ cover $\{x_0\} \times I$, which is compact.
    So there is a finite subcover given by $t_1,\ldots,t_n$.
    Let $\delta = \min\{\delta_{t_1},\ldots,\delta_{t_n}\}$.
    The result follows.
\end{proof}

\begin{proof}[Proof of Proposition~\ref{prop:contractible-bottleneck}]
 Let $H$ be a strong deformation retraction of $(X,d)$ to $x_0$.
  Define $\overline{D}(X,x_0) \times I \to \overline{D}(X,x_0)$ by
  $\overline{H}(\alpha,t) = \sum_i H(x_i,t)$ where $\hat{\alpha} = \sum_i x_i$.
  By item \eqref{it:sdr3} in the definition of a strong deformation retraction, this map is well-defined.

Let $\alpha \in \overline{D}_{\infty}(X,x_0)$.
 By Lemma~\ref{lem:contractible}, for each $\eps > 0$, there is a $\delta > 0$ such that for all $x \in X$ and $s \in I$, if $d(x,x_0) < \delta$ then $d(H(x,s) , x_0) < \eps$.
  It follows that $\abs{u_{\eps}(\overline{H}(\alpha,s))} \leq \abs{u_{\delta}(\alpha)}$.
  Since $\alpha \in \overline{D}_{\infty}(X, x_0)$, $|u_{\delta}(\alpha)| < \infty$ and thus $|u_{\eps}(\overline{H}(\alpha, s))| < \infty$. Hence $\overline{H}(\alpha, s) \in \overline{D}_{\infty}(X, x_0)$.
  Therefore, $\overline{H}$ restricts to a map
  $\overline{H}: (\overline{D}_{\infty}(X,x_0),W_{\infty}) \times I \to (\overline{D}_{\infty}(X,x_0),W_{\infty})$.
  It remains to show that this map is continuous.

  Let $\alpha \in \overline{D}_{\infty}(X,x_0)$ and let $t \in I$.
  Let $\eps > 0$. By Lemma \ref{lem:contractible}, there is a $\delta>0$ such that for all $x \in X$ and $s \in I$, if $d(x, x_0) < \delta$ then $d(H(x, s), x_0) < \eps$.
  By definition, there exist $x_1,\ldots,x_n \in X$ such that $u_{\frac{\delta}{2}}(\alpha) = \sum_{i=1}^n x_i$.
  By the continuity of $H$, for $1 \leq i \leq n$, there exists a $\delta_i$ such that for $x \in X$ and $s \in I$ with $d(x,x_i) < \delta_i$ and $\abs{s-t} < \delta_i$, it follows that $d(H(x,s),H(x_i,t)) < \eps$.
  Let $\delta_0 = \min \{ \frac{\delta}{2}, \delta_1, \ldots, \delta_n \}$.
  Let $\beta \in \overline{D}_{\infty}(X,x_0)$ such that $W_{\infty}(\alpha,\beta) < \delta_0$, and $s \in I$ such that $\abs{s-t} < \delta_0$.
  By definition, there exist $y_1,\ldots,y_n \in X$ such that for all $1 \leq i \leq n$, $d(x_i,y_i) < \delta_i$, $\beta = \sum_{i=1}^n y_i + \beta''$ and $W_{\infty}(\beta'',0) < \delta_0 \leq \frac{\delta}{2}$.
  Therefore, $W_{\infty}(\overline{H}(\beta,s),\overline{H}(\alpha,t)) < \eps$.
  Thus 
  $\overline{H}: (\overline{D}_{\infty}(X,x_0),W_{\infty}) \times I \to (\overline{D}_{\infty}(X,x_0),W_{\infty})$ 
is continuous.

 By definition,  if $\alpha \in D^n(X,x_0)$ then $\overline{H}(\alpha,s) \in D^n(X,x_0)$.
It follows that $\overline{H}$ restricts to continuous maps
  $(D(X,x_0),W_{\infty}) \times I \to (D(X,x_0),W_{\infty})$ 
and   
$(D^n(X,x_0),W_{\infty}) \times I \to (D^{n}(X,x_0),W_{\infty})$.
\end{proof}

To extend this result to $p \in [1,\infty)$ we need an additional hypothesis.
Say that a strong deformation retract of $(X,d)$ to $x_0$ is \emph{locally Lipschitz at $x_0$} if there exist $\epsilon_0>0$ and a constant $L$ such that for all $x \in X$ with $d(x, x_0) < \eps_0$ and all $t \in I$, $d(H(x,t),x_0) < L d(x,x_0)$.

\begin{proposition} \label{prop:contractible-wasserstein}
   Let $(X, d, x_0)$ be a pointed metric space.
    Let $p \in [1,\infty)$.
    If there is a strong deformation retraction of $(X,d)$ to $x_0$
    that is locally Lipschitz at $x_0$
  then there is
  a strong deformation retract of $(\overline{D}_{p}(X, x_0), W_{p})$ to $0$,
  which restricts to
  a strong deformation retract of $(D(X, x_0), W_{p})$ to $0$,
  and for all $n$,
  a strong deformation retract of $(D^n(X, x_0), W_{p})$ to $0$.
\end{proposition}

\begin{proof}
  Let $H$ be a strong deformation retraction of $(X,d)$ to $x_0$ that is locally Lipschitz at $x_0$.
  Thus there exist $\epsilon_0>0$ and a constant $L$ such that for all $x \in X$  with $d(x, x_0) < \eps_0$  and all $t \in I$, $d(H(x,t),x_0) < L d(x,x_0)  < L \eps_0$.  
  Define $\overline{D}(X,x_0) \times I \to \overline{D}(X,x_0)$ by
  $\overline{H}(\alpha,t) = \sum_i H(x_i,t)$ where $\hat{\alpha} = \sum_i x_i$. Observe that if $\alpha \in \overline{D}_p(X, x_0)$ and $\hat{\alpha}=\sum_i x_i$ where for all $i$, $d(x_i, x_0) < \eps_0$ then for all $t \in I$ 
  	\[W_p(\overline{H}(\alpha, t), 0)=\norm{(d(H(x_i, t), x_0))_i}_p <L \norm{(d(x_i, x_0))_i}_p=L W_p(\alpha, 0).\]

Let $\alpha \in \overline{D}_p(X,x_0)$ and $t \in I$.
By Lemma~\ref{lem:overlineDp}, $\abs{u_{\eps_0}(\alpha)} < \infty$ and $W_p(\ell_{\eps_0}(\alpha),0) < \infty$.
By our assumption on $\eps_0$,
$\abs{u_{L\eps_0}(\overline{H}(\alpha,t))} < \abs{u_{\eps_0}(\alpha)} < \infty$
and
$W_p(\ell_{L\eps_0}(\overline{H}(\alpha,t)),0) \leq L W_p(\ell_{\eps_0}(\alpha),0) < \infty$.
Therefore, by Lemma~\ref{lem:overlineDp}, $\overline{H}(\alpha,t) \in \overline{D}_p(X,x_0)$.

Let $\alpha \in \overline{D}_p(X,x_0)$ and $t \in I$.
Let $\eps > 0$. Let $\delta_0 = \min(\frac{\eps}{6L},\frac{\eps_0}{2})$.
Then for all $x \in X$ and for all $s \in I$, whenever $d(x,x_0) < 2\delta_0$, it follows that
$d(H(x,s),x_0) < L d(X,x_0) < \eps$.
By Lemma~\ref{lem:dist_between_lower_part_and_0_sum_case}, $\alpha = \alpha' + \alpha''$, where $\alpha' = \sum_{i=1}^n x_i$ and $W_p(\alpha'',0) < \delta_0  \leq  \frac{\eps}{6L}$.
By the continuity of $H$, for $1 \leq i \leq n$ there exists $\delta_i$ such that for all $x \in X$ and $s \in I$, whenever $d(x,x_i) < \delta_i$ and $\abs{s-t} < \delta_i$ it follows that $d(H(x,s),H(x_i,t)) < \frac{\eps}{2n}$.
Let $\delta = \min\{\delta_0,\delta_1,\ldots,\delta_n\}$.

Let $\beta \in \overline{D}_p(X,x_0)$ and let $s \in I$ such that $W_p(\alpha,\beta) < \delta$ and $\abs{s-t} < \delta$.
Then there exists a matching $\sigma$ of $\alpha$ and $\beta$ such that $\sigma = \sigma' + \sigma''$ where $\sigma' = \sum_{i=1}^n (x_i,y_i)$, $\beta = \beta' + \beta''$ where $\beta' = \sum_{i=1}^n y_i$, $\sigma''$ is a matching of $\alpha''$ and $\beta''$, and $\Cost_p(\sigma) = \norm{(\Cost_p(\sigma'),\Cost_p(\sigma''))}_p < \delta$.
Then $d(x_i, y_i) < \delta_i$ and \[W_p(\beta'',0) \leq W_p(\beta'',\alpha'') + W_p(\alpha'',0) < \delta + \delta_0 \leq 2\delta_0 \leq \textstyle\frac{\eps}{3L}.\]
Furthermore, by Lemma \ref{lem:isomorphism}\ref{it:subadditivie},
\begin{align*}
W_p(\overline{H}(\alpha,t), \overline{H}(\beta,s)) &\leq
W_p( \overline{H}(\alpha',t),\overline{H}(\beta',s) )+ 
                                                     W_p( \overline{H}(\alpha'',t),\overline{H}(\beta'',s))\\
                                                   & \leq \sum_{i=1}^n d((H(\alpha_i, t), H(\beta_i,s))) + W_p(\overline{H}(\alpha'',t),0) + W_p(\overline{H}(\beta'',s),0)\\
                                                 &< n(\textstyle\frac{\eps}{2n}) + L W_p(\alpha'',0) + L W_p(\beta'',0)\\
  &< \textstyle\frac{\eps}{2} + \frac{\eps}{6}+\frac{\eps}{3} = \eps.
\end{align*}
  Thus 
  $\overline{H}: (\overline{D}_{p}(X,x_0),W_{p}) \times I \to (\overline{D}_{p}(X,x_0),W_{p})$ 
is continuous.
\end{proof}

\section{Optimal matchings}

In this section, we give sufficient conditions for the existence of optimal matchings between persistence diagrams. Let $p \in [1,\infty]$.
To start, recall from \cref{lem:optimal-matching-finite-pointed-metric-space} that finite persistence diagrams on a pointed metric space have an optimal matching.
The following example shows that this result does not hold in general for metric pairs.

\begin{example}
 Consider the metric space given by the following wedge sum of pointed metric spaces where we identify the base points. For $k \in \N$, let $X_k$ be the closed interval $[0,1 + \frac{1}{k}]$ with the base point $x_k=0$ and $a_k = 1 + \frac{1}{k}$. Let $X = \bigvee_{k=1}^{\infty} (X_k,x_k)$.
  Let $x$ be the point of $X$ obtained by identifying the $x_k$.
  Let $A = \{a_1,a_2,a_3,\ldots\}$.
  Note that $A$ is a closed subset of $X$.
  Furthermore $d(x,A) = 1$ but $d(x,a) > 1$ for all $a \in A$.
  Consider $\alpha,\beta \in D(X,A)$ with
 $\hat{\alpha} = x$ and $\hat{\beta} = 0$. 
 Then for any $p \in [1,\infty]$, $W_p(\alpha,\beta) = d(x,A) = 1$.
 A matching $\sigma$ of $\alpha$ and $\beta$ has $\hat{\sigma} = (x,a_k)$ for some $k$ and thus $\Cost_p(\sigma) = 1 + \frac{1}{k} > 1$.
  Therefore, there does not exist an optimal matching of $\alpha$ and $\beta$ in $(D(X,A),W_p)$.
\end{example}

In order to generalize \cref{lem:optimal-matching-finite-pointed-metric-space} to optimal matchings for finite persistence diagrams on metric pairs, we will need the following definition.

\begin{definition} \label{def:distance-minimizing}
For a metric space $(X,d)$, say that a subset $A \subset X$ is \emph{distance minimizing} if for each $x \in X$ with $d(x,A) < \infty$ there exists an $a \in A$ such that $d(x,A) = d(x,a)$.
\end{definition}

	Recall that $d(x, A)= \inf_{a \in A} d(x, a)$. 
	
\begin{lemma}\label{lem:A_compact}
  If $(X, d,A)$ is a metric pair where $A$ is compact then
$A$ is distance minimizing.
\end{lemma}

\begin{proof}
  Let $x \in X$ with $d(x, A) < \infty$.
  It is easy to check that the map $d(x, -): X \to \R_+$ is continuous.
  By definition, there is a sequence $(a_n) \in A$ such that
  \begin{equation} \label{eq:d(x,A)} 
  d(x,A) \leq d(x, a_n) \leq d(x, A)+ \frac{1}{n}.
  \end{equation}
Since $A$ is compact then by Theorem \ref{thm:equiv_types_compts}, the sequence $(a_n)$ has a convergent subsequence $(a_{n_k}) \to a \in A$.
  By the continuity of $d(x,-)$, $d(x,a) = d(x,A)$.
\end{proof}

A metric space $(X, d)$ is \emph{proper} if every closed ball in $X$ is compact.

\begin{lemma}\label{lem:A_dist_minimizing}
	If $(X, d,A)$ is a metric pair where $(X, d)$ is proper then
	$A$ is distance minimizing.
\end{lemma}

\begin{proof}
  Let $x \in X$ with $d(x, A) < \infty$.
  Let $r = d(x,A)$.
  For every $n \in \N$, let
  $Y_n = \overline{B}_{r + \frac{1}{n}}(x)$.
  By assumption, $Y$ is compact. Since
$A$ is closed, 
$Y_n \cap A$ is also compact.
By definition, for all $n$, $Y_n \cap A \neq \emptyset$ and for $n \leq m$, $Y_m \subset Y_n$.
Therefore, there exists $a \in \cap_{n=1}^{\infty} (Y_n \cap A)$.
It follows that
$d(x, A)=d(x, a)$.
\end{proof}

In the  next example we have a metric pair $(X, d, A)$ where $(X, d)$ is not proper and $A$ is distance minimizing.

\begin{example}\label{ex:non-proper_and_A_minimizing}		
  Let $X=\R \setminus \{0\}$ with the Euclidean distance $d$. Let $A=\overline{B}_2(1)=[-1, 0)\cup (0,3]$. Note that $A$ is closed in $X$ and is distance minimizing. However, the closed ball $[-1, 0)\cup (0,3]$ is not compact in $X$. Indeed, its open cover $\cup_{n=1}^{\infty} [-1, -\frac{1}{n}) \cup (\frac{1}{n}, 3]$ does not have a finite subcover. Therefore, $(X,d)$ is not proper.
\end{example}

We have the following generalization of \cref{lem:optimal-matching-finite-pointed-metric-space}.

\begin{lemma} \label{lem:optimal-matching-finite}
 Let $(X,d,A)$ be a metric pair such that $A$ is distance minimizing.
 Let $\alpha,\beta \in D(X,A)$, where $\hat{\alpha} = \sum_{i=1}^m x_i$ and $\hat{\beta} = \sum_{j=1}^n y_j$.
For $1 \leq j \leq n$, let $x_{m+j} = a_j$, where $a_j \in A$ such that $d(y_j,A) = d(y_j,a_j)$.
For $1 \leq i \leq m$, let $y_{n+i} = b_i$, where $b_i \in A$ such that $d(x_i,A) = d(x_i,b_i)$. 
 Then
\begin{equation*}
  W_p(\alpha,\beta) = \min_{\sigma \in S_{m+n}} \norm{ (d(x_i,y_{\sigma(i)}))_{i=1}^{m+n}}_p,
\end{equation*}
where $S_{m+n}$ denotes the permutation group on $m+n$ letters.
\end{lemma}

\begin{proof}
    By Lemma~\ref{lem:matching}, a matching $\gamma$ of $\alpha$ and $\beta$ has the form $\sum_{i=1}^{m+n} (x_i,y_{\sigma(i)})$, where
    $\sigma \in S_{m+n}$ and
for $1 \leq j \leq n$, $x_{m+j} \in A$ and
for $1 \leq i \leq m$, $y_{n+i} \in A$.
  Since $A$ is distance minimizing, in Definition~\ref{def:wasserstein} we can restrict to matchings as in the statement of the lemma.
\end{proof}

\begin{lemma} \label{lem:sequence-of-matched-dots}
 Let $(X,d,A)$ be a metric pair.
  Given $\alpha \in \overline{D}_p(X,A)$, let $S = \supp(\hat{\alpha}) \cup A$.
  Then any sequence $(y_k)$ in $S$ has a subsequence $(y_{n_k})$ that is either constant or for which the sequence $(d(y_{n_k},A))$ converges to $0$.
\end{lemma}

\begin{proof}
Let $(y_n)$ be a sequence in $S$. Suppose that $(y_n)$ does not have a constant subsequence. Then it has a subsequence $(y_{n_k})$ whose elements are all distinct. By Lemma \ref{lem:dist_between_lower_part_and_0_sum_case}, the sequence $(d(y_{n_k},A))$ converges to $0$.
\end{proof}

\begin{theorem} \label{thm:optimal-matching}

  Let $(X, d, A)$ be a metric pair.  
 There exists an optimal matching $\sigma \in \overline{D}(X\times X, A\times A)$ between every $\alpha, \beta \in \overline{D}_p(X, A)$ with $W_p(\alpha,\beta) < \infty$ if and only if $A$ is distance minimizing.
\end{theorem}

\begin{proof}
Suppose there exists an optimal matching $\sigma \in \overline{D}(X\times X, A\times A)$ between every $\alpha, \beta \in \overline{D}_p(X, A)$ with $W_p(\alpha,\beta) < \infty$. Let $x \in X$ with $d(x,A)< \infty$. Set $\alpha=x $ and $\beta=0$. Note that $W_p(\alpha, \beta)=d(x,A)$. Let $\sigma$ be an optimal matching of $\alpha$ and $\beta$. Thus $\hat{\sigma} =  (x, a)$ for some $a \in A$. Therefore, $\Cost_p(\sigma)=d(x,a)$. Hence $d(x,A)=d(x,a)$.

Conversely, suppose $A$ is distance minimizing.
  Let $\alpha,\beta \in \overline{D}_p(X,A)$ with $W_p(\alpha,\beta) = T < \infty$.
  If $\alpha,\beta \in D(X,A)$ then apply Lemma~\ref{lem:optimal-matching-finite} and we are done.
If not, then write $\hat{\alpha} = \sum_{i=1}^{\infty} x_i$, and $\hat{\beta} = \sum_{j \in J} y_j$.
Let $(\sigma_{\ell})$ be a sequence of matchings of $\alpha$ and $\beta$ such that $\lim_{\ell \to \infty} \Cost_p(\sigma_{\ell})=T$.
Without loss of generality, assume that for each $\ell$ and each $(a,b) \in \supp(\sigma_{\ell})$, $d(a,b) \leq d(a,A) + d(A,b)$.
If not, then replace $(a,b)$ with $(a,z) + (w,b)$, where $d(a,A) = d(a,z)$ and $d(A,b) = d(w,b)$, to obtain a matching of $\alpha$ and $\beta$ with lower cost.

Here is the strategy for the proof.
Restrict $(\sigma_{\ell})$ to a subsequence of matchings in which $x_1$ is matched to some $z_1$.
Further restrict to a subsequence of matchings in which $x_2$ is matched to some $z_2$, and so on. 
Match each $x_i$ to the corresponding $z_i$ and match the remaining unmatched $y_j$ to their nearest points in $A$.
Unfortunately, this doesn't quite work.
There may be $x_i$ for which there does not exists a subsequence of matchings in which $x_i$ is matched to a single $z_i$.
However, in this case there is a subsequence of matchings in which $x_i$ is  matched to a sequence of terms that converge to $A$.
In this case, match $x_i$ to the nearest point in $A$.
Now let us return to the proof.

  Let $L_0 = \N$.
  By Lemma~\ref{lem:matching}, for each $\ell \in L_0$ there exists $K_{\ell} \subset \N$ and injection $\varphi_{\ell}: K_{\ell} \to J$ such that
  \[\sigma_{\ell} = \sum_{i=1}^{\infty}(x_i,u_i^{(\ell)}) + \sum_{j \in J \setminus \varphi_{\ell}(K_{\ell})} (w_j^{(\ell)},y_j),\]
  where $u_i^{(\ell)} \in \supp(\hat{\beta}) \cup A$ and $w_j^{(\ell)} \in A$.
Consider the sequence $(u_1^{(\ell)})_{\ell \in L_0} \in \supp(\hat{\beta}) \cup A$.
By Lemma~\ref{lem:sequence-of-matched-dots}, we may restrict to a subsequence $L_1 \subset L_0$ such that $(u_1^{(\ell)})_{\ell \in L_1}$ is constant or $d(u_1^{(\ell)},A)_{\ell \in L_1} \to 0$ as $\ell \to \infty$.
Consider the sequence $(u_2^{(\ell)})_{\ell \in L_1} \in \supp(\hat{\beta}) \cup A$.
By Lemma~\ref{lem:sequence-of-matched-dots}, we may restrict to a subsequence $L_2 \subset L_1$ such that $(u_2^{(\ell)})_{\ell \in L_2}$ is constant or $d(u_2^{(\ell)},A)_{\ell \in L_2} \to 0$ as $\ell \to \infty$.
By induction, for each $i \in \N$, we may restrict to a subsequence $L_i \subset L_{i-1}$ such that $(u_i^{(\ell)})_{\ell \in L_i}$ is constant or $d(u_i^{(\ell)},A)_{\ell \in L_i} \to 0$ as $\ell \to \infty$.

  For each $i \in \N$,
  if the sequence $(u_i^{(\ell)})_{\ell \in L_i}$ is constant, then let $z_i$ denote this constant term.
  If not, then choose $z_i \in A$ such that $d(x_i,A) = d(x_i,z_i)$.
  In the first case,
  \begin{equation}
    \label{eq:equals}
    \text{for each } \ell \in L_i, \quad d(x_i,u_i^{(\ell)}) = d(x_i,z_i).
  \end{equation}
  In the second case, by the triangle inequality,
  \begin{equation} \label{eq:converging}
    \text{as } \ell \to \infty, \quad
(d(x_i,u_i^{(\ell)}))_{\ell \in L_i} \to d(x_i,A) = d(x_i,z_i).
  \end{equation}

  Let $K = \{i \in \N \ | \ z_i \in \supp(\hat{\beta})\}$.
  Then for $i \in \N \setminus K$, $z_i \in A$.
  By construction, $\sum_{k \in K} z_k \leq \hat{\beta}$ as elements of $\Z_+^{X \setminus A}$.
  Therefore, there exists an injection $\varphi:K \to J$
  such that for all $k \in K$, $z_k = y_{\varphi(k)}$.
  For $j \in J \setminus \varphi(K)$, choose $w_j \in A$ such that $d(y_j,A) = d(y_j,w_j)$.
  Let
  \[\tau = \sum_{i=1}^{\infty} (x_i,z_i) + 
\sum_{j \in J \setminus \varphi(K)} (w_j,y_j) =
    \sum_{k \in K}(x_k,y_{\varphi(k)}) +
  \sum_{i \in I \setminus K} (x_i,z_i) +
  \sum_{j \in J \setminus \varphi(K)} (w_j,y_j).\]
  Then by Lemma~\ref{lem:matching}, $\tau$ is a matching of $\alpha$ and $\beta$.


We will now show that $\tau$ is an optimal matching.
Let $\eps>0$.
Write 
$\tau=\sum_{j=1}^{\infty} (a_j, b_j)$.
By Lemma \ref{lem:dist_between_lower_part_and_0_sum_case}, there are $m, n \in \N$ such that $W_p(\sum_{j=m}^{\infty} a_j, 0) < \frac{\eps}{4}$ and $W_p(\sum_{j=n}^{\infty} b_j, 0) < \frac{\eps}{4}$. Let $N=\max(m, n)$. Choose $M \in \N$ such that $\sum_{j=1}^{N-1} a_j \leq \sum_{i=1}^M x_i$.
Then $\sum_{j=1}^{N-1} (a_j, b_j) \leq \sum_{i=1}^M (x_i, z_i)$.
Now consider the sequence of matchings $(\sigma_{\ell})_{\ell \in L_M}$. %
Choose $\ell_0 \in L_M$ such that $\Cost_p(\sigma_{\ell}) < T+\frac{\eps}{4}$ for all $\ell \in L_M$ such that $\ell \geq \ell_0$.

For $1 \leq i \leq M$, by \eqref{eq:equals} and \eqref{eq:converging}, we may choose $\ell_i \in L_i$ such that $d(x_i, z_i) < d(x_i, u_i^{(\ell)}) + \frac{\eps}{4M}$ for all $\ell \in L_i$ such that $\ell \geq \ell_i$.
Let $\tilde{\ell} \in L_M$ such that $\tilde{\ell} \geq \max_{0 \leq i \leq M} \ell_i$.
Then
\begin{align*}
  \Cost_p(\tau) &= \norm{(d(a_j, b_j))_{j=1}^{\infty}}_p 
                          \leq
 \norm{(d(a_j, b_j))_{j=1}^{N-1}}_p +  \norm{(d(a_j, b_j))_{j=N}^{\infty}}_p\\
&\leq          
  \norm{(d(x_i, z_i))_{i=1}^{M}}_p +
  \norm{(d(a_j, A))_{j=N}^{\infty} +
  (d(A, b_j))_{j=N}^{\infty}}_p\\
&\leq          
  \norm{(d(x_i, z_i))_{i=1}^{M}}_p +
  \norm{(d(a_j,A))_{j=N}^{\infty}}_p +
  \norm{(d(A, b_j))_{j=N}^{\infty}}_p\\
&<
 \norm{(d(x_i, u_i^{(\tilde{\ell})}+\textstyle\frac{\eps}{4M})_{i=1}^{M}}_p +
 \norm{(d(a_j,A))_{j=m}^{\infty}}_p +
 \norm{(d(A, b_j))_{j=n}^{\infty}}_p\\
&<
 \norm{(d(x_i,u_i^{(\tilde{\ell})})_{i=1}^M}_p + \norm{(\textstyle\frac{\eps}{4M})_{i=1}^{M}}_p + \textstyle\frac{\eps}{4} + \textstyle\frac{\eps}{4}\\
&\leq          
\Cost_p (\sigma_{\tilde{\ell}}) + \norm{(\textstyle\frac{\eps}{4M})_{i=1}^{M}}_1 + \textstyle\frac{\eps}{2}
< T + \textstyle\frac{\eps}{4} + \textstyle\frac{\eps}{4} + \textstyle\frac{\eps}{2} = T + \eps          
\end{align*}
Therefore, $\Cost_p(\tau)=T$ and hence $\tau$ is an optimal matching of $\alpha$ and $\beta$.
\end{proof}

\section{Paths in spaces of persistence diagrams}

In this section we study paths in spaces of persistence diagrams.
Throughout this section, we will assume that $p \in [1,\infty]$.

\subsection{Connectedness}\label{sec:connected}

We start with an easy observation on connectedness.

	
\begin{proposition}\label{prop:connectness}
  Let $(X,d, A)$ be a metric pair. If $(X, d)$ is disconnected and there is $\delta >0$ such that $A^{\delta}$ is contained in one of the connected components of $X$ then $(D(X,A), W_p)$ is disconnected.
\end{proposition}
	
\begin{proof}
Suppose there exist nonempty open subsets $S, T \subset X$ such that $X=S \cup T$ and $S \cap T = \emptyset$. Furthermore, suppose there is $\delta >0$ such that $B_{\delta}(A) \subset S$. By Lemma \ref{lem:S_is_cl_and_op}, $D(S, A)$ is closed and open in $D(X, A)$. Also note that $D(S, A) \neq D(X, A)$ and $D(S, A) \neq \emptyset$.  Thus $D(X, A)$ is not connected.
\end{proof}

\subsection{Path connectedness}
\label{subsection:connectness}

Next we consider path connectedness.
For a topological space $X$ and $x, y \in X$, a \emph{path} from $x$ to $y$ is a continuous map
$\gamma: [0,1] \to X$ where $\gamma(0)=x$ and $\gamma(1)=y$.
Let $(X,d)$ be a metric space.
Then every path in $X$ is uniformly continuous.

Let $(X,d, A)$ be a metric pair. Suppose $(X,d)$ is path connected. Since the image of a path-connected space under a continuous map is path connected, $(X/A, d_{p}) $ is also path connected.


	

\begin{theorem} \label{thm:path-connected}
Let $(X,d,A)$ be a metric pair.
  Assume that $(X,d)$ is path connected. Let $n \in \N$.
  Then $(D(X,A),W_p)$ and $(D^n(X,A),W_p)$, are path connected.
\end{theorem}

\begin{proof}
  Suppose $(X, d)$ is path connected.
  Let $x_0 \in A$.
  Let $\alpha \in D(X, A)$, where $\hat{\alpha} = \sum_{i=1}^n x_i$.
  Define $\gamma:[0,1] \to D(X, A)$ by $\gamma(t)=\sum_{i=1}^n \gamma_{i}(t)$, where $\gamma_{i}$ is the image of a path in $X$ from $x_i$ to $x_0$ in $D(X,A)$.
  Note that
  $\gamma$ is a path from $\alpha$ to $0$.
  Let 
  $\eps >0$.
  Since $\gamma_{i}$ is uniformly continuous there is $\delta_i >0$ such that for any $t, t' \in [0,1]$ with $|t-t'|<\delta_i$, we have $d(\gamma_{i}(t), \gamma_{i}(t')) < \frac{\eps}{N}$.
  Let $\delta=\min \{\delta_1, \dots, \delta_N\}$.
  Then for any $t, t' \in [0,1]$ such that $|t-t'|<\delta$,
\begin{equation*}
  W_p(\gamma(t), \gamma(t')) \leq \norm{(d(\gamma_{i}(t), \gamma_{i}(t'))\big)_{i=1}^N}_p < \frac{\eps}{N} \norm{\underbrace{(1, \dots, 1)}_{N\ \text{times}}}_p \leq \eps.
\end{equation*}
Therefore, $(D(X,A),W_p)$ and $(D^n(X,A),W_p)$ are path connected.
\end{proof}

 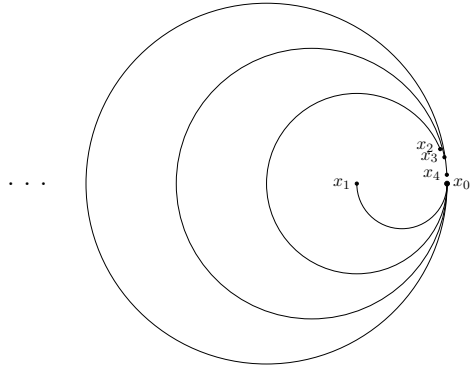
\begin{figure}[t!]
	\begin{center}
		\scalebox{0.6}{
		\begin{tikzpicture}
			\draw (1,0) coordinate (A) arc [radius=1, start angle=360, end angle=180] coordinate (x_1) ;
			\draw  (1,0) coordinate (A) arc [radius=2, start angle=360, end angle=180/8] coordinate (x_2);
			\draw  (1,0) coordinate (A) arc [radius=3, start angle=360, end angle=180/16] coordinate (x_3);
			\draw  (1,0) coordinate (A) arc [radius=4, start angle=360, end angle=180/64] coordinate (x_4);
			\filldraw (A) circle (1.5pt);
			\filldraw (x_1) circle (1pt);
			\filldraw (x_2) circle (1pt);
			\filldraw (x_3) circle (1pt);
			\filldraw (x_4) circle (1pt);
			\node at (A) [right] {$x_0$};
			\node at (x_1) [left] {$x_1$};
			\node at (x_2) [left, yshift=0.5mm] {$x_2$};
			\node at (x_3) [left] {$x_3$};
			\node at (x_4) [left] {$x_4$};
			\node at (x_1) [left, xshift=-6.5cm]  {$\textbf{\mydots}$};
		\end{tikzpicture}}
	\end{center}
	\caption{A path connected pointed metric space $(X,d,x_0)$ whose space of persistence diagrams does not have a path from $\sum_{n=1}^{\infty} x_n$ to $0$. See Example~\ref{ex:circles}.} \label{fig:circles} 
\end{figure}
   The example below shows that it need not be the case that $(\overline{D}_p(X,A),W_p)$ is path connected.
\begin{example}\label{ex:circles}
   	For $n \in \N$, consider the circle of radius $n$ centered at $0$ with the metric given by arc length. Let $C_n$ denote the subspace given by $\{ne^{i\theta} \ | \ \frac{\pi}{n^3} \leq \theta \leq 2\pi\}$. We have a pointed metric space $(C_n, c_n)$ where $c_n = ne^{2\pi i} = (n,0)$. Let $X = \bigvee_{n=1}^{\infty} C_n = \coprod_{n=1}^{\infty} C_n / (c_n \sim c_m)_{n,m}$.
   	Denote the equivalence class $[c_n]$ by $x_0$.
   	See Figure \ref{fig:circles} for an illustration.
   	We have the pointed metric space $(X,d,x_0)$, where
   	\begin{equation*}
   		d(n e^{i\theta},m e^{i\varphi}) =
   		\begin{cases} \min ( n \abs{\theta -\varphi},
   			n(2\pi - \abs{\theta -\varphi}) ), & \text{ if } n=m\\
   			\min ( n \theta, n(2\pi-\theta)) + \min(m \varphi,
   			m(2\pi-\varphi)), & \text{ if } n \neq m.
   		\end{cases}
   	\end{equation*}
   	Note that if $n \neq m$ then
   	$d(n e^{i \theta}, m e^{i \varphi}) = d(n e^{i\theta},x_0) + d(m e^{i\varphi},x_0)$.
   	
   	Observe that $X$ is path connected. Let $p \in [1,\infty]$.
   	We will show that $(\overline{D}(X,x_0),W_p)$ is not path connected.
   	Let
   	$\alpha=\sum_{n=1}^{\infty} x_n$, where
   	$x_n=n e^{\pi / n^3}$. Then
   	$W_p(\alpha, 0) =\|(d(x_n, x_0))_{n})\|_p = \|(\frac{\pi}{n^2})_{n})\|_p = \frac{\pi^3}{6}$.
   	Hence $\alpha \in \overline{D}_p(X, x_0)$.
   	Assume that there exists a path $\gamma:[0,1] \to (\overline{D}(X,x_0),W_p)$ from $\alpha$ to $0$.
   	Since $[0,1]$ is compact, so is the image of $\gamma$. Thus the diameter of the image of $\gamma$ is finite. Choose $N$ to be greater that this diameter.
   	Consider the subspace of $C_N$ given by $C'_N = \{ N e^{i\theta} \ | \ \frac{\pi}{N^3} \leq \theta \leq \frac{\pi}{N^3} + 1\}$.
   	Since $N > \diam \supp (\im \gamma)$,  if $x \in C'_N$ and $y$ is in the support of the image of $\gamma$ and $d(x,y) < d(x_N,x_0)$ then $y \in C'_N$.
   	
   	Let $S = \{t \in [0,1] \ | \ \supp (\gamma(t)) \cap C'_N \neq \emptyset\}$.
   	By assumption $0 \in S$. Thus $S \neq \emptyset$.
   	We will show that $S$ is open and closed and thus equal to $[0,1]$, which contradicts that $\gamma(1) = 0$.
   	First we show that $S$ is open.
   	Assume $t_0 \in S$.
   	Then by definition there exists $x \in \supp (\gamma(t_0)) \cap C'_N$.
   	Since $\gamma$ is continuous, there exists $\delta > 0$ such that for all $t \in [0,1]$ with $0 < \abs{t-t_0} < \delta$, $W_p(\gamma(t_0),\gamma(t)) < d(x_N,x_0)$.
   	Consider such a $t$.
   	Then by the definition of $W_p$ there exists $y \in \supp (\gamma(t))$ with $d(x,y) < d(x_N,x_0)$.
   	As observed above, this implies that $y \in C'_N$.
   	Thus $t \in S$ and hence $S$ is open.
   	Next we show that $S$ is closed.
   	Assume $t \in \overline{S}$.
   	Then there exists a sequence $(t_k)$ in $S$ converging to $t$.
   	Since $\gamma$ is continuous, $\gamma(t_k)$ converges to $\gamma(t)$.
   	Choose $K$ such that $W_p(\gamma(t_K),\gamma(t)) < d(x_N,x_0)$.
   	By definition, there exists $x \in \supp (\gamma(t_K)) \cap C'_N$.
   	By the definition of $W_p$, there exists $y \in \supp (\gamma(t))$ such that $d(x,y) < d(x_N,x_0)$.
   	Again, by the observation above, $y \in C'_N$.
   	Therefore, $t \in S$ and hence $S$ is closed.
   \end{example}

\subsection{Rectifiable paths}
\label{sec:rectifiable-paths}

As preparation for the remainder of the section, we study paths with finite length.

The \emph{length} of a path $\gamma$ in $(X,d)$ is defined by
\[\Len_{d}(\gamma)=\sup \sum_{n=1}^{N} d(\gamma(t_{n-1}),
  \gamma(t_n))\] where the supremum is taken over all $N \in \N$ and
all partitions $0=t_0 < t_1 < \dots < t_N=1$ of $[0,1]$. The path
$\gamma$ is \emph{rectifiable} if $\Len_{d}(\gamma) < \infty$.
For a path $\gamma$ and $0 \leq a \leq b \leq 1$, we have the \emph{sub-path} $\gamma^{a,b}$, given by $\gamma^{a,b}(t) = \gamma(a + t(b-a))$.
If the path $\gamma$ is rectifiable then so is $\gamma^{a,b}$.

A path $\gamma$ has \emph{constant speed}  if it is rectifiable and for all $0 \leq a \leq b \leq 1$, $\Len_d(\gamma^{a,b}) = (b-a) \Len_d(\gamma)$.

\begin{proposition}
	Let $(X, d)$ be a metric space. Then every rectifiable path $\gamma:[0,1] \to X$ can be represented in the form $\gamma=\overline{\gamma} \circ \phi$ where $\overline{\gamma}: [0, 1] \to X$ is a path with  constant speed and $\phi:[0,1] \to [0, 1]$ is a continuous nondecreasing surjective map.
\end{proposition}
 
\begin{proof}
  If $\Len_d(\gamma) = 0$ then we may take $\phi$ to be identity and $\bar{\gamma} = \gamma$. Assume $\Len_d(\gamma) > 0$.
     Following the proof of \cite[Proposition 2.5.9]{burago2001course}, take $\phi(t)=\frac{\Len_d(\gamma^{0,t})}{\Len_d(\gamma)}$
     and $\overline{\gamma}(\tau)=\gamma(t)$ where $\tau= \phi(t)$.
\end{proof}



\begin{proposition} \label{prop:path}
Let $(X,d,A)$ be a metric pair and let $\alpha,\beta \in \overline{D}_p(X,A)$ with $W_p(\alpha,\beta) < \infty$.
  Assume that $\alpha = \sum_{i \in I} x_i$, $\beta = \sum_{i \in I} y_i$ and that for each $i \in I$,  $\gamma_i$ is a path in $X$ from $x_i$ to $y_i$\ with a constant speed and   $\norm{(\Len_d(\gamma_i))_{i \in I}}_p < \infty$.
  Let $\gamma: [0,1] \to \overline{D}(X,A)$ be given by $\gamma(t) = \sum_{i \in I} \gamma_i(t)$. Then
  \begin{enumerate}[label=(\alph*)]
  \item \label{it:path-1}
    $\gamma$ is a path in $(\overline{D}_p(X,A),W_p)$ from $\alpha$ to $\beta$, and 
  \item \label{it:path-2}
    $\Len_{W_p}(\gamma) \leq \norm{(\Len_d(\gamma_i))_{i \in I}}_p$.
  \end{enumerate}
\end{proposition}

\begin{proof}
First we show that for each $t \in [0,1]$, $\gamma(t) \in \overline{D}_p(X,A)$.
  By the triangle inequality and the Minkowski inequality, we have that 
\begin{align*}
W_p(\gamma(t),0) &= \norm{(d(\gamma_i(t),0))_{i \in I}}_p
\leq \norm{(d(\gamma_i(t),\gamma_i(0)) + d(\gamma_i(0),A))_{i \in I}}_p \\
&= \norm{(d(\gamma_i(t),\gamma_i(0)))_{i \in I} + (d(\gamma_i(0),A))_{i \in I}}_p 
 \leq \norm{(d(\gamma_i(t),\gamma_i(0)))_{i \in I}}_p + W_p(\alpha,0).
  \end{align*}
  The second term is finite by assumption. The first term is bounded by $\norm{(\Len_d(\gamma_i^{0,t}))_{i \in I}}_p$, which equals $t \norm{(\Len_d(\gamma_i))_{i \in I}}_p$, which is also finite by assumption.

  Second we show that $\gamma$ is continuous.
  Assume without loss of generality that $I = \N$ or $I =\{1,\ldots,n\}$ for some $n \in \N$.
  Let $\eps > 0$.
  Choose $N$ such that $\norm{(\Len_d(\gamma_i))_{i > N}}_p <  \frac{\eps}{2}$.
  For $1 \leq i \leq N$, there exists $\delta_i$ such that for all $s,t \in [0,1]$ with $\abs{s-t}< \delta_i$, $d(\gamma_i(s),\gamma_i(t)) < \frac{\eps}{2N}$.
  Let $\delta = \min \{\delta_1,\ldots,\delta_N\}$.
  Then for all $s,t \in [0,1]$ with $\abs{s-t} < \delta$,
  \begin{multline*}
    W_p(\gamma(s),\gamma(t)) \leq \norm{(d(\gamma_i(s),\gamma_i(t)))_{i \in I}}_p
    \leq \norm{(d(\gamma_i(s),\gamma_i(t)))_{i=1}^{N}}_p + 
      \norm{(d(\gamma_i(s),\gamma_i(t)))_{i > N}}_p \\
      <  \textstyle\norm{(\frac{\eps}{2N})_{i=1}^N}_p+\norm{(\Len_d(\gamma_i))_{i>N}}_p
      \leq
      \textstyle\frac{\eps}{2}+\textstyle\frac{\eps}{2} = \eps. 
    \end{multline*}
    
 Third we show that $\Len_{W_p}(\gamma) \leq \norm{(\Len_d(\gamma_i))_{i \in I}}_p$.
  Let $0 \leq s \leq t \leq 1$.
  Then 
\[W_p(\gamma(s),\gamma(t)) \leq \norm{(d(\gamma_i(s),\gamma_i(t)))_{i \in I}}_p \leq \norm{(\Len_d^{s,t}(\gamma_i))_{i \in I}}_p = (t-s)\norm{(\Len_d(\gamma_i))_{i \in I}}_p.\]
  Therefore, for any $N \in \N$ and $0 = t_0 < t_1 < \cdots < t_N = 1$,
  $\sum_{n=1}^N W_p(\gamma(t_{n-1}),\gamma(t_n)) \leq \norm{(\Len_d(\gamma_i))_{i \in I}}_p$.
  Hence $\Len_{W_p}(\gamma) \leq \norm{(\Len_d(\gamma_i))_{i \in I}}_p$.
\end{proof}

\subsection{Length spaces}\label{subsection:length_space}

Here we consider metric spaces for which distances may be obtained from the lengths of paths.
A metric space $(X,d)$ is called a \emph{length space} or an \emph{inner metric space} if for any $x, y \in X$ with $d(x,y) < \infty$ 
\[d(x,y)= \inf \Len_{d}(\gamma)\]
where the infimum is taken over all rectifiable paths $\gamma$ from $x$ to $y$. 

\begin{theorem} \label{thm:length}
Let $(X,d,A)$ be a metric pair and assume that $(X,d)$ is a length space. Let $n \in \N$. Then
$(\overline{D}_p(X,A),W_p)$, and $(D(X,A),W_p)$ are also length spaces.
\end{theorem}

\begin{proof}
  Let $\alpha,\beta \in (\overline{D}_p(X,A),W_p)$ with $W_p(\alpha, \beta) < \infty$.
    Let $\eps > 0$.
  By \cref{def:wasserstein} there exists a matching $\sigma$ of $\alpha$ and $\beta$ with $\hat{\sigma} = \sum_{i \in I}(x_i,y_i)$ where $I = \N$ or $I = \{1,\ldots, N\}$ for some $N \in \N$ and $\abs{I} \leq \abs{\alpha + \beta}$,
  with $\Cost_p(\sigma) = \norm{(d(x_i,y_i))_{i \in I}}_p \leq W_p(\alpha,\beta) +\frac{\eps}{2}$.

  Since $(X,d)$ is a length space, for all $i \in I$ there is a constant speed path $\gamma_i$ in $X$ from $x_i$ to $y_i$ with $\Len_d(\gamma_i) < d(x_i,y_i) + \frac{\eps}{2^{i}}$.
  Let $\gamma = \sum_{i \in I} \gamma_i$.
  Then 
\begin{align*}
   \norm{(\Len_d(\gamma_i))_{i \in I}}_p
  & < \norm{(d(x_i,y_i) + \frac{\eps}{2^{i+1}})_{i \in I}}_p
  = \norm{(d(x_i,y_i))_{i \in I} + (\frac{\eps}{2^{i+1}})_{i \in I}}_p \\
  & \leq \norm{(d(x_i,y_i))_{i \in I}}_p + \norm{(\frac{\eps}{2^{i+1}})_{i \in I}}_p
  \leq W_p(\alpha,\beta) + \eps < \infty.
  \end{align*}
  Therefore, by \cref{prop:path}\ref{it:path-2},
\[\Len_{W_p}(\gamma) \leq \norm{(\Len_d(\gamma_i))_{i \in I}}_p
  < W_p(\alpha,\beta) + \eps.\]
Hence $(\overline{D}_p(X,A),W_p)$ and $(D(X,A),W_p)$ are length spaces.
\end{proof}

The following example shows that for $p \in (1,\infty]$, if $(X,d,A)$ is a metric pair and $(X,d)$ is a length space then it need not be the case that $(X/A,d_p)$ is a length space and, more generally, that for $n \geq 1$, it need not be the case that $(D^n(X,A),W_p)$ is a length space.

\begin{example} \label{ex:non-length-space}
  Consider $(\R_{\leq}^2,d,\Delta)$ where $d$ is the metric induced by the $1$-norm. Then $(\R^2_{\leq}, d)$ is a length space.
  Let $x = (0,1)$ and $y=(10,11)$. Then in $\R_{\leq}^2/\Delta$, $d_p(x,y) = 2^{\frac{1}{p}}$.
  However, since the open balls in $(\R_{\leq}^2/\Delta,d_p)$ of radius $1$ centered at $x$ and $y$ are disjoint, for any path $\gamma$ in $\R^2_{\leq}/\Delta$ from $x$ to $y$, $L_{d_p}(\gamma) \geq 2$.
  Note that this distance is achieved by the path $\gamma$ in $(D^2(\R^2_{\leq},\Delta),W_p)$ from $x$ to $y$ given by $\gamma(t) = (0,1-t) + (10,10+t)$, which has length $2^{\frac{1}{p}}$.
  For $n \geq 1$, let $\alpha = x + (n-1) z$ and $\beta = y + (n-1) w$, where $z = (0,10)$ and $w=(0,11)$.
  Then $W_p(\alpha,\beta) = (n+1)^{\frac{1}{p}}$.
  This distance is realized by the path $\gamma$ in $(D^{n+1}(\R_{\leq}^2,\Delta),W_p)$ given by $\gamma(t) = (0,1-t) + (10,10+t) + (n-1)(0,10+t)$, but not by any path in $(D^n(\R_{\leq}^2,\Delta),W_p)$.
\end{example}

\begin{proposition} \label{prop:length}
If $(X,d, A)$ is a metric pair where $(X, d)$ is a length space and $n \geq 1$ then $(D^n(X,A),W_1)$ is a length space. As a special case, $(X/A,d_1)$ is a length space.
\end{proposition}

\begin{proof}
 Let $\alpha,\beta \in D^n(X,A)$ with $W_1(\alpha, \beta) < \infty$.
  Let $\eps > 0$.
  By \cref{def:wasserstein}, there exists a matching $\sigma$ of $\alpha$ and $\beta$ with $\hat{\sigma} = \sum_{i \in I}(x_i,y_i)$ where $\abs{I} \leq 2n$ and $\Cost_1(\sigma) = \sum_{i \in I}d(x_i,y_i) < W_1(\alpha,\beta) + \frac{\eps}{2}$.
  Since $(X,d)$ is a length space, for all $i \in I$ there is a constant speed path $\gamma_i$ in $X$ from $x_i$ to $y_i$ with $\Len_d(\gamma_i) < d(x_i,y_i) + \frac{\eps}{2n}$.
  Let $\gamma$ be the path obtained by moving the $x_i$ to $y_i$ one at a time, starting with the $x_i$ in $X \setminus A$.
  Then $\gamma \in D^n(X,A)$ and $\Len_{W_1}(\gamma) = \sum_{i \in I} \Len_d(\gamma_i) < W_1(\alpha,\beta) + \eps$.
  Therefore, $(D^n(X,A),W_1)$ is a length space.
  The special case follows from taking $n=1$.
\end{proof}

\subsection{Geodesic spaces}\label{subsection:geodesics}

Next we consider metric spaces in which each finite distance is obtained from a path.
Let $(X,d)$ be a metric space and $x, y \in X$ with $d(x,y) < \infty$. A \emph{geodesic} between $x, y \in X$ is a rectifiable path $\gamma$ from $x$ to $y$ with  $\Len_d(\gamma) = d(x,y)$.
  Say that $(X,d)$ is a \emph{geodesic space} if there exists a geodesic between any $x, y \in X$ with $d(x,y) < \infty$.
  
\begin{theorem} \label{thm:geodesic}
 Let $(X,d,A)$ be a metric pair in which $(X,d)$ is a geodesic space and $A$ is distance minimizing. Let $n \in \N$. Then
 $(\overline{D}_p(X,A),W_p)$ and $(D(X,A),W_p)$ 
 are also geodesic spaces.
\end{theorem}

\begin{proof}
  Let $\alpha,\beta \in (\overline{D}_p(X,A),W_p)$ with $W_p(\alpha, \beta) < \infty$.
  Let $\sigma$ be an optimal matching of $\alpha$ and $\beta$ with $\hat{\sigma} = \sum_{i \in I}(x_i,y_i)$, where $I = \N$ or $I = \{1,\ldots, N\}$ for some $N \in \N$ and $\abs{I} = \abs{\alpha + \beta}$ (\cref{lem:optimal-matching-finite}, \cref{thm:optimal-matching}).
Thus $\Cost_p(\sigma) = \norm{(d(x_i,y_i))_{i \in I}}_p = W_p(\alpha,\beta)$.
  Since $(X,d)$ is a geodesic space, for all $i \in I$ there is a constant speed path $\gamma_i$ in $X$ from $x_i$ to $y_i$ with $\Len_d(\gamma_i) = d(x_i,y_i)$.
  Let $\gamma = \sum_{i \in I} \gamma_i$.
  Then $\norm{(\Len_d(\gamma_i))_{i \in I}}_p
  = \norm{(d(x_i,y_i))_{i \in I}}_p
  = W_p(\alpha,\beta) < \infty$. Therefore, 
  by \cref{prop:path}\ref{it:path-2},
  $\Len_{W_p}(\gamma) \leq \norm{(\Len_d(\gamma_i))_{i \in I}}_p
  = W_p(\alpha,\beta)$.
  Therefore, $(\overline{D}_p(X,A),W_p)$ and $(D(X,A),W_p)$ are geodesic spaces.
\end{proof}

	

\begin{proposition}\label{prop:X/A_geodesic}
 Let $(X,d, A)$ be a metric pair for which $(X, d)$ is a geodesic space and $A$ is distance minimizing. For $n \geq 1$, $(D^n(X,A),W_1)$ is a geodesic space. As a special case, $(X/A,d_1)$ is a geodesic space.
\end{proposition}

\begin{proof}
 Let $\alpha,\beta \in D^n(X,A)$ with $W_1(\alpha, \beta) < \infty$.
  Let $\sigma$ be an optimal matching of $\alpha$ and $\beta$ with $\hat{\sigma} =  \sum_{i=1}^m(x_i,y_i)$, where $m \leq 2n$ (\cref{lem:optimal-matching-finite}).
  Thus $\Cost_1(\sigma) = \sum_{i=1}^md(x_i,y_i) = W_1(\alpha,\beta)$.
  Since $(X,d)$ is a geodesic space, for all $i \in \{1,\ldots,m\}$ there is a constant speed path $\gamma_i$ in $X$ from $x_i$ to $y_i$ with $\Len_d(\gamma_i) = d(x_i,y_i)$.
  Let $\gamma$ be the path obtained by moving the $x_i$ to $y_i$ one at a time, starting with the $x_i$ in $X \setminus A$.
  That is, for $i \in \{1,\ldots,m\}$, let $\bar{\gamma}_i$ be the path in $(D^n(X,A),W_1)$ given by $\bar{\gamma}_i(t) = \sum_{j=1}^{i-1} y_j + \gamma_i(t) + \sum_{j=1+1}^m x_j$.
    Let $\gamma$ be the concatenation of the paths $\bar{\gamma}_1,\ldots,\bar{\gamma}_m$.
    
  Then $\gamma \in D^n(X,A)$ and $\Len_{W_1}(\gamma) = \sum_{i \in I} \Len_d(\gamma_i) = W_1(\alpha,\beta)$.
  Therefore, $(D^n(X,A),W_1)$ is a geodesic space.
The special case follows from taking $n=1$.
\end{proof}
	
Let $n \geq 1$.
The following example shows that for a metric pair $(X,d,A)$ in which $X$ is a geodesic space but $A$ is not distance minimizing then 
$(\overline{D}_p(X,A),W_p)$, $(D(X,A),W_p)$, $(D^n(X,A),W_p)$ and $(X/A,d_{p})$
need not be geodesic spaces.
	
\begin{example}
  Let $X=\big\{(x,y) \in \R^2 \ \big| \ x\geq 1, \ y\geq 0 \big\} \cup \big\{(x,y) \in \R^2 \ \big| \ x < 1, \ y > 0 \big\}$ with the Euclidean distance $d$. Let $A=\big\{(x,y) \in \R^2 \ \big| \ x \leq 0, \ y > 0 \big\} \subset X$. Note that $A$ is closed in $X$.
  Note that $X$ is convex and is therefore geodesic. Also note that $d((1,0), A) \neq d((1,0), a)$ for all $a \in A$. Now 
  $d_p([(1,0)],A) = d((1,0),A) = 1$. However,
    any path $\gamma$ in $X/A$ between $[(1,0)]$ and $A$ has $\Len_{d_p}(\gamma) > 1$.
Furthermore, any path $\gamma$ in $\overline{D}_p(X,A)$ between $\alpha = (1,0)$ and $\beta = 0$ has $\Len_{W_p}(\gamma) > 1$.
\end{example}

Let $(X,d,x_0)$ be a pointed metric space and $p > 1$. The next example shows that if $(\overline{D}_p(X,x_0),W_p)$ or $(D(X,x_0),W_p)$ is a geodesic space, then it need not be the case that $(X,d_p)$ is a geodesic space.
\begin{example}
First, consider the metric pair $(\R_{\leq}^2,d,\Delta)$ where $d$ is the metric induced by the $1$-norm.
    Then $(\R_{\leq}^2,d)$ is a geodesic space and $\Delta$ is distance minimizing.
    By \cref{thm:geodesic},  $(\overline{D}_p(\R_{\leq}^2,\Delta),W_p)$ and $(D(\R_{\leq}^2,\Delta),W_p)$ are geodesic spaces.
    Second, consider the associated pointed metric space,  $(\R_{\leq}^2/\Delta,d_p,\Delta)$.
    Since
    $(D(\R_{\leq}^2/\Delta,\Delta),W_p) \isom (D(\R_{\leq}^2,\Delta),W_p)$ and
    $(\overline{D}_p(\R_{\leq}^2/\Delta,\Delta),W_p) \isom (\overline{D}_p(\R_{\leq}^2,\Delta),W_p)$, 
    $(D(\R_{\leq}^2/\Delta,\Delta),W_p)$ and
    $(\overline{D}_p(\R_{\leq}^2/\Delta,\Delta),W_p)$ are both geodesic spaces.
    However, by \cref{ex:non-length-space}, $(\R_{\leq}^2/\Delta,d_p)$ is not a length space. Therefore,  $(\R_{\leq}^2/\Delta,d_p)$ is not a geodesic space.
\end{example}

\begin{proposition}
  Let $(X,d,x_0)$ be a pointed metric space. If either $(D(X,x_0),W_1)$ or $(\overline{D}_1(X,x_0),W_1)$ is a geodesic space then so is $(X,d)$.
\end{proposition}

\begin{proof}
  First observe by \cref{def:wasserstein}, that if $\alpha \in D^m(X,x_0)$, $\beta \in D^n(X,x_0)$, and $\mu \in \overline{D}(X,x_0)$ then there exists $\nu \in D^{m+n}(X,x_0)$ with $\hat{\nu} \leq \hat{\mu}$ (as elements of $\Z_{+}^{X \setminus x_0}$) and $W_1(\alpha,\nu) + W_1(\nu, \beta) \leq W_1(\alpha,\mu) + W_1(\mu, \beta)$.

 Let $x, y \in (X, d)$ with $d(x,y) < \infty$. There are two cases to consider. In the first case, $d(x,y) = d(x,x_0) + d(x_0,y)$.
 By assumption, there is a geodesic $\gamma$ in $(D(X,x_0),W_1)$ or $(\overline{D}_1(X,x_0),W_1)$ from $x$ to $0$.
  Apply the observation with $\alpha=x$ and $\beta=0$. Then there is a geodesic $\gamma_1$ in $D^1(X,x_0)$ from $x$ to $0$. Similarly, there is a geodesic $\gamma_2$ in $D^1(X,x_0)$ from $0$ to $y$.
  Let $\gamma$ be the concatenation of $\gamma_1$ and $\gamma_2$.
  Then $\gamma$ is a geodesic from $x$ to $y$ that lies in $(D^1(X,x_0),W_1) \isom (X,d)$.

 In the second case, $d(x,y) < d(x,x_0) + d(x_0,y)$.
 Again by the observation above with $\alpha=x$, $\beta=y$, and $\mu = \gamma(t)$ where $\gamma$ is a geodesic between $x$ and $y$ in $(D(X,x_0),W_1)$ or $(\overline{D}_1(X,x_0),W_1)$ and $t \in [0,1]$, $\mu$ has to belong to $(D^2(X,x_0),W_1)$. Hence $\gamma$ is a geodesic in $(D^2(X,x_0),W_1)$ from $x$ to $y$. 
  Assume there exists a $t \in [0,1]$ with $\gamma(t) = u + v$, $W_1(x,\gamma(t)) = d(x,u) + d(x_0,v)$, and $W_1(\gamma(t),y) = d(u,x_0) + d(v,y)$.
  Then by \cref{lem:isomorphism}\ref{it:isomorphism-d},  
\begin{align*}
	d(x,y) & = W_1(x,y) = W_1(x,\gamma(t)) + W_1(\gamma(t),y) \\
	& = d(x,u) + d(u,x_0) + d(x_0,v) + d(v,y) \geq d(x,x_0) + d(x_0,y),
\end{align*}
  which is a contradiction.
  Therefore, we have a geodesic from $x$ to $y$ in $(D^1(X,x_0),W_1) \isom (X,d)$.
\end{proof}

\subsection{Uniqueness of geodesics} \label{sec:unique-geodesics}

Finally we show that whenever we have more than one optimal matching, we have non-unique geodesics.

\begin{proposition}\label{prop:unique_geodesic}
 Let $(X,d,A)$ be a metric pair in which $(X,d)$ is geodesic and $A$ is distance minimizing.
  If there exist $\alpha,\beta \in (\overline{D}_p(X,A),W_p)$ with distinct optimal matchings $\sigma,\tau $ then the geodesic space $(\overline{D}_p(X,A),W_p)$ does not have unique geodesics. Similarly if there exist $\alpha,\beta \in (D(X,A),W_p)$ with distinct optimal matchings $\sigma,\tau $ then the geodesic space $(D(X,A),W_p)$ does not have unique geodesics.
\end{proposition}

\begin{proof}
 Let $\sigma,\tau \in \overline{D}(X \times X, A \times A)$ be distinct optimal matchings of $\alpha,\beta \in (\overline{D}_p(X,A),W_p)$
  with $\hat{\sigma} = \sum_{i \in I} (x_i,y_i)$ and $\hat{\tau} = \sum_{j \in J} (x'_j,y'_j)$.
  For all $i \in I$, let $\gamma_i$ be a constant speed path in $X$ from $x_i$ to $y_i$ with $\Len_d(\gamma_i) = d(x_i,y_i)$ and
  for all $j \in J$, let $\gamma'_j$ be a constant speed path in $X$ from $x'_j$ to $y'_j$ with $\Len_d(\gamma'_j) = d(x_j,y_j)$.
  Let $\gamma = \sum_{i \in I} \gamma_i$ and $\gamma' = \sum_{j \in J} \gamma'_j$.
  Then $\gamma$ and $\gamma'$ are distinct geodesics from $\alpha$ to $\beta$. If $\alpha, \beta \in (D(X, A), W_p)$ and $\sigma,\tau \in D(X \times X, A \times A)$ then $\gamma, \gamma'$ are distinct geodesics in $D(X, A)$.
\end{proof}

\section{Polish spaces of persistence diagrams}        

In this section we consider conditions under which our spaces are separable and complete, which are useful conditions for measure theory and probability.

\subsection{Separability}

We will show that if we start with a metric pair on a separable metric space, then the corresponding space of persistence diagrams is separable. A topological space $X$ is \emph{separable} if it contains a countable dense subset $M$; that is, there exists a countable subset $M$ whose closure is $X$.

\begin{lemma} \label{lem:separable-subspace}
  Let $(X,d)$ be a metric space and $A \subset X$. If $X$ is separable then so is $A$.
\end{lemma}

\begin{proof}
  For a metric space (as defined in Definition~\ref{def:metric}) separability is equivalent to second countability~\cite[Lemma 17]{bubenik2018topological} and subspaces of second countable spaces are second countable~\cite[Theorem 16.2(b)]{willard2004general}.
\end{proof}

\begin{theorem}\label{prop:separability}
Let $(X,d,A)$ be a metric pair and $p \in [1, \infty]$. If $(X, d)$ is separable  then $(\overline{D}_p(X,A), W_p) $ is separable. Conversely, if $(\overline{D}_p(X,A), W_p)$ is separable then $(X/A, d_{p})$ is separable.
\end{theorem}
	
\begin{proof}
  Suppose $(X, d)$ is separable.
  By \cref{lem:separable-subspace},
  $X \setminus A$ is separable. Let $M$ be
  a countable dense subset of $X \setminus A$. Then $D(M)$ is
  countable. Now consider the isomorphism
  $f: D(X \setminus A) \isomto D(X,A)$. Note that $f(D(M))$ is
  countable. Let $\alpha \in D(X, A)$,
  $\hat{\alpha}=\sum_{i=1}^n x_i $ and $\eps >0$. Since $M$ is dense
  in $X \setminus A$ there exist $m_{1}, \dots, m_{{n}} \in M$ such
  that $m_{i} \in B_{\frac{\eps}{n}}(x_i)$ for all $1 \leq i \leq
  n$. Let $m=m_{1}+\dots+m_{n}$. Then
  \begin{align*}
    W_p(\alpha, m) & \leq  \big\| \big(d(x_i, m_{i})\big)_{i=1}^{n} \big\|_p \leq \big\| \big(d(x_i, m_{i})\big)_{i=1}^{n} \big\|_1
                     <\eps.  
  \end{align*}
  So $m \in B_{\eps}(\alpha)$ and hence $f(D(M))$ is dense in
  $D(X, A)$. Now let $\alpha \in \overline{D}_p(X,A)$. By Lemma
  \ref{lem:dist_between_lower_part_and_0}, there is $\delta >0$ such
  that
  \[W_p(\alpha, u_{\delta}(\alpha))\leq W_p(\ell_{\delta}(\alpha), 0)
    < \eps.\] Since $u_{\delta}(\alpha) \in D(X,A)$, $D(X,A)$ is dense
  in $\overline{D}_p(X,A)$.  Hence $\overline{D}_p(X,A)$ is separable.
		
  Conversely, suppose $(\overline{D}_p(X,A), W_p)$ is separable.  By
\cref{lem:separable-subspace},
  $(D^1(X,A), W_p)$ is separable. Thus by Lemma \ref{lem:isomorphism}, it follows that $(X /A, d_{p})$ is separable.
\end{proof}

\subsection{Completeness}
\label{subsection:completeness}

Next we investigate conditions under which our spaces of persistence diagrams are complete.
A metric space $X$ is \emph{complete} if every Cauchy sequence is convergent in $X$.
Note that under our definition of metric space, limits need not be unique.
	
First, we generalize the proof of the completeness of the space of barcodes equipped with the bottleneck distance in \cite{blumberg2014robust} to the space of persistence diagrams on a metric pair. 
	
\begin{theorem}\label{thm:completeness_inf}
If $(X, d, A)$ is a metric pair where $(X,d)$ is complete then $(\overline{D}_{\infty}(X, A), W_{\infty})$ is complete. 
\end{theorem} 
	
\begin{proof}
  Suppose $(X, d, A)$ is a metric pair where $(X,d)$ is complete.
  Let $(\alpha_n)$ be a Cauchy sequence in $(\overline{D}_{\infty}(X,A), W_p)$. 
  By restricting to a subsequence if necessary, we may assume that for all $n$, $W_{\infty}(\alpha_{n}, \alpha_{n+1}) < \frac{1}{2^{n+1}}$.
  Let $\sigma_n$ be a matching between $\alpha_n$ and $\alpha_{n+1}$ such that $\Cost_{\infty}(\sigma_n) < \frac{1}{2^{n+1}}$.
  Let $M_n=\big|u_{\frac{1}{2^{n-1}}}(\alpha_n)\big|$, which by assumption is finite.
  Let $\alpha_n = \sum_{i=1}^{\infty} x_i^{(n)}$, where the sequence $(d(x_i^{(n)},A))_{i=1}^{\infty}$ is nonincreasing. 
  It follows that $\sum_{i=1}^{M_n} x_i^{(n)} = u_{\frac{1}{2^{n-1}}}(\alpha_n)$.
  If $i \leq M_n$ and $(x_i^{(n)},x_j^{(n+1)}) \in \supp ( \sigma_n )$ then $d(x_j^{(n+1)},A) \geq d(x_i^{(n)},A) - d(x_i^{(n)},x_j^{(n+1)}) \geq \frac{1}{2^{n-1}} - \frac{1}{2^{n+1}} > \frac{1}{2^n}$.
Therefore, $j \leq M_{n+1}$. Since $\sigma_n$ is a matching, $M_n \leq M_{n+1}$.

For all $n \geq 1$, define an injective map $f_{n}:\{1,\ldots,M_n\} \to \{1,\ldots,M_{n+1}\}$ recursively by
\[ f_n(i) = \min \{j \ | \ (x_i^{(n)},x_j^{(n+1)}) \in \supp(\sigma_n) \text{ and } j \not\in \{f_n(1),\ldots,f_n(i-1)\}\}.
  \]
  Note that if all $x_j^{(n+1)} \in \supp({\alpha}_{n+1})$ have been previously matched, then $x_i^{(n)}$ can be matched to $x_j^{(n+1)} \in A$ and thus $f_n$ is well defined.
  Let $I_n = \{i \in \{1,\ldots,M_n\} \ | \ i \not\in \im(f_{n-1})\}$, where $M_0 = 0$.
  For $i \in I_n$, we have the sequence $x_i^{(n)},x_{f_n(i)}^{(n+1)},x_{f_{n+1}(f_n(i))}^{(n+2)},\ldots$, where the distance between consecutive terms is bounded by the sequence $\frac{1}{2^{n+1}}, \frac{1}{2^{n+2}},\ldots$.
  Thus it is a Cauchy sequence.
  Since $(X,d)$ is a complete, let $x_{n,i}$ be a limit of this Cauchy sequence.
  Note that $d(x_i^{(n)},x_{n,i}) \leq \sum_{k=1}^{\infty} \frac{1}{2^{n+k}} = \frac{1}{2^n}$.
  It follows that $d(x_{n,i},A) \geq d(x_i^{(n)},A) - d(x_{n,i},x_i^{(n)}) > \frac{1}{2^{n-1}} - \frac{1}{2^n} = \frac{1}{2^n}$.
  Let $\alpha = \sum_{n=1}^{\infty} \sum_{i \in I_n} x_{n,i}$.

  We will show that $\alpha \in \overline{D}_{\infty}(X,A)$.
  Let $n \geq 1$.
Let $i \in I_{n+1}$.
Since $\sigma_n$ is a matching of $\alpha_n$ and $\alpha_{n+1}$ and  $i \not\in \im(f_{n})$, there exists $(x_j^{(n)},x_i^{(n+1)}) \in \supp(\sigma_n)$ with $j > M_n$.
Therefore, $d(x_j^{(n)},A) < \frac{1}{2^{n-1}}$.
By the triangle inequality,
\[ d(x_{n+1,i},A) \leq d(x_{n+1,i},x_i^{(n+1)}) + d(x_i^{(n+1)},x_j^{(n)}) + d(x_j^{(n)},A) < \frac{1}{2^{n+1}} + \frac{1}{2^{n+1}} + \frac{1}{2^{n-1}} = \frac{3}{2^n}.
  \]
Therefore, $d(x_{m,i},A) < \frac{3}{2^n}$ for all $m > n$ and all $i \in I_m$.
Hence $\abs{u_{\frac{3}{2^n}}(\alpha)} \leq \sum_{m=1}^n I_m = M_n$.
It follows that for all $\delta > 0$, $\abs{u_{\delta}(\alpha)}$ is finite
and thus $\alpha \in \overline{D}_{\infty}(X,A)$.

  Finally, we claim that for $n \geq 1$, $W_{\infty}(\alpha_n,\alpha) \leq \frac{3}{2^{n}}$.
  For $i > M_n$, $d(x_i^{(n)},A) < \frac{1}{2^{n-1}}$, so there exists $a_{n,i} \in A$ such that $d(x_i^{(n)},a_{n,i}) < \frac{1}{2^{n-1}}$.
  For $m > n$ and $i \in I_m$, $d(x_{m,i},A) < \frac{3}{2^n}$, so there exists $b_{m,i} \in A$ such that $d(x_{m,i},b_{m,i}) < \frac{3}{2^n}$.
  Consider the matching of $\alpha_n$ and $\alpha$ given by
  \[ \sigma = 
  \sum_{m=1}^n \sum_{i \in I_m} (x_{f_m \circ \cdots \circ f_{n-1}(i)}^{(n)}, x_{m,i})
  + \sum_{i= M_n + 1}^{\infty} (x_i^{(n)},a_{n,i}) + \sum_{m=n+1}^{\infty} \sum_{i \in I_m} (x_{m,i},b_{m,i}).
\]
  For $1 \leq m \leq n$ and $i \in I_m$,
    $d(x_{f_m \circ \cdots \circ f_{n-1}(i)}^{(n)}, x_{m,i}) \leq \sum_{k=1}^{\infty} \frac{1}{2^{n+k}} = \frac{1}{2^n}$.
Thus $\Cost_{\infty}(\sigma) \leq \frac{3}{2^{n}}$ and 
hence $W_{\infty}(\alpha_n,\alpha) \leq \frac{3}{2^{n}}$.
Therefore, the Cauchy sequence $(\alpha_n)$ converges to $\alpha$.
\end{proof}

Second, we generalize the proof of the completeness of the space of classical persistence diagrams equipped with $p$-Wasserstein distance for $p \in [1, \infty)$~\cite{mileyko2011probability} to the space of persistence diagrams on a metric pair.

Let $(X, d, A)$ be a metric pair where $(X,d)$ is complete and let $p \in [1,\infty)$.
Let $(\alpha_n)$ be a Cauchy sequence in $(\overline{D}_p(X, A), W_p)$.

\begin{lemma}\label{lem:for_completeness_up_part}
  Let $c > 0$.
  Then there exist constants $M_{c}$ and $\delta_{c} \in (0,c)$ such that for any $\delta \in [\delta_{c}, c)$ there is $N_{\delta}$ such that for any $n> N_{\delta}$ we have $|u_{\delta}(\alpha_n)|=M_{c}$.
\end{lemma}

\begin{proof}
  Let $\delta \in (0,c)$.
By \cref{lem:DpDq}, for all $n$,
  $\alpha_{n} \in \overline{D}_{\infty}(X,A)$.
  Thus $\abs{u_{\delta}(\alpha_{n})} < \infty$.
%
  Let $M_{\sup}^{\delta}=\limsup_{n \to \infty} \abs{u_{\delta}(\alpha_n)}$ and $M_{\inf}^{\delta}=\liminf_{n \to \infty}\abs{u_{\delta}(\alpha_n)}$.
Since $(\alpha_n)$ is a Cauchy sequence, there is an $N$ such that whenever $m,n > N$ we have that $W_p(\alpha_m,\alpha_n) < \frac{\delta}{2}$.
Suppose $M_{\sup}^{\delta} =\infty$. Then there exists $N'$
such that $\abs{u_{\delta}(\alpha_{N'})}> \abs{u_{\frac{\delta}{2}}(\alpha_N)}$ which implies that 
$W_p(\alpha_{N'}, \alpha_N) \geq W_{\infty}(\alpha_{N'},\alpha_N) \geq \frac{\delta}{2}$,
which is a contradiction.
Thus $M_{\sup}^{\delta} <\infty$. It follows that $M^{\delta}_{\inf} < \infty$.

  Consider $0 < \delta_1 < \delta_2 < c$.
  Then $\abs{u_{\delta_2}(\alpha_n)} \leq |u_{\delta_1}(\alpha_n)|$.
  This implies $M_{\sup}^{\delta_2} \leq M_{\sup}^{\delta_1}$ and $M_{\inf}^{\delta_2} \leq M_{\inf}^{\delta_1}$.
  So as $\delta \in (0,c)$ increases, $M_{\sup}^{\delta}$ and $M_{\inf}^{\delta}$ are both decreasing sequences of non-negative integers.
  So there exists $\delta_{c} \in (0,c)$ such that for any $\delta \in [\delta_{c}, c)$, $M_{\sup}^{\delta}$ and $M_{\inf}^{\delta}$ are constant.
  Let $M_{\sup}$ and $M_{\inf}$ denote these constants.

	Suppose $M_{\inf} < M_{\sup}$. Let $\delta \in (\delta_{c}, c)$ and set $\eps=\delta-\delta_c > 0$. Pick a subsequence $(\alpha_{n_k})_{k}$ such that $|u_{\delta}(\alpha_{n_k})|=M_{\sup}$ and a subsequence $(\alpha_{n_l})_{l}$ such that $|u_{\delta_{c}}(\alpha_{n_l})|=M_{\inf}$. Since $(\alpha_n)$ is Cauchy we have that there is $N \in \N$ that for $k>l > N$ we have
	$W_p(\alpha_{n_k},\alpha_{n_l}) < \frac{\eps}{2}$. On the other hand, since $|u_{\delta}(\alpha_{n_k})|=M_{\sup} > M_{\inf}=|u_{\delta_{c}}(\alpha_{n_l})|$ then for any matching $\sigma=\sum_{i=1}^{\infty} (x^{n_k}_i, x^{n_l}_i) \in \overline{D}(X \times X, A \times A)$ between $\alpha_{n_k}$ and $\alpha_{n_l}$ there exists $x_i^{n_k} \in \supp( u_{\delta}(\alpha_{n_k}) )$ such that  $d(x_i^{n_l}, A) < \delta_c$. Then we have
	\[d(x_i^{n_k}, x_i^{n_l}) \geq d(x_i^{n_k}, A) - d(A, x_i^{n_l})> \delta -\delta_c =\eps,\]
	which gives us a contradiction. Therefore, we set $M_c=M_{\sup}=M_{\inf}$.

        Therefore, for all $\delta \in [\delta_c,c)$,
        $\lim_{n \to \infty}|u_{\delta}(\alpha_n)|=M_{c}$.
        Hence for all $\delta \in [\delta_c,c)$, there exists
        $N_{\delta}$ such that for any $n> N_{\delta}$ we have $|u_{\delta}(\alpha_n)|=M_{c}$. 
\end{proof}

\begin{lemma}\label{lem:for_completeness_up_part_is_Cauchy}
  Let $c > 0$ and $\delta_c$ be as in Lemma \ref{lem:for_completeness_up_part}. Then the sequence $\big(u_{\delta_c}(\alpha_n)\big)_{n=1}^{\infty}$ is Cauchy in $(\overline{D}_p(X,A), W_p)$.
\end{lemma}

\begin{proof}
  Let $\delta \in (\delta_c, c)$. By Lemma \ref{lem:for_completeness_up_part}, there exists $N_{\delta} \in \N$ such that for any $n> N_{\delta}$ we have $|u_{\delta}(\alpha_n)|=M_c$ and there exists $N_{\delta_c} \in \N$ such that for any $n> N_{\delta_c}$ we have $|u_{\delta_c}(\alpha_n)|=M_c$. Let $N=\max(N_{\delta}, N_{\delta_c})$. Then for any $n > N$, $|u_{\delta}(\alpha_n)|=|u_{\delta_c}(\alpha_n)|$.
  Thus $u_{\delta}(\alpha_n) = u_{\delta_c}(\alpha_n)$.
  Let $\eps> 0$ and set $\tilde{\eps} =\min (\eps, \delta-\delta_c)$. Since $(\alpha_n)$ is Cauchy there is $\tilde{N} \in \N$ so that for any $m,n > \tilde{N}$ we have $W_p(\alpha_m, \alpha_n) < \tilde{\eps}$.
  Now let $m,n > \max ( N, \tilde{N} )$.
  Choose a matching $\sigma 
  \in \overline{D}(X \times X, A \times A)$ between $\alpha_n$ and $\alpha_m$ such that
  $ \Cost_p(\sigma) < 
  \tilde{\eps}$.
  Let $(x,y) \in \supp(\sigma)$ such that $x \in \supp(u_{\delta_c}(\alpha_n)) = \supp(u_{\delta}(\alpha_n))$.
        Then 
	$d(x, y) \leq \Cost_p(\sigma) <
          \tilde{\eps} \leq \delta -\delta_c$.
	Thus
	\[d(y, A) \geq d(x, A) - d(x, y) \geq \delta - (\delta - \delta_c)=\delta_c.\]
	So $y \in \supp(u_{\delta_c}(\alpha_m)) = \supp(u_{\delta}(\alpha_m))$.
  Therefore, $\sigma$ restricts to a matching of $u_{\delta_c}(\alpha_n)$ and $u_{\delta_c}(\alpha_m)$.
          Hence,
            $W_p(u_{\delta_c}(\alpha_n), u_{\delta_c}(\alpha_m)) \leq 
              \Cost_p(\sigma) <
              \tilde{\eps} \leq \eps$.
\end{proof}


\begin{lemma}\label{lem:for_completeness_up_part_is_conver}
  For each $c>0$ there exists  $\alpha_c \in \overline{D}_p(X, A)$ with $|\alpha_c|=M_c$, $u_{\delta_c}(\alpha_c)=\alpha_c$, and $(u_{\delta_c}(\alpha_n)) \to \alpha_c$, where $M_c$ and $\delta_c$ are as in \cref{lem:for_completeness_up_part}.
  Furthermore, given $c_1> c_2$ and $\alpha_{c_1}$ then we may choose $\alpha_{c_2}$ such that $\alpha_{c_1} \leq \alpha_{c_2}$ (as elements of $\Z_+^{X,A}$).
\end{lemma}

\begin{proof}
  
  Let $c > 0$.
  Use \cref{lem:for_completeness_up_part}, to choose $\delta_c$ and $M_c$.
  Using \cref{lem:for_completeness_up_part}, there exists $N$ such that whenever $n > N$, $\abs{u_{\delta_c}(\alpha_n)} = M_c$.
  Let $\eps \in (0, \frac{\delta_{c}}{2})$.
  By \cref{lem:for_completeness_up_part_is_Cauchy}, $(u_{\delta_c}(\alpha_n))_n$ is a Cauchy sequence.
  Choose a subsequence $(\alpha_{n_k})_k$ with $n_1 > N$ such that
  $W_p(u_{\delta_c}(\alpha_{n_k}),u_{\delta_c}(\alpha_m)) < \frac{\eps}{2^k}$ for all $m \geq n_k$.
  Let $\sigma_k$ be a matching of $u_{\delta_c}(\alpha_{n_k})$ and $u_{\delta_c}(\alpha_{n_{k+1}})$ with $\Cost_p(\sigma_k) < \frac{\eps}{2^k}$.
  Since $\eps < \frac{\delta_c}{2}$,
  $\supp(\sigma_k) \subset (X \setminus A) \times (X \setminus A)$.
  So we can write for all $k \geq 1$, $\sigma_k = \sum_{j=1}^{M_c} (x_j^k, x_j^{k+1})$.
  Thus for all $j$, $d(x_j^k,x_j^{k+1}) \leq \Cost_p(\sigma_k) < \frac{\eps}{2^k}$.
  Therefore, $(x_j^k)_k$ is a Cauchy sequence in $(X,d)$.
  Since $(X,d)$ is complete, let $x_j \in X$ be a limit of this sequence.
  Let $\alpha_c = \sum_{j=1}^{M_c} x_j$.
  Then $u_{\delta_c}(\alpha_c) = \alpha_c$ and $\abs{\alpha_c} = M_c$.

  Let $\eps' > 0$.
  For $1 \leq j \leq M_c$, there exists $K_{j}$ such that whenever $k \geq K_{j}$,
  $d(x_j^k,x_j) < \frac{\eps'}{2} \frac{1}{M_c^{\frac{1}{p}}}$.
  Choose $n_{K_0}$ such that whenever $m \geq n_{K_0}$,
  $W_p(u_{\delta_c}(\alpha_{n_{K_0}}),u_{\delta_c}(\alpha_m)) < \frac{\eps'}{2}$
  for all $m \geq n_{K_0}$.
  Let $n_K = \max \{n_{K_0}, n_{K_1},\ldots, n_{K_M}\}$.
  Then for all $m \geq n_K$,
  \[ W_p(u_{\delta_c}(\alpha_m),\alpha_c) \leq W_p(u_{\delta_c}(\alpha_m),u_{\delta_c}(\alpha_{n_K})) + W_p(u_{\delta_c}(\alpha_{n_K}),\alpha_c) < \frac{\eps'}{2} + \frac{\eps'}{2} = \eps'. \]
  Therefore, $u_{\delta_c}(\alpha_n) \to \alpha_c$.

  Given $c_1 < c_2$ and $\alpha_{c_1}$, let $\delta_1 \in (\delta_{c_1},c_1)$.
  Using \cref{lem:for_completeness_up_part} three times, we have that there exists an $N$ such that whenever $n > N$,
  $\abs{u_{\delta_1}(\alpha_n)} = M_{c_1} = \abs{u_{\delta_{c_1}}(\alpha_n)}$
  and $\abs{u_{\delta_{c_2}}(\alpha_n)} = M_{c_2}$.
  Let $\eps \in (0,\min (\frac{\delta_{c_2}}{2},\frac{\delta_1 - \delta_{c_1}}{2}))$.
  Then the support of $\sigma_k$, a matching of $u_{\delta_{c_1}}(\alpha_{n_k})$ and $u_{\delta_{c_1}}(\alpha_{n_{k+1}})$, only contains pairs $(x,y)$ with
  $x \in \supp(u_{\delta_{c_1}}(\alpha_{n_k}))$ and
  $y \in \supp(u_{\delta_{c_1}}(\alpha_{n_{k+1}}))$.
  Therefore, we can choose $\alpha_{c_2}$ to extend $\alpha_{c_1}$.
\end{proof}

Let $\alpha = \cup_{k=1}^{\infty} \alpha_{\frac{1}{k}}$.

\begin{lemma}\label{lem:for_completeness_up_part_and_alpha}
  $\alpha \in \overline{D}_p(X, A)$ and
$\lim_{k \to \infty} W_p(\alpha_{\frac{1}{k}},\alpha) = 0$.
\end{lemma}

\begin{proof}

  Since $(\alpha_n)$ is a Cauchy sequence, there is an $M$ such that
  whenever $m,n > M$ it follows that
  $u_{\infty}(\alpha_m) = u_{\infty}(\alpha_n)$ and
  there is a $K$ such that
  for all $n$, $W_p(\ell_{\infty}(\alpha_n),0) \leq K$.
  
 Consider $k \in \N$.
  By Lemma \ref{lem:for_completeness_up_part_is_conver}, $(u_{\delta_{\frac{1}{k}}}(\alpha_n)) \to \alpha_{\frac{1}{k}}$.
  Then for all $n > M$ we have $|u_{\infty}(\alpha_{\frac{1}{k}})|=|u_{\infty}(\alpha_n)|$ and
  there is an $N$ such that whenever $n > N$ it follows that
  $W_p(\alpha_{\frac{1}{k}}, u_{\delta_{\frac{1}{k}}}(\alpha_n)) < 1$.
  So by the triangle inequality, for $n > N$ we have
  \begin{align*}
  	 W_p(\ell_{\infty}(\alpha_{\frac{1}{k}}), 0) & \leq W_p(\ell_{\infty}(\alpha_{\frac{1}{k}}), u_{\delta_{\frac{1}{k}}}(\ell_{\infty}(\alpha_n)))+W_p(u_{\delta_{\frac{1}{k}}}(\ell_{\infty}(\alpha_n)), 0) \\
  	 &\leq W_p(\alpha_{\frac{1}{k}}, u_{\delta_{\frac{1}{k}}}(\alpha_n)) + W_p(\ell_{\infty}(\alpha_n),0) \leq 1+K.
  \end{align*}
  Since this bound is independent of $k$,
  $W_p(\ell_{\infty}(\alpha), 0) \leq 1+K$.
  Thus $\alpha \in \overline{D}_p(X, A)$.

  Finally,
  $W_p(\alpha, \alpha_{\frac{1}{k}}) \leq W_p(\ell_{\delta_{\frac{1}{k}}}(\alpha),0) \to 0$ as $k \to \infty$, since $W_p(\ell_{\infty}(\alpha),0) < \infty$.
\end{proof}


\begin{lemma}\label{lem:for_completeness_low_and_up}
 For all $\eps > 0$ there exists $c_0 > 0$ such that for all $c \in (0,c_0]$ and for all $n$, $W_p(\ell_c(\alpha_n),0) < \eps$.
    Hence $W_p(u_c(\alpha_n),\alpha_n) < \eps$.
\end{lemma}

\begin{proof}
	To start, we claim that the first statement of the lemma is equivalent to the following statement: for all $\eps > 0$ there exists $c_0 >0$ such that for all $n$, $W_p(\ell_{c_0}(\alpha_n),0) < \eps$. Note that this is the special case of the lemma. And if we assume this statement is true then for all $c \in (0, c_0]$, $\ell_c(\alpha_n) \leq \ell_{c_0}(\alpha_n)$ and hence $W_p(\ell_{c}(\alpha_n),0) \leq W_p(\ell_{c_0}(\alpha_n),0) < \eps$.
	
We prove this equivalent statement by contradiction.
  Suppose there exists an $\eps >0$ such that for all $c>0$ 
  there is an $n$ with
  \[W_p(\ell_{c}(\alpha_{n}), 0) \geq \eps.\]
  Take such $\eps > 0$.
Starting with $c_1 = 1$, construct sequences $(c_k)$, $(\delta_k)$, and $(n_k)$ recursively as follows. For each $k$, apply \cref{lem:for_completeness_up_part} to $c_k$ to obtain $\delta_k \in (0,c_k)$, apply our assumption to $\delta_k$ to obtain an $n_k$ such that $W_p(\ell_{\delta_k}(\alpha_{n_k}),0) \geq \eps$, and let $c_{k+1} = \frac{\delta_k}{2}$. Note that for all $k$, $0 < c_{k+1} < \delta_k < c_k$.

Let $0< \delta < \frac{\eps}{3}$.
Since $(\alpha_{n_{k}})$ is a Cauchy sequence, there is an $N$ such that whenever $k \geq N$  we have $W_p(\alpha_{n_{k}}, \alpha_{n_{N}}) < \delta$. Since the sequence $(c_k)$ decreases to $0$ then, by Lemma \ref{lem:dist_between_lower_part_and_0},  we can pick $m \geq N$ such that $W_p(\ell_{c_m}(\alpha_{n_N}), 0) < \delta$.
Note that for all $k \geq m$, $\delta_{k} < c_m$
and thus
$W_p(\ell_{\delta_{k}}(\alpha_{n_N}), 0) < \delta$.
We showed earlier that
for all $k$, we have
$W_p(\ell_{\delta_{k}}(\alpha_{n_k}), 0) \geq \eps$.
Therefore, for all $k \geq m$, by the triangle inequality, we have 
\[W_p(\ell_{\delta_{k}}(\alpha_{n_k}), \ell_{\delta_{m}}(\alpha_{n_N})) \geq W_p(\ell_{\delta_{k}}(\alpha_{n_k}), 0) -W_p(0, \ell_{\delta_{m}}(\alpha_{n_N})) > \eps - \delta > 2\delta.\]

We will show that this inequality leads to a contradiction.
Let $k > m$. First observe that if $x \in A^{\delta_k} \setminus A$ and $y \in X \setminus A^{\delta_m}$ then
\begin{equation} \label{eq:dxy}
  d(x,y) \geq d(y,A) - d(x,A) > \delta_m - \delta_k > \delta_m - c_k \geq c_k > \delta_k > d(x,A).
\end{equation}

By \cref{def:wasserstein}, there exists a matching $\sigma$ of $\alpha_{n_k}$ and $\alpha_{n_N}$ with $\Cost_p(\sigma) < \delta$.
Next, we use $\sigma$ to obtain a matching of $\ell_{\delta_k}(\alpha_{n_k})$ and $\ell_{\delta_m}(\alpha_{n_N})$ as follows.
Let $\tau$ be the restriction of $\sigma$ to $(A^{\delta_k} \setminus A) \times X$.
Then $\tau$ is a matching of $\ell_{\delta_k}(\alpha_{n_k})$ and $(\pi_2)_*(\tau)$, and $\Cost_p(\tau) \leq \Cost_p(\sigma)$.
  We will now use \eqref{eq:dxy} to change $\tau$ to a matching $\varphi$ with lower cost.
  Let $\theta$ be the restriction of $\tau$ to $(A^{\delta_k} \setminus A) \times (A^{\delta_{m}} \setminus A)$.
  Then $(\pi_1)_*(\theta) \leq (\pi_1)_*(\tau) = \ell_{\delta_k}(\alpha_{n_k})$.
  Let $\varphi = \theta + ((\pi_1)_*(\tau) - (\pi_1)_*(\theta)) \times A$.
  Then $(\pi_1)_*(\varphi) = (\pi_1)_*(\tau) = \ell_{\delta_k}(\alpha_{n_k})$ and $(\pi_2)_*(\varphi) = (\pi_2)_*(\theta) \leq \ell_{\delta_m}(\alpha_{n_N})$.
  By \cref{def:wasserstein} and \eqref{eq:dxy}, $\Cost_p(\varphi) \leq \Cost_p(\tau)$.
  Let $\rho = \varphi + A \times (\ell_{\delta_m}(\alpha_{n_N}) - (\pi_2)_*(\varphi))$.
  Then
  $\rho$ is a matching of $\ell_{\delta_k}(\alpha_{n_k})$ and $\ell_{\delta_m}(\alpha_{n_N})$ and
  \begin{equation*}
    \Cost_p(\rho)^p \leq \Cost_p(\varphi)^p + W_p(\ell_{\delta_m}(\alpha_{n_N}),0)^p \leq \Cost_p(\sigma)^p + W_p(\ell_{\delta_m}(\alpha_{n_N}),0)^p < 2\delta^p.
  \end{equation*}
Thus $W_p(\ell_{\delta_{k}}(\alpha_{n_k}), \ell_{\delta_{m}}(\alpha_{n_N})) < (2)^{1/p}\delta < 2\delta$, which gives the desired contradiction.
	
For the second statement in the lemma,
$W_p(u_{c}(\alpha_{n}), \alpha_n) \leq W_p(0, \ell_{c}(\alpha_{n})) < \eps.$
\end{proof}

\begin{theorem}\label{thm:completeness_p}
	Let $(X,d,A)$ be a metric pair and $p \in [1, \infty)$. If $(X, d)$ is complete  then $(\overline{D}_p(X, A), W_p)$ is complete.
\end{theorem}

\begin{proof}
 Let $\eps > 0$.
  Use \cref{lem:for_completeness_low_and_up}
  to choose a $c_0$ such that 
  whenever $c \in (0,c_0]$ for all $n$, $W_p(\alpha_n,u_c(\alpha_n)) < \frac{\eps}{3}$.
  Use \cref{lem:for_completeness_up_part_and_alpha} to choose a $k \geq \frac{1}{c_0}$ such that $W_p(\alpha_{\frac{1}{k}},\alpha) < \frac{\eps}{3}$.
  Use \cref{lem:for_completeness_up_part_is_conver} to choose an $N$ such that whenever $n > N$, $W_p(u_{\delta_{\frac{1}{k}}}(\alpha_n),\alpha_{\frac{1}{k}}) < \frac{\eps}{3}$.
  Then whenever $n > N$,
  $W_p(\alpha_n,\alpha) \leq W_p(\alpha_n,u_{\delta_{\frac{1}{k}}}(\alpha_n)) + W_p(u_{\delta_{\frac{1}{k}}}(\alpha_n),\alpha_{\frac{1}{k}}) + W_p(\alpha_{\frac{1}{k}},\alpha) < \eps$.
\end{proof}

\begin{proposition} \label{prop:complete}
	  If $(\overline{D}_p(X,A), W_p)$ is complete for $1 \leq p \leq \infty$ then $(X/A, d_{p})$ is also complete.
\end{proposition}

\begin{proof}
  By Proposition \ref{prop:D^n_closed}, $D^1(X,A)$ is a closed subset of $\overline{D}_p(X,A)$. Therefore, $D^1(X,A)$ is complete.
  By Lemma~\ref{lem:isomorphism}, $(D^1(X,A),W_p)$ is isometrically isomorphic with $(X/A,d_p)$. 
   Hence $(X/A, d_p)$ is complete.
\end{proof}

   \begin{theorem}
   	Let $(X, d, A)$ be a metric pair where $(X,d)$ is complete and let $p \in [1, \infty]$. Then
   	$(D(X,A), W_p)$ is complete  if and only if $(D(X,A), W_p)=(\overline{D}_p(X,A), W_p)$.
   \end{theorem}

\begin{proof}
  Suppose $(D(X,A), W_p)$ is complete. Assume  $A$  is not isolated. 
  Then there is a sequence $(x_n)$ of distinct elements of $X$ such that $0< d(x_n, A) < \frac{1}{2^{n}}$ and $d(x_{n+1}, A) \leq d(x_n, A)$. Now consider the following sequence in $D(X, A)$:
   	\[\alpha_1=x_1, \ \alpha_2=x_1+x_2, \ \dots, \ \alpha_n=x_1+x_2+ \dots +x_n, \ \dots .\]
   	For $k > n$, we have 
   	\begin{align*}
   		W_p(\alpha_n, 
   		\alpha_k)=\norm{(d(x_{n+1}, A), \dots, d(x_{k}, A))}_p <  \norm{(\textstyle\frac{1}{2^{n+1}}, \dots, \textstyle\frac{1}{2^k})}_p  
   		\leq \norm{(\textstyle\frac{1}{2^{n+1}}, \dots, \textstyle\frac{1}{2^k})}_1 < \frac{1}{2^n}.
   	\end{align*}
   	So $(\alpha_n)$ is a Cauchy sequence. Since $D(X, A)$ is complete, then $(\alpha_n) \to \alpha$ where $\alpha \in D^m(X,A)$ for some $m \in \N$. By Lemma \ref{lem:W_p_ineq}, for any $k>m$ we have 
  $W_p(\alpha_k, \alpha) \geq W_p(\sum_{i=m+1}^k x_i, 0) \geq d(x_{m+1}, A)$
   	which leads us to a contradiction. Hence $A$ is isolated. 
   	Therefore, by Proposition \ref{prop:D_equal_D_p}, $\overline{D}_p(X,A)=D(X,A)$.
   	
   	The proof in the converse direction follows from Theorems \ref{thm:completeness_p} and \ref{thm:completeness_inf}.
   \end{proof}

	We call a complete metric space $(\overline{X}, \overline{d})$ a \emph{completion} of a metric space $(X,d)$ if $X$ is a dense subset of $\overline{X}$.
	
	\begin{proposition}\label{prop:completion}
		Let $(X, d, A)$ be a metric pair where $(X,d)$ is complete and let $p \in [1,\infty]$. Then $(\overline{D}_p(X, A), W_p)$ is a completion of $(D(X, A), W_p)$.
	\end{proposition}
	
	\begin{proof}
		Suppose $(X, d)$ is a complete metric space. Then by Theorems \ref{thm:completeness_p} and \ref{thm:completeness_inf}, $(\overline{D}_p(X, A), W_p)$ is complete. Let $\alpha \in  \overline{D}_p(X, A)$ and $\eps>0$. Then by Lemma \ref{lem:dist_between_lower_part_and_0}, there is $\delta >0$  such that $W_p(\ell_{\delta}(\alpha), 0) < \eps$. Hence,
$W_p(\alpha,  u_{\delta}(\alpha))\leq W_p(\ell_{\delta}(\alpha), 0)<\eps$.
By \cref{lem:DpDq}, $\alpha \in \overline{D}_{\infty}(X,A)$. Thus,
 $u_{\delta}(\alpha) \in D(X,A)$. Therefore, we have that $D(X,A)$ is dense in $\overline{D}_p(X, A)$.
	\end{proof}


\section{Compactness results}
\label{subsection:compactness}

In this section we study various notions of compactness in spaces of persistence diagrams.
Recall that we use a more general definition of metric than in standard (Definition~\ref{def:metric}).

\subsection{Basic definitions and elementary results}

First, we observe that a number of standard results hold for our notion of metric space. In particular, the standard proofs hold verbatim.
A topological space $X$ is \emph{limit point compact} if every infinite subset of $X$ has a limit point.
A topological space $X$ is \emph{sequentially compact} if any sequence in $X$ has a convergent subsequence.

\begin{lemma}[{\cite[Theorem 28.2]{munkres2000topology}}]\label{thm:equiv_types_compts}
  It is equivalent for a metric space to be compact, limit point compact, and sequentially compact.
\end{lemma}

\begin{lemma}[{\cite[Lemma 43.1]{munkres2000topology}}]
  A metric space is complete if and only if every Cauchy sequence has a convergent subsequence.
\end{lemma}

A subset $A$ of a metric space $(X, d)$ is \emph{totally bounded} if and only if for each $\eps >0$, there exists a finite subset $S=\{x_1,\dots, x_n\} \subset A$ such that $A \subset \cup_{i=1}^n B_{\eps}(x_i)$.

\begin{lemma}\label{lem:tot_bd_subset}
	Let $(X, d)$ be a metric space and let $Y \subset X$. If $X$ is totally bounded then $Y$ is totally bounded.
\end{lemma}

\begin{proof}
	Let $\eps >0$. Since $X$ is totally bounded, there exist a finite indexing set $I$ and  a set $\{x_i\}_{i\in I} \subset X$ such that $X \subset \cup_{i \in I} B_{\frac{\eps}{2}}(x_i)$.  Let $J=\{i \in I \ | \ B_{\frac{\eps}{2}}(x_i) \cap Y \neq \emptyset \}$. For every $i \in J$, pick $y_i \in B_{\frac{\eps}{2}}(x_i) \cap Y$.  Let $y \in Y$. Then there is $i \in J$ such that $y \in B_{\frac{\eps}{2}}(x_i)$. Therefore, $d(y, y_i) \leq d(y, x_i)+d(x_i, y_i) < \frac{\eps}{2}+\frac{\eps}{2} =\eps$. Hence $Y \subset \cup_{i \in J} B_{\frac{\eps}{2}}(y_i)$.
\end{proof}
	
\begin{lemma}[{\cite[Theorem 45.1]{munkres2000topology}}] \label{thm:compl_and_tot_bd}
  A metric space is compact if and only if it is complete and totally bounded.
\end{lemma}

\begin{proof}[Proof sketch.]
  $(\Rightarrow)$ Compact implies sequentially compact implies complete.
  Compact implies that a covering by $\eps$-balls has a finite subcover.
  $(\Leftarrow)$ For a sequence of points, use total boundedness to construct a subsequence that is Cauchy. Therefore, the space is sequentially compact and  hence compact.
\end{proof}

A subset $S$ of a topological space is said to be \emph{relatively compact} if its closure $\overline{S}$ is compact.

\begin{lemma} \label{lem:rel-cpct-then-tot-bdd}
 Let $(X,d)$ be a metric space and let $S \subset X$.
    If $S$ is relatively compact then $S$ is totally bounded.
\end{lemma}

\begin{proof}
  If $S$ is relatively compact then $\overline{S}$ is compact.
    By \cref{thm:compl_and_tot_bd}, $\overline{S}$ is totally bounded, and
    by \cref{lem:tot_bd_subset}, $S$ is totally bounded.
\end{proof}

\begin{lemma} \label{lem:complete-rel-cpct-tot-bdd}
Assume $(X,d)$ is complete and $S \subset X$.
  Then $S$ is relatively compact if and only if $S$ is totally bounded.
\end{lemma}

\begin{proof}
 By \cref{lem:rel-cpct-then-tot-bdd} if $S$ is relatively compact then $S$ is totally bounded. Assume $S$ is totally bounded.
  Then $\overline{S}$ is totally bounded.
  Since $\overline{S}$ is a closed subset of a complete metric space, $\overline{S}$ is complete.
  By \cref{thm:compl_and_tot_bd}, $\overline{S}$ is compact.
\end{proof}

\begin{lemma}\label{thm:relative_compct}
Let $(X, d)$ be a metric space. A set $S \subset X$ is relatively compact if and only if every sequence in $S$ has a subsequence that converges in $X$.
\end{lemma}
	
\begin{proof}
 Suppose $S \subset X$ is relatively compact.
 Let $(x_n)$ be a sequence in $S$.
 Then $(x_n)$ is a sequence in $\overline{S}$. Since by \cref{thm:equiv_types_compts}, $\overline{S}$ is sequentially compact,  $(x_n)$ has a subsequence $(x_{n_k})$ that converges to $x \in  \overline{S} \subset X$.
		
  Conversely, suppose that every sequence in $S$ has a subsequence that converges in $X$. Let $(x_n)$ be a sequence in $\overline{S}$. For every $n$, there exists $y_n \in S$ such that $W_p(x_n, y_n) <\frac{1}{n}$. Then $(y_n)$ is a sequence in $S$. Therefore, it has a subsequence $(y_{n_k})$ that converges to $y \in X$. Since $(y_{n_k})$ is a sequence in $\overline{S}$, $y \in \overline{S}$. Furthermore, since $(y_{n_k}) \to y$,  $(x_{n_k}) \to y$. Hence $\overline{S}$ is sequentially compact. Finally, by \cref{thm:equiv_types_compts}, $\overline{S}$ is compact.
\end{proof}

\subsection{Compactness}

We show that the space of persistence diagrams of bounded cardinality on a compact metric space is compact.

\begin{proposition}\label{prop:D^n(X)_is_cmpt}
  Let $(X, d, A)$ be a metric pair and
  $p \in [1, \infty]$. If $(X,d)$ is compact then
  $(D^n(X, A), W_p)$ is compact for all $n \in \N$.
\end{proposition}

	\begin{proof}
	Let $n \in \N$. Let $(\alpha_m)$ be a sequence in  $D^n(X, A)$. By adding elements of $A$ if necessary, each element of the sequence $(\alpha_m)$ can be written as $\alpha_m=x^{(m)}_1+ \dots +x^{(m)}_n$. By Theorem \ref{thm:equiv_types_compts}, the sequence $(x_1^{(m)})_m$ has a convergent subsequence $(x_1^{(m_{j_1})})_{j_1}$, where $(x_1^{(m_{j_1})}) \to x_1$ for some $x_1 \in X$. Similarly, $(x_2^{(m_{j_1})})_{j_1}$ has a convergent subsequence $(x_2^{(m_{j_2})})_{j_2}$, where $(x_2^{(m_{j_2})}) \to x_2$ for some $x_2 \in X$. 
	By picking a convergent subsequence of the previous subsequence $n$ times, in the end we obtain a sequence $(x_n^{(m_{j_n})})_{j_n}$, where $(x_n^{(m_{j_n})}) \to x_n$ for some $x_n \in X$.
	
	Let $\eps >0$. For every $1 \leq k \leq n$, consider the sequence $(x_k^{(m_{j_n})})_{j_n}$. For every $1 \leq k \leq n$, there exists $N_k \in N$ such that for all $j_n>N_k$ we have $d(x_k^{(m_{j_n})}, x_k) < \frac{\eps}{n}$. Set $\alpha = x_1 + \dots + x_n$. Note that $\alpha \in D^n(X, A)$. Now choose $N=\max\{N_1, \dots, N_n\}$. Then for all $j_n > N$  we have
	$W_p(\alpha_{m_{j_n}}, \alpha) \leq \norm{(d(x_k^{(m_{j_n})}, x_k))_{k=1}^{n} }_{p} < \textstyle\frac{\eps}{n} \norm{\underbrace{(1, \dots, 1)}_{n\ \text{times}}}_p \leq \eps.$
	Therefore, $(D^n(X, A), W_p)$ is sequentially compact and, by Theorem \ref{thm:equiv_types_compts}, $D^n(X, A)$ is compact.
\end{proof}


\subsection{Local compactness}

Next, we show that many spaces of persistence diagrams are not locally compact.
A topological space $X$ is \emph{locally compact} if each point has a compact neighborhood, i.e for each $x \in X$ there is a compact subset $V \in X$ such that $x \in U \subset V$ for some open subset $U \subset X$.
For a metric space $(X,d)$, since a closed subset of a compact space is compact, a point $x \in X$ has a compact neighborhood if and only if there is $\eps> 0$ such that $\overline{B}_{\eps}(x)$ is compact. Note that the second part of the previous statement holds if and only if $\overline{B}_{\eps}(x)$ is compact for all sufficiently small $\eps > 0$.

\begin{proposition} \label{prop:local-compactness}
 Let $(X,d,A)$ be a metric pair such that $A$ is not isolated and let $p \in [1, \infty]$.
  Let $\alpha \in (D(X,A),W_p)$ and let
  $ \eps \in (0,  \min_{y \in \supp(\alpha)} d(y, A))$.
  Then there is a $\delta > 0$ and a sequence $(\beta_n)$ in $D(X,A)$ such that
  for all $n$, $W_p(\beta_n,\alpha) < \eps$ and
  for all $n \neq k$, $W_p(\beta_n,\beta_k) \geq \delta$.
\end{proposition}

\begin{proof}
  Suppose $p=\infty$. Since
$A$ is not isolated,
  we can pick $x \in X$ such that
  $0 < d(x,A) < \eps$. Let $\delta = d(x,A)$.
  Now consider the sequence $(\beta_n)$ in  $D(X, A)$ given by 
\[\hat{\beta}_n=\hat{\alpha}+ n x.\]
Then for all $n \in \N$, we have
$W_{\infty}(\beta_n, \alpha)= d(x, A) =\delta < \eps,$
and for any $n, k$ with $n>k$ we have $W_{\infty}(\beta_k, \beta_n)= d(x, A) =\delta$.
		
Now suppose $p \in [1, \infty)$. Since $A$ is not isolated, there is a sequence $(x_n)$ in $X$ such that if $\eps_n=d(x_n, A)$ then $\eps_1 < \frac{\eps}{2^{1/p}}$ and for all $n \geq 1$, $\eps_{n+1} < \frac{\eps_n}{2^{1/p}}$. 
Consider the sequence $(a_n)$ starting at $\frac{\eps}{2^{1/p}}$ and strictly decreasing to $0$ given by $a_n=\frac{\eps}{2^{n/p}}$. For each $a \in (0, \frac{\eps}{2^{1/p}})$ there exists unique $n \geq 1$ such that $a \in [a_{n+1}, a_n)$.
So for each $\eps_n$ there is a unique $m_n \geq 1$ such that
\begin{equation} \label{eq:lc1}
  \frac{\eps}{2^{(m_n+1)/p}} \leq \eps_n < \frac{\eps}{2^{m_n/p}}.
\end{equation}
Using our assumption on the sequence $(x_n)$ and the second inequality in \eqref{eq:lc1}, we have
\begin{equation} \label{eq:lc3}
\eps_{n+1} < \frac{\eps_n}{2^{1/p}} < \frac{\eps}{2^{(m_n+1)/p}}.
\end{equation}
Therefore, by the first half of \eqref{eq:lc1}
  and by \eqref{eq:lc3}, we have that
  $\frac{\eps}{2^{(m_{n+1}+1)/p}} \leq \eps_{n+1} < \frac{\eps}{2^{(m_n+1)/p}}$.
  Thus
  $\frac{\eps}{2^{(m_{n+1}+1)/p}}  < \frac{\eps}{2^{(m_n+1)/p}}$, which implies
\begin{equation} \label{eq:lc2}
  m_{n+1} \geq m_n+1.
\end{equation}
For $n\geq 1$, let $\beta_n=\alpha+2^{m_n}x_n$. By the canonical matching, $W_p(\beta_n, \alpha) \leq W_p(2^{m_n}x_n, 0)$. By Lemma \ref{lem:W_p_ineq}, $W_p(\beta_n, \alpha) \geq W_p(2^{m_n}x_n, 0)$. Therefore, $W_p(\beta_n, \alpha)=W_p(2^{m_n}x_n, 0)=2^{m_n /p} \eps_n < \eps$ by \eqref{eq:lc1}.
Hence $\beta_n \in B_{\eps}(\alpha)$. Let $n > k \geq 1$. By Lemma \ref{lem:W_p_ineq} and \eqref{eq:lc2},
\begin{align*}
  W_p(\beta_n, \beta_k) & \geq W_p((2^{m_n}-2^{m_k})x_n, 0) > W_p(2^{m_n-1}x_n, 0) \\
	& =2^{(m_n-1)/p}\eps_n \geq 2^{(m_n-1)/p}\frac{\eps}{2^{(m_n+1)/p}}=\frac{\eps}{2^{2/p}} \geq \frac{\eps}{4}. \qedhere
\end{align*}
%
\end{proof}

\begin{theorem}\label{thm:not_loc_compactness}
  Let $(X,d,A)$ be a metric pair such that $A$ is not isolated and let $p \in [1, \infty]$. Then no persistence diagram in $(D(X,A), W_p)$ has a compact neighborhood. Hence  $(D(X,A), W_p)$ is not locally compact. Therefore, any compact set in $D(X,A)$ has an empty interior. 
\end{theorem}

\begin{proof}
  Let $\alpha \in (D(X,A),W_p)$ and let
  $ \eps \in (0,  \min_{y \in \supp(\alpha)} d(y, A))$.
  By Proposition~\ref{prop:local-compactness}, there is a $\delta > 0$ and a sequence $(\beta_n)$ in $D(X,A)$ such that
  for all $n$, $W_p(\beta_n,\alpha) < \eps$ and
  for all $n \neq k$, $W_p(\beta_n,\beta_k) \geq \delta$.
  Therefore, the sequence $(\beta_n) \subset \overline{B}_{\eps}(\alpha)$ does not have a limit point.
  Hence $\overline{B}_{\eps}(\alpha)$ is not limit point compact and thus by Theorem~\ref{thm:equiv_types_compts}, $\overline{B}_{\eps}(\alpha)$ is not compact.
\end{proof}

\begin{theorem} \label{thm:not_loc_compactness_D_p}
Let $(X,d,A)$ be a metric pair where $A$ is not isolated. Let $p \in [1, \infty]$. Then no element in $(\overline{D}_{p}(X,A), W_p)$ has a compact neighborhood. Hence $(\overline{D}_{p}(X,A), W_p)$ is not locally compact. Therefore, any compact set in $\overline{D}_{p}(X,A)$ has an empty interior.
\end{theorem}
	
\begin{proof}
  Let $\beta \in \overline{D}_p(X,A)$ and let $\eps > 0$.
  By Corollary \ref{cor:seq_convergent}, there is $\alpha \in D(X,A)$ such that $W_p(\alpha, \beta) < \frac{\eps}{2}$.
  Therefore, $B_{\frac{\eps}{2}}(\alpha) \subset B_{\eps}(\beta)$. 
By Proposition~\ref{prop:local-compactness}, there is a $\delta > 0$ and a sequence $(\beta_n)$ in $D(X,A)$ such that
  for all $n$, $W_p(\beta_n,\alpha) < \frac{\eps}{2}$ and
  for all $n \neq k$, $W_p(\beta_n,\beta_k) \geq \delta$.
  Therefore, the sequence $(\beta_n) \subset \overline{B}_{\frac{\eps}{2}}(\alpha) \subset \overline{B}_{\eps}(\beta)$ does not have a limit point.
  Hence $\overline{B}_{\eps}(\beta)$ is not limit point compact and thus by Theorem~\ref{thm:equiv_types_compts}, $\overline{B}_{\eps}(\beta)$ is not compact.
\end{proof}	

		
\begin{theorem}\label{thm:loc_compactness}
Let $(X,d,A)$ be a metric pair with $A$ isolated, $(X,d)$ locally compact and $p \in [1, \infty]$. Then $(D(X,A), W_p)$ is locally compact and $(D(X,A), W_p)=(\overline{D}_p(X,A), W_p)$.
\end{theorem}
 
\begin{proof} 
		Let $\delta>0$ such that $A^{\delta}=A$. By Proposition \ref{prop:D_equal_D_p}, $\overline{D}_p(X,A)=D(X,A)$. Let $\alpha \in D(X,A)$ with $\hat{\alpha}=x_1+\dots +x_m$. Since $X$ is locally compact then for each $x_i$ where $1\leq i \leq m$ there is an $\eps_i >0$ such that $\overline{ B}_{\eps_i}(x_i)$ is compact. 
		Let $\eps=\min \{\frac{\delta}{2}, \eps_1, \dots, \eps_m\}$. Since for $1 \leq i \leq m$ $\overline{B}_{\frac{\eps}{2}}(x_i)$ is closed in a compact set,  $\overline{B}_{\frac{\eps}{2}}(x_i)$ is compact in $X$. Now consider the closed ball $\overline{B}_{\eps}(\alpha)$ in $D(X, A)$. 
		
		First, we will show that $\overline{B}_{\eps}(\alpha)$ is totally bounded in $D(X, A)$. Given $\eps_0>0$, let $\tilde{\delta}=\frac{\eps_0}{m^{1/p}}$. For $1 \leq i \leq m$, $\overline{B}_{\eps}(x_i)$ is totally bounded in $X$. Thus there exist $k_i \in \N$ and a set $Z_i=\{z_1^i, \dots, z_{k_i}^i\} \subset X$ where $\overline{B}_{\eps}(x_i) \subset \cup_{j=1}^{k_i}B_{\tilde{\delta}}(z_j^i)$. Now let $Z=\{z_1+\dots +z_m | \ z_i \in Z_i\} \subset D(X, A)$. Note that $Z$ is a finite set. Let $\beta \in \overline{B}_{\eps}(\alpha)$. Since $A^{\delta}=A$ and $\eps \leq \frac{\delta}{2}$, we may write $\hat{\beta}=y_1+\dots+y_m$.  Let $\sigma=\sum_{i=1}^m(x_i, y_i)$ be a matching between $\alpha$ and $\beta$ such that $W_p(\alpha, \beta)= \Cost_p(\sigma)$. Then for $1 \leq i \leq m$, we have
		$d(x_i, y_i) \leq W_p(\alpha, \beta) \leq \eps$.
		Therefore, $y_i \in \overline{B}_{\eps}(x_i)$ for all $1 \leq i \leq m$. Thus, for $1 \leq i \leq m$, there exists $z_i \in Z_i$ such that $y_i \in B_{\tilde{\delta}}(z_i)$. Then
		\[W_p(\beta, z_{1}+\dots + z_{m}) \leq \norm{(d(y_1, z_{1}), \dots,d(y_m, z_{m}))}_p 
			< \norm{\underbrace{(\textstyle\frac{\eps_0}{m^{1/p}}, \dots,\frac{\eps_0}{m^{1/p}})}_{ \text{$m$ times}}}_p =\eps_0.\]
		Therefore, $\beta \in B_{\eps_0}(z_{1}+\dots+z_{m})$ which implies that $\overline{B}_{\eps}(\alpha) \subset \cup_{z \in Z}B_{\eps_0}(z)$. Hence $\overline{B}_{\eps}(\alpha) \cup Z$ is totally bounded in $D(X, A)$. By Lemma \ref{lem:tot_bd_subset}, $\overline{B}_{\eps}(\alpha)$ is totally bounded in $D(X, A)$.
		
		Second, we will show that $\overline{B}_{\eps}(\alpha)$ is complete in $D(X, A)$. Let $(\beta_n)$ be a Cauchy sequence in $\overline{B}_{\eps}(\alpha)$. Then we may write  $\hat{\beta}_n=\sum_{i=1}^m y_i^n$ such that for all $1\leq i \leq m$, $(y_i^n)_n$ is a Cauchy sequence in $\overline{B}_{\eps}(x_i)$. Since $\overline{B}_{\eps}(x_i)$ is compact and hence complete, there  is a $y_i \in \overline{B}_{\eps}(x_i)$ such that $(y_i^n)_n \to y_i$. Let $\beta=y_1+\dots+y_n$. 
	    Now let $\eps'>0$.  For every $1 \leq i \leq m$ there exists $N_i \in \N$ such that for any $n > N_i$ we have $d(y_i^n, y_i) < \frac{\eps'}{m^{1/p}}$. This implies that for all $n > \max \{N_1, \dots, N_m\}$ we have
		\begin{align*}
			W_p(y_1^n+\dots+y_m^n, y_1+\dots+y_m) & \leq \norm{(d(y_1^n, y_1), \dots,d(y_m^n, y_m) )}_p 
			 \leq \norm{\underbrace{(\textstyle\frac{\eps'}{m^{1/p}}, \dots,\frac{\eps'}{m^{1/p}} )}_{ \text{$m$ times}}}_p =\eps'.
		\end{align*}
		Thus $(\beta_n) \to y_1+\dots+y_m$. Since $\overline{B}_{\eps}(\alpha)$ is closed,
		$y_1+\dots+y_m \in \overline{B}_{\eps}(\alpha)$. Hence $\overline{B}_{\eps}(\alpha)$ is complete.
		
		 Finally, by Theorem \ref{thm:compl_and_tot_bd}, $\overline{B}_{\eps}(\alpha)$ is a compact neighborhood of $\alpha$.
    \end{proof}

\begin{lemma}[{\cite[Corollary 29.3]{munkres2000topology}}]\label{lem:loc_compt_subset}
    Let $X$ be locally compact topological space and let $C$ be a closed subset of $X$. Then $C$ is locally compact.
\end{lemma}
		
	\begin{proposition}\label{prop:X/A_loc_compact}
		If $(\overline{D}_p(X,A), W_p)$ is locally compact  then $(X/A, d_{p})$ is locally compact and $(\overline{D}_p(X,A), W_p) = (D(X,A), W_p)$.  
	\end{proposition}
	
	\begin{proof}
          By Proposition \ref{prop:D^n_closed}, $D^1(X,A)$ is a closed subset of $\overline{D}_p(X,A)$. By Lemma \ref{lem:loc_compt_subset}, $D^1(X,A)$ is locally compact. From Lemma \ref{lem:isomorphism} it follows that $(X/A, d_{p})$ is locally compact.
          Also,
          by Theorem \ref{thm:not_loc_compactness_D_p},
          $A$ is isolated and hence, by Proposition \ref{prop:D_equal_D_p}, $(\overline{D}_p(X,A), W_p)=(D(X,A), W_p)$.
	\end{proof}

\subsection{$\sigma$-compactness}
A topological space $X$ is \emph{$\sigma$-compact} if it is the union of countably many compact subspaces.

    	
    	
    	
    	
	
	\begin{proposition}\label{prop:sigma_comp_pointed_case}
		Let $(X, d, x_0)$ be a pointed metric space such that $X$ is $\sigma$-compact. Then $(D(X, x_0), W_p)$ is $\sigma$-compact.
	\end{proposition}
	
	\begin{proof}
		Suppose $X=\cup_{i=1}^{\infty}K_i $ where $K_i \subset X$ is compact for all $i$. Since a finite union of compact sets is compact, then without loss of generality, we can assume that $x_0 \in K_1 \subset K_2 \subset \dots $. Let $\alpha \in D(X, x_0)$ such that $\hat{\alpha}=x_1+\dots+x_m$ for some  $m \in \N$. Then there is $n \geq m$ such that  $x_i \in K_n$ for $1 \leq i \leq m$. Thus $\alpha \in D^n(K_n, x_0)$ which implies that  $D(X, x_0)=\cup_{n=1}^{\infty}D^n(K_n, x_0)$. By Proposition \ref{prop:D^n(X)_is_cmpt},  $D^n(K_n, x_0)$ is compact for all $n \in \N$.  Therefore, $D(X, x_0)$ is $\sigma$-compact.
		
	\end{proof}

\begin{proposition} \label{prop:sigma-compact}
Let $(X,d,A)$ be a metric pair where $X$ is $\sigma$-compact, then $(D(X, A), W_p)$ is $\sigma$-compact.
\end{proposition}
	
	\begin{proof}
		Since $(X, d)$ is $\sigma$-compact and compactness is preserved by taking quotients then $(X/A, d_p)$ is also $\sigma$-compact. By Proposition \ref{prop:sigma_comp_pointed_case}, $(D(X/A, A), W_p)$ is $\sigma$-compact. Finally, from Lemma \ref{lem:isomorphism} it follows that $(D(X, A), W_p)$ is $\sigma$-compact.
	\end{proof}
	
	\begin{example}
		The spaces of classical persistence diagrams $D(\R^2_{\leq}, \Delta)$ and $D(\overline{\R}^2_{\leq}, \overline{\Delta})$ are $\sigma$-compact.
	\end{example}

Following \cite{perea2019approximating} we have the next result.

\begin{proposition}\label{prop:D_p_is_not_sigma_compact}
 Let $(X,d,A)$ be a metric pair and $p \in [1, \infty]$. If $(X,d)$ is complete and separated and $A$ is not isolated then $(\overline{D}_p(X,A), W_p)$ is not $\sigma$-compact.
\end{proposition}

\begin{proof}
	By Theorem \ref{thm:not_loc_compactness_D_p}, every compact set in $\overline{D}_p(X, A)$ has empty interior.
    By Theorems \ref{thm:completeness_p} and \ref{thm:completeness_inf},  $(\overline{D}_p(X,A),W_p)$ is complete. By the Baire Category Theorem \cite[Theorem 4.35]{bubenik2018topological}, a complete metric space cannot be expressed as a countable union of closed sets with empty interior. Since $(X,d)$ is separated, so is $(\overline{D}_p(X,A), W_p)$. Hence $\overline{D}_p(X,A)$ is Hausdorff and thus a compact subset of $\overline{D}_p(X,A)$ is closed. Therefore, $(\overline{D}_p(X,A), W_p)$ is not $\sigma$-compact.
\end{proof}

\begin{example}
	The spaces $\overline{D}_p(\R^2_{\leq}, \Delta)$ and $\overline{D}_p(\overline{\R}^2_{\leq}, \overline{\Delta})$ are not $\sigma$-compact.
\end{example}

 \subsection{Hemicompactness} 

A topological space $X$ is \emph{hemicompact} if it has a sequence of compact subsets such that every compact subset of the space lies inside some compact set in the sequence.



\begin{lemma}\label{lem:hemic_implies_loc_comp}
	A hemicompact metric space is locally compact.
\end{lemma}

\begin{proof}
	Let $\{K_n\}$ be a sequence of compact sets such that any compact subset of $X$ is contained in one of the $K_n$. Let $x \in X$ and $\{U_n\}$ be a countable local base at $x$. After replacing each $U_n$ with $\bigcap_{i=1}^n U_i$, we can assume that the sequence $\{U_n\}$ is decreasing.
	
	Suppose $U_n \not\subset K_n$ for all $n \in \N$. Then for all $n \in \N$ there exists $x_n \in U_n \setminus K_n$. Let $S=\{x_n \ | \  n \in \N\} \cup \{x\}$. Since $\{U_n\}$ is a decreasing local base of $x$, $(x_n) \to x$. So $S$ is limit point compact and by Theorem \ref{thm:equiv_types_compts}, it is compact. Thus, by assumption, there is $m \in \N$ such that $S \subset K_m$. However, $x_m \notin K_m$, which is a contradiction. Therefore, there exists $n \in \N$ such that $U_n \subset K_n$. So $x$ has an open neighborhood contained in a compact neighborhood. Hence $X$ is locally compact.
\end{proof}

\begin{theorem}\label{thm:not_hemicompact}
	Let $(X, d, A)$ be a metric pair with $A$ not isolated and $p \in [1, \infty]$. Then $(D(X, A), W_p)$ and $(\overline{D}_p(X, A), W_p)$ are not hemicompact.
\end{theorem}

\begin{proof}
	By Theorems \ref{thm:not_loc_compactness} and \ref{thm:not_loc_compactness_D_p}, $(D(X, A), W_p)$ and $(\overline{D}_p(X, A), W_p)$ are not locally compact. By Lemma \ref{lem:hemic_implies_loc_comp}, they are not hemicompact.
\end{proof}

\begin{example}
	The spaces $D(\R^2_{\leq}, \Delta)$ and $D(\overline{\R}^2_{\leq}, \overline{\Delta})$ are not hemicompact.
\end{example}

\begin{proposition}
	If  $(\overline{D}_p(X,A), W_p)$ is hemicompact  then $(X/A, d_p)$ is locally compact and $(\overline{D}_p(X,A), W_p)=(D(X,A), W_p)$.
\end{proposition}

\begin{proof}
	Suppose $(\overline{D}_p(X,A), W_p)$ is hemicompact. Then by Lemma \ref{lem:hemic_implies_loc_comp} it is also locally compact. Then by Theorem \ref{prop:X/A_loc_compact}, $(X/A, d_p)$ is locally compact and $(\overline{D}_p(X,A), W_p)=(D(X,A), W_p)$.
\end{proof}
	
\begin{proposition}[{\cite[Theorem 1.4]{mccoy1980countability}}] 
  \label{prop:CXY_is_metrizable}
Let $X$ be a completely regular space and let $Y$ be a topological  space that contains a nontrivial path. If the space $C(X,Y)$ with the compact-open topology is first countable then $X$ is hemicompact.
\end{proposition} 

\begin{theorem} \label{thm:not-metrizable}
	Let $(X, d, A)$ be a metric pair with $A$ not isolated. Let $Y$ be a topological space that contains a nontrivial path and let $p \in [1, \infty]$. Then $C((D(X, A), W_p), Y)$ and $C((\overline{D}_p(X,A), W_p), Y)$ are not metrizable.
\end{theorem}

\begin{proof}
Let $Z = (D(X,A),W_p)$ or $(\overline{D}_p(X,A),W_p)$.
  Then $Z$ is a metric space and hence completely regular.
  Assume that $C(Z,Y)$ is metrizable.
  Then $C(Z,Y)$ is first countable.
  Then by \cref{prop:CXY_is_metrizable}, $Z$ is hemicompact.
  But this contradicts \cref{thm:not_hemicompact}.
\end{proof}


\subsection{Total boundedness} \label{sec:rel_compactness}

We wish to characterize relatively compact sets of persistence diagrams.
Recall that in complete metric spaces relative compactness is equivalent to
total boundedness (\cref{lem:complete-rel-cpct-tot-bdd}).
In this section we characterize totally bounded sets of persistence diagrams.
Observe that in a metric space $(X,d)$ in which there exist $x,y$ with $d(x,y)=\infty$, the set $\{x,y\}$ is totally bounded but not bounded.
Let $(X,d,A)$ be a metric pair.

\begin{definition}\label{def:unif_upper_finite}
A set $S \subset \overline{D}(X,A)$ is \emph{uniformly upper finite} if for any $\eps >0$ there exists $M_{\eps} \geq 0$ such that $|u_{\eps}(\alpha)| \leq M_{\eps}$ for all $\alpha \in S$.
\end{definition}

\begin{definition}\label{def:upper_tot_bd}
  A set $S \subset \overline{D}(X,A)$ is \emph{upper totally bounded} if for any $\eps >0$, the set
   $u_{\eps}(S)=\cup_{\alpha \in S} \supp( u_{\eps}(\alpha) )\subset X$ 
  is totally bounded.
\end{definition}


\begin{definition}\label{def:lower_uniform}
  Let $1 \leq p \leq \infty$.
  A set $S \subset (\overline{D}(X,A),W_p)$ is
 \emph{uniformly $p$-vanishing} if for any $\eps >0$ there exists $\delta>0$ such that $W_p(\ell_{\delta}(\alpha), 0) < \eps$ for all $\alpha \in S$.
\end{definition}

Note that for every $\alpha \in \overline{D}(X, A)$ and every $\eps>0$, $W_{\infty}(\ell_{\frac{\eps}{2}}(\alpha), 0) < \eps$. Therefore,  every $S \subset \overline{D}_{\infty}(X,A)$ is uniformly $\infty$-vanishing. For an example of a set $S$ that is not uniformly $p$-vanishing when $ p \in [1, \infty)$ see \cite[Example 19]{mileyko2011probability}.

\begin{theorem}[\textbf{Criterion of totally bounded sets}]\label{thm:criterion_tot_bd}
  A set $S \subset (\overline{D}_{\infty}(X,A),W_{\infty})$ is totally bounded
  if and only if
  $S$ is
  uniformly upper finite and
  upper totally bounded.
  For $p \in [1, \infty)$, a set $S \subset (\overline{D}_p(X,A),W_p)$ is totally bounded
  if and only if
  $S$ is
  uniformly upper finite,
  upper totally bounded,
  and uniformly $p$-vanishing.
\end{theorem}


	
\begin{lemma}\label{lem:S_is_unif_upper_finite}
  Let $p \in [1, \infty]$.
  If $S \subset (\overline{D}_p(X,A),W_p)$ is totally bounded then $S$ is uniformly upper finite.
\end{lemma}

\begin{proof}
Suppose $S \subset \overline{D}_{p}(X,A)$ is not uniformly upper finite. Then there exist $\eps>0$ and a sequence $(\beta_n)$ in $ S$ such that for all $n$, $|u_{\eps}(\beta_n)| \geq n$.
Since $S$ is totally bounded, there exist $\alpha_1, \dots, \alpha_n \in S$ such that $S \subset B_{\frac{\eps}{2}}(\alpha_1) \cup \dots \cup B_{\frac{\eps}{2}}(\alpha_n)$. Now, there exists $1 \leq i \leq n$ such that $B_{\eps}(\alpha_i)$ contains a subsequence $(\beta_{n_k})$ of sequence $(\beta_n)$. Therefore, $W_p(\alpha_i, \beta_{n_k}) < \frac{\eps}{2}$.
For sufficiently large $K$,
$|u_{\eps}(\beta_{n_{K}})| > |u_{\frac{\eps}{2}}(\alpha_i)|$. This implies that for any matching between $\beta_{n_{K}}$ and $\alpha_i$ there must be a point in  $ u_{\eps}(\beta_{n_{K}})$ that is matched to a point in $\ell_{\frac{\eps}{2}}(\alpha_{i})$.
Note that for every $y \in \supp(u_{\eps}(\beta_{n_{K}}))$ and $x \in \supp(\ell_{\frac{\eps}{2}}(\alpha_{i}))$, $d(x, y) \geq d(y, A)-d(x, A)> \eps -\frac{\eps}{2}=\frac{\eps}{2}$. Hence $W_p(\beta_{n_{K}}, \alpha_i) >\frac{\eps}{2}$, which is a contradiction. 
\end{proof}	
	
\begin{lemma}\label{lem:S_is_upper_tot_bounded}
  Let $p \in [1, \infty]$.
  If  $S \subset (\overline{D}_p(X,A),W_p)$ is totally bounded then $S$ is upper totally bounded.
\end{lemma}

\begin{proof}
Assume  $S \subset \overline{D}_p(X,A)$ is totally bounded and let $\eps>0$. Choose $\delta < \eps$. Then there exist $\alpha_1, \dots, \alpha_n \in S$ such that $S \subset B_{\frac{\delta}{2}}(\alpha_1) \cup \dots \cup B_{\frac{\delta}{2}}(\alpha_n)$. Let $\alpha \in S$. Then there is $1 \leq i \leq n$ such that $W_p(\alpha, \alpha_i) < \frac{\delta}{2}$. By \cref{def:wasserstein}, there is a matching $\sigma$ between $\alpha$ and $\alpha_i$ such that $\Cost_p(\sigma) < \frac{\delta}{2}$. For any $x \in  \supp (u_{\delta}(\alpha))$ and any $y \in \supp(\ell_{\frac{\delta}{2}}(\alpha_i))$, $d(x, y)\geq d(x, A)-d(y, A)> \delta -\frac{\delta}{2}=\frac{\delta}{2}$. Therefore, under the matching $\sigma$ every $x \in  \supp (u_{\delta}(\alpha))$ has to be matched to some $y \in \supp(u_{\frac{\delta}{2}}(\alpha_i))$.  Let $Z=\cup_{1 \leq i \leq n}\supp (u_{\frac{\delta}{2}}(\alpha_i)) \subset X$. Note that by Definition \ref{def:D_p} and Lemma \ref{lem:DpDq}, $Z$ is a finite set. Thus for any $x \in u_{\delta}(S)$ there is $z \in Z$ such that $d(x,z)< \delta$. Hence $u_{\delta}(S) \cup Z $ is totally bounded. Since $\delta < \eps$, $\supp(u_{\eps}(\alpha)) \subset \supp(u_{\delta}(\alpha))$. Thus $u_{\eps}(S) \subset u_{\delta}(S) \cup Z$. Finally, by Lemma \ref{lem:tot_bd_subset},  $u_{\eps}(S)$ is totally bounded.
\end{proof}
	
\begin{lemma}\label{lem:S_is_lower_uniform}
Let $p \in [1, \infty)$. If  $S \subset (\overline{D}_p(X,A),W_p)$ is totally bounded then $S$ is uniformly $p$-vanishing.
\end{lemma}
	
\begin{proof}
Let $\eps>0$ and let $0 < \eps_0 < \min(\frac{\eps}{4}, \eps(1-\frac{1}{2^p})^{1/p})$. Assume $S \subset \overline{D}_p(X,A)$ is totally bounded. Then there are $\alpha_1, \dots, \alpha_m \in S$ such that $S \subset B_{\eps_0}(\alpha_1) \cup \dots \cup B_{\eps_0}(\alpha_m)$.  Using Lemma \ref{lem:dist_between_lower_part_and_0} $m$ times we can pick  $0<\delta <\eps_0$ such that $W_p(\ell_{\delta}(\alpha_n), 0)\leq \frac{\eps}{4}$ for all $1 \leq n \leq m$. 
		
We prove the lemma by contradiction.
Suppose there is an $\alpha \in S$ such that $W_p(\ell_{\frac{\delta}{2}}(\alpha), 0)>\eps$.
Since $\alpha \in S$,
$\alpha \in  B_{\eps_0}(\alpha_n)$ for some $1 \leq n \leq m$.
Let $\sigma$ be a matching between $\alpha$ and $\alpha_n$ with $\Cost_p(\sigma) < \eps_0$.
Let $\tau$ be the restriction of $\sigma$ to $(A^{\frac{\delta}{2}} \setminus A) \times  (X \setminus A^{\delta})$. Let $\theta$ be the restriction of $\sigma$ to $(A^{\frac{\delta}{2}} \setminus A) \times  (A^{\delta} )$. Then $\ell_{\frac{\delta}{2}}(\alpha)=(\pi_1)_*(\tau)+(\pi_1)_*(\theta)$, $(\pi_2)_*(\tau) \leq u_{\delta}(\alpha_n)$ and $(\pi_2)_*(\theta) \leq \ell_{\delta}(\alpha_n)$. 

We consider two cases.
In the first case, $W_p((\pi_1)_*(\theta), 0) > \frac{\eps}{2}$.
Then, by the triangle inequality, 
\[\Cost_p(\sigma) \geq \Cost_p(\theta) \geq W_p((\pi_1)_*(\theta), (\pi_2)_*(\theta)) \geq W_p((\pi_1)_*(\theta), 0)-W_p((\pi_2)_*(\theta), 0) > \textstyle\frac{\eps}{2}-W_p(\ell_{\delta}(\alpha_n), 0) \geq \textstyle\frac{\eps}{4}.\]
In the second case, $W_p((\pi_1)_*(\theta), 0)\leq  \frac{\eps}{2}$.
Then
\[W_p((\pi_1)_*(\tau), 0)^p =W_p(\ell_{\frac{\delta}{2}}(\alpha), 0)^p- W_p((\pi_1)_*(\theta), 0 )^p > \eps^p-(\textstyle\frac{\eps}{2})^p \geq \eps^p(1-\frac{1}{2^p}). \]  
Furthermore, whenever $x \in \supp ((\pi_1)_*(\tau))$ and $y \in \supp ((\pi_2)_*(\tau))$ we have $d(x, A) < \frac{\delta}{2}$ and \[d(x, y) \geq d(y, A)-d(x, A) \geq \delta - \textstyle\frac{\delta}{2}=\frac{\delta}{2}.\]
		This implies that $\Cost_p(\tau) \geq W_p((\pi_1)_*(\tau), 0)$.
		Therefore,
		$\Cost_p(\sigma) \geq \Cost_p(\tau) \geq W_p((\pi_1)_*(\tau), 0) > \eps(1-\frac{1}{2^p})^{1/p}.$
		Hence in either case $\Cost_p(\sigma) > \min(\frac{\eps}{4}, \eps(1-\frac{1}{2^p})^{1/p} )$. Since $\sigma$ was an arbitrary matching between $\alpha$ and $\alpha_n$, we have that $W_p(\alpha, \alpha_n) \geq\min(\frac{\eps}{4}, \eps(1-\frac{1}{2^p})^{1/p} ) > \eps_0$,
                which is a contradiction.
	\end{proof}
	
\begin{proof}[\textbf{Proof of Theorem \ref{thm:criterion_tot_bd}}]
Let $p \in [1, \infty]$. Suppose $S \subset (\overline{D}_{p}(X,A),W_p)$ is totally bounded.
By Lemma \ref{lem:S_is_upper_tot_bounded}, $S$ is upper totally bounded.
By Lemma \ref{lem:S_is_unif_upper_finite}, $S$ is uniformly upper finite. Also, 
for  $p \in [1, \infty)$,
by Lemma \ref{lem:S_is_lower_uniform},
$S$ is uniformly $p$-vanishing.
		
Conversely, suppose that $S \subset \overline{D}_p(X,A)$ is upper totally bounded,
uniformly upper finite, and
uniformly $p$-vanishing.
Note that if $p=\infty$ then the last condition is automatically satisfied.
Let $\eps >0$. Since $S$ is uniformly $p$-vanishing there is a $\delta>0$ such that $W_p(\ell_{\delta}(\alpha),0)< \frac{\eps}{2}$ for all $\alpha \in S$. Since $S$ is uniformly upper finite there exists an $M$ such that $|u_{\delta}(\alpha)|\leq M$ for all $\alpha \in S$. Since $S$ is upper totally bounded, $u_{\delta}(S)$ is totally bounded.
Hence for some $n \geq 0$ there exist $x_1, \dots, x_n \in u_{\delta}(S)$ such that for any $x \in u_{\delta}(S)$
$d(x, x_i)< \frac{\eps}{2M}$
for some $1 \leq i \leq n$.
Let $Z = \{ \sum_{i=1}^n c_i x_i \ | \ c_i \in \Z_{\geq 0}, \sum_{i=1}^n c_i \leq M\}$.
Note that $Z$ is finite.
Now let $\beta \in S$.
Then $u_{\delta}(\beta) = y_1 + \cdots + y_N$ for some $N \leq M$.
For each $1 \leq j \leq N$ there exists $1 \leq i_j \leq n$ such that $d(y_j,x_{i_j}) \leq \frac{\eps}{2M}$.
Let $\alpha = x_{i_1} + \cdots + x_{i_N}$.
Then $\alpha \in Z$ and
\[W_p(u_{\delta}(\beta) , \alpha) < \norm{(\underbrace{\textstyle\frac{\eps}{2M}, \dots, \frac{\eps}{2M}}_{N \text{ times }})}_p \leq \norm{(\underbrace{\textstyle\frac{\eps}{2M}, \dots, \frac{\eps}{2M}}_{N \text{ times}})}_1 \leq \textstyle\frac{\eps}{2}.\]
Therefore,
$W_p(\beta, \alpha) \leq  \norm{(W_p(u_{\delta}(\beta), \alpha), W_p(\ell_{\delta}(\beta),0))}_p <
\norm{(\textstyle\frac{\eps}{2},\textstyle\frac{\eps}{2})}_p \leq \eps.$
Hence $\beta \in B_{\eps}(\alpha)$. Therefore, $S \cup Z$ is totally bounded. By Lemma \ref{lem:tot_bd_subset} , $S$ is totally bounded.
\end{proof}

\subsection{Relative compactness}
\label{sec:relative-compactness}
 
In this section we characterize relatively compact sets of persistence diagrams.
Let $(X,d,A)$ be a metric pair.
We will denote sequences by indexed sets $(x_i)_{i \in I}$, where $I$ is a totally ordered set with $I \isom \N$.
A subsequence is given by $(x_i)_{i \in I'}$, where $I' \subset I$ and $I' \isom \N$.
We will use the convention that $\frac{1}{0} = \infty$.

\begin{definition}\label{def:upper_compact}
A set $S \subset \overline{D}(X,A)$ is \emph{upper  relatively compact} if for any $\eps >0$, the set
$u_{\eps}(S)=\cup_{\alpha \in S} \supp( u_{\eps}(\alpha) )\subset X$ 
is relatively compact.
\end{definition}

\begin{theorem}[\textbf{Criterion of relatively compact sets}] \label{thm:criterion_rel_compct}
 A set $S \subset (\overline{D}_{\infty}(X,A),W_{\infty})$ is relatively compact
 if and only if
 $S$ is
 uniformly upper finite and
 upper relatively compact.
 For $p \in [1, \infty)$, a set $S \subset (\overline{D}_p(X,A),W_p)$ is
 relatively compact
 if and only if $S$ is
 uniformly upper finite,
 upper relatively compact,
 and uniformly $p$-vanishing.
\end{theorem}

\begin{lemma} \label{lem:upper-relatively-compact}
  Let $p \in [1,\infty]$.
  If $S \subset (\overline{D}_p(X,A),W_p)$ is relatively compact then $S$ is upper relatively compact.
\end{lemma}

\begin{proof}
  Let $\eps > 0$.
  Let $(x_i)_{i \in I}$ be a sequence in $u_{\eps}(S)$.
  For each $i \in I$, choose $\alpha_i \in S$ such that
  $x_i \in \supp(\alpha_i)$.
  Since $S$ is relatively compact, there is a subset $I_0 \subset I$ and an $\alpha \in S$ such that the subsequence $(\alpha_i)_{i \in I_0}$ converges to $\alpha$.

  Since $S$ is relatively compact, by \cref{lem:rel-cpct-then-tot-bdd}, $S$ is totally bounded, and hence
  by \cref{lem:S_is_unif_upper_finite},
  $S$ is uniformly upper finite. In particular, there exists an $M$ such that for all $i \in I_0$, $\abs{u_{\frac{\eps}{4}}(\alpha_i)} \leq M$.
  Therefore, $\abs{u_{\frac{\eps}{2}}(\alpha)} \leq M$.
  Let $U = \supp(u_{\frac{\eps}{2}}(\alpha))$ and $V = \supp(u_{\eps}(\alpha))$.
  Then $V \subset U$ and $\abs{U} \leq M$.
  Let $\delta_1 = \eps - \max_{y \in U \setminus V} d(y,A)$.
  Then $\delta_1 > 0$.
  Let $\delta_2 = \min_{y \neq y' \in U} \frac{1}{2} d(y,y')$.
  Then $\delta_2 > 0$.
  Let $\delta = \min\{\delta_1, \delta_2,\frac{\eps}{2}\}$.

  Consider $\{B_{\delta}(y)\}_{y \in \supp(\alpha)}$.
  Since $\delta \leq \delta_1$ and $\delta \leq \frac{\eps}{2}$,
  if $y \in \supp(\ell_{\eps}(\alpha))$ then $B_{\delta}(y) \in A^{\eps}$.
  Since $\delta \leq \delta_2$,
  if $y \neq y' \in \supp(u_{\eps}(\alpha))$ then $B_{\delta}(y) \cap B_{\delta}(y') = \emptyset$.
  Since $(\alpha_i)_{i \in I_0} \to \alpha$, there exists a subsequence $(\alpha_i)_{i \in I_1}$, where $I_1 \subset I_0$ such that for all $i \in I_1$, $W_p(\alpha_i,\alpha) < \delta$.
  Therefore, for all $i \in I_1$, there exists a $y_i \in \supp(u_{\eps}(\alpha))$ such that $x_i \in B_{\delta}(y_i)$.
  Since there are only finitely many such $y_i$,
  there exists a $y \in \supp(u_{\eps}(\alpha))$ and
  a subsequence $(x_i)_{i \in I_2}$ where $I_2 \subset I_1$ such that
  for all $i \in I_2$, $x_i \in B_{\delta}(y)$.
  Finally, since $(\alpha_i)_{i \in I_2} \to \alpha$,
  $(x_i)_{i \in I_2} \to y$.  
\end{proof}

\begin{lemma} \label{lem:relatively-compact}
  Let $(X,d,A)$ be a metric pair, where $X \setminus A$ is relatively compact.
  Let $(\alpha_i)_{i \in I}$ be a sequence in $D^n(X,A)$.
  Then there exists a subsequence $(\alpha_j)_{j \in J}$ where $J \subset I$ and $0 \leq m \leq n$ such that for all $1 \leq k \leq m$, there exists $x_k \in X \setminus A$ such that
  $\alpha_j = \sum_{k=1}^m x_j^{(k)}$ and $(x_j^{(k)})_{j \in J} \to x_k$.
\end{lemma}

\begin{proof}
  The proof is by induction on $n$.
  If $n=0$ then we can take $m=0$ and $J=I$  and we are done.
  Assume the statement is true for $n-1$.
  If there exists a subsequence $J \subset I$ with $\alpha_j=0$ for all $j \in J$ then choose this sequence are we are done.
  If not, choose a subsequence $I' \subset I$ with $\abs{\alpha_i} \geq 1$ for all $i \in I'$.
  For each $i \in I'$ choose $x_i \in \supp(\alpha_i)$ and let $\beta_i = \alpha_i - x_i$.
  Consider the sequence $(x_i)_{i \in I'}$ in $X \setminus A$.
  Since $X \setminus A$ is relatively compact,
  there exists $x \in X \setminus A$ and a subsequence $(x_i)_{i \in I''}$ where $I'' \subset I'$ such that $(x_i)_{i \in I''} \to x$.
  Consider the sequence $(\beta_i)_{i \in I''}$.
  Now $\abs{\beta_i} \leq n-1$ for all $i \in I''$.
  Apply the induction hypothesis to get a subsequence $(\beta_j)_{j \in J}$, where $J \subset I''$ and $0 \leq m \leq n-1$ such that
  for all $1 \leq k \leq m$, there exists $x_k$ such that
  $\beta_j = \sum_{k=1}^m x_j^{(k)}$ and $(x_j^{(k)})_{j \in J} \to x_k$.
  Also, we have that $(x_j)_{j \in J} \to x$.
  Therefore, $(\alpha_j)_{j \in J}$ where $\alpha_j = x_j + \sum_{k=1}^m x_j^{(k)}$ gives us the desired subsequence.  
\end{proof}

\begin{proposition}
  Let $(X,d,A)$ be a metric pair.
  Consider a subset $S \subset \overline{D}(X,A)$ which is upper finite and upper relatively compact.
  Let $(\alpha_i)_{i \in I}$ be a sequence in $S$.
  Let $n \in \Z_{\geq 0}$.
  Then there exists a nested sequence of subsequences
  $(\alpha_i)_{i \in I} \supset
  (\alpha_i)_{i \in I_0} \supset 
  (\alpha_i)_{i \in I_1} \supset \cdots \supset
  (\alpha_i)_{i \in I_n}$,
  where $I \supset I_0 \supset I_1 \supset \cdots \supset  I_n$,
  $0 = N_{-1} \leq N_0 \leq N_1 \leq \cdots \leq N_n$
  and
  $x_1,\ldots,x_{N_n} \in X$
  such that
  for all $0 \leq m \leq n$ and for all $i \in I_m$
  \begin{equation*}
    u_{\frac{1}{m}}(\alpha_i) = \sum_{k=0}^m \sum_{j=N_{k-1}+1}^{N_k} x_i^{(j)},
  \end{equation*}
  where for $1 \leq j \leq N_0$,
  $d(x_i^{(j)},A) = \infty$,
  $d(x_j,A) = \infty$, and
  $(x_i^{(j)})_{i \in I_{0}} \to x_j$,
  and
  for $1 \leq k \leq m$ and $N_{k-1} < j \leq N_k$,
  $\frac{1}{k} \leq d(x_i^{(j)},A) < \frac{1}{k-1}$,
  $\frac{1}{k} \leq d(x_j,A) \leq \frac{1}{k-1}$, and
  $(x_i^{(j)})_{i \in I_m} \to x_j$.
\end{proposition}

\begin{proof}
  The proof is by induction on $n$.
  For $n=0$, apply \cref{lem:relatively-compact} to $(u_{\infty}(S) \amalg A, d, A)$ and $(u_{\infty}(\alpha_i))_{i \in I}$ to get
  $I_0 \subset I$, $0 \leq N_0$, $x_1,\ldots, x_{N_0} \in u_{\infty}(S)$ such that for all $i \in I_0$,
  $u_{\infty}(\alpha_i) = \sum_{j=1}^{N_0} x_i^{(j)}$,
  and for all $1 \leq j \leq N_0$,
  $d(x_i^{(j)},A) = \infty$,
  and $(x_i^{(j)}) \to x_j$.

  Assume the statement holds for $n$.
  Apply \cref{lem:relatively-compact} to $(u_{\frac{1}{n+1}}(S) \amalg A, d, A)$ and $(u_{\frac{1}{n+1}}(\alpha_i))_{i \in I_n}$ to get
  $I_{n+1} \subset I_n$, $N_n \leq N_{n+1}$, $x_{N_n+1},\ldots,x_{N_{n+1}} \in u_{\frac{1}{n+1}}(S)$ such that
  for all $i \in I_{n+1}$,
  $u_{\frac{1}{n+1}}(\alpha_i) = \sum_{k=0}^{n+1} \sum_{j=N_{k-1}+1}^{N_k} x_i^{(j)}$, and for all $N_n < j \leq N_{n+1}$, $\frac{1}{n+1} \leq d(x_i^{(j)},A) < \frac{1}{n}$, and $(x_i^{(j)})_{i \in I_{n+1}} \to x_j$.
\end{proof}

\begin{corollary} \label{cor:relatively-compact}
   Let $(X,d,A)$ be a metric pair.
  Consider a subset $S \subset \overline{D}(X,A)$ which is upper finite and upper relatively compact.
  Let $(\alpha_i)_{i \in I}$ be a sequence in $S$.
  Then there exists $I \supset I_0 \supset I_1 \supset I_2 \supset \cdots$
  with $I_m \isom \N$ for all $m \in \Z_{\geq 0}$,
  $0=N_{-1} \leq N_0 \leq N_1 \leq N_2 \leq \cdots$, and
  $x_1, x_2, \ldots \in X$
  such that
  for all $m \in \Z_{\geq 0}$ and for all $i \in I_m$,
  \begin{equation*}
    u_{\frac{1}{m}}(\alpha_i) = \sum_{k=0}^m \sum_{j=N_{k-1}+1}^{N_k} x_i^{(j)},
  \end{equation*}
  where for $1 \leq j \leq N_0$,
  $d(x_i^{(j)},A) = \infty$,
  $d(x_j,A) = \infty$, and
  $(x_i^{(j)})_{i \in I_0} \to x_j$, and
  for $1 \leq k \leq m$ and $N_{k-1} < j \leq N_k$,
  $\frac{1}{k} \leq d(x_i^{(j)},A) < \frac{1}{k-1}$,
  $\frac{1}{k} \leq d(x_j,A) \leq \frac{1}{k-1}$, and
  $(x_i^{(j)})_{i \in I_m} \to x_j$.
  Let $\alpha = \sum_{k=0}^{\infty} \sum_{j=N_{k-1}+1}^{N_k} x_j = x_1 + x_2+ \cdots
  \in \overline{D}_{\infty}(X,A)$.
\end{corollary}

\begin{lemma}
  Let $(X,d,A)$ be a metric pair and let $1 \leq p < \infty$.
  Consider a subset $S \subset (\overline{D}(X,A),W_p)$ which is upper finite, upper relatively compact, and uniformly $p$-vanishing.
  Let $(\alpha_i)_{i \in I}$ be a sequence in $S$.
  Let $\alpha \in \overline{D}_{\infty}(X,A)$ be given by \cref{cor:relatively-compact}.
  Then $\alpha \in \overline{D}_p(X,A)$.
\end{lemma}

\begin{proof}
  Let $\{I_{m}\}$, $\{N_{m}\}$, $\{x_{i}^{(j)}\}$, and $\{x_j\}$ be given by \cref{cor:relatively-compact}.
  Assume that $\alpha \not\in \overline{D}_p(X,A)$.
  Then there exists $\eps>0$ such that for all $m'$,
  $\left( \sum_{j > N_{m'}} d(x_j,A)^p\right)^{\frac{1}{p}} \geq \eps$.
  Since $S$ is uniformly $p$-vanishing, there exists an $m$ such that $W_p(\ell_{\frac{1}{m}}(\beta),0) < \frac{\eps}{4}$ for all $\beta \in S$.
  Since $\left( \sum_{j > N_{m}} d(x_j,A)^p\right)^{\frac{1}{p}} \geq \eps$,
  there exists an $n$ such that
  $\left( \sum_{j = N_m+1}^{N_n} d(x_j,A)^p\right)^{\frac{1}{p}} \geq \frac{\eps}{2}$.
  Note that this implies that $N_n > N_m$.

  Let $i \in I_n$.
  For $N_m < j \leq N_n$, $(x_i^{(j)})_{i \in I_n} \to x_j$.
  Therefore,  there exists an $M_j$ such that for all $i \geq M_j$,
  $d(x_i^{(j)}, x_j) < \frac{\eps}{4(N_n - N_m)}$.
  Let $M = \max_{N_m < j \leq N_n} M_j$.
  Then $d(x_j^{(M)},x_j) < \frac{\eps}{4(N_n-N_m)}$.
  Therefore,
  by the triangle inequality and the Minkowski inequality,
  \begin{align*}
    \left( \sum_{j=N_m+1}^{N_n} d(x_j,A)^p \right)^{\frac{1}{p}}
    &\leq
      \left( \sum_{j=N_m+1}^{N_n} \left( d(x_j,x_j^{(M)}) + d(x_j^{(M)},A) \right)^p \right)^{\frac{1}{p}} \\
    &\leq
      \left( \sum_{j=N_m+1}^{N_n} d(x_j,x_j^{(M)})^p \right)^{\frac{1}{p}} +
      \left( \sum_{j=N_m+1}^{N_n} d(x_j^{(M)},A)^p \right)^{\frac{1}{p}} \\
    &\leq \sum_{j=N_m+1}^{N_n} d(x_j,x_j^{(M)}) + W_p(\ell_{\frac{1}{m}}(\alpha_M),0)\\
    &< \frac{\eps}{4} + \frac{\eps}{4} = \frac{\eps}{2},
  \end{align*}
  which is a contradiction.
\end{proof}

\begin{proof}[\textbf{Proof of Theorem \ref{thm:criterion_rel_compct}}]
  Let $p \in [1,\infty]$.
    Let $S \subset (\overline{D}_p(X,A),W_p)$ be a relatively compact set.
    By \cref{lem:upper-relatively-compact}, $S$ is upper relatively compact.
    Furthermore, $S$ is totally bounded (\cref{lem:rel-cpct-then-tot-bdd}).
    So, by \cref{lem:S_is_unif_upper_finite}, $S$ is uniformly upper finite, and if $p < \infty$ then by \cref{lem:S_is_lower_uniform}, $S$ is uniformly $p$-vanishing.

  In the other direction, let $S \subset (\overline{D}_p(X,A),W_p)$ be a uniformly upper finite and upper relatively compact set. If $p < \infty$ then assume that $S$ is uniformly $p$-vanishing.
  Let $(\alpha_i)_{i \in I}$ be a sequence in $S$.
  Let $\{I_{m}\}$, $\{N_{m}\}$, $\{x_i^{(j)}\}$, $\{x_j\}$ and $\alpha$ be given by \cref{cor:relatively-compact}.
  Consider the subsequence $(\alpha_{i_j})_{j=1}^{\infty}$ of $(\alpha_i)_{i \in I}$ defined recursively as follows.
  Let $\alpha_{i_1}$ be the first term of $(\alpha_i)_{i \in I_1}$.
  Given $\alpha_{i_n}$, let $\alpha_{i_{n+1}}$ be the first term in $(\alpha_i)_{i \in I_{n+1}}$ that appears after $\alpha_{i_n}$ in the original sequence $(\alpha_i)_{i \in I}$.
  We will show that $(\alpha_{i_j})_{j=1}^{\infty} \to \alpha$ in $(\overline{D}_p(X,A),W_p)$.

  Let $\eps > 0$.
  Since $S$ is uniformly $p$-vanishing, there is an $m_1 \in \N$ such that
  for all $j$, $W_p(\ell_{\frac{1}{m_1}}(\alpha_{i_j}),0) < \frac{\eps}{3}$.
  Since $\alpha \in \overline{D}_p(X,A)$, there is an $m_2$ such that
  $W_p(\sum_{j > N_{m_2}} x_j, 0) < \frac{\eps}{3}$.
  Let $m = \max(m_1,m_2)$.
  Consider $(\alpha_{i_j})_{j=m}^{\infty}$.
  For $j \geq m$, $\alpha_{i_j} \in (\alpha_i)_{i \in I_j}$.
  Thus $u_{\frac{1}{j}}(\alpha_{i_j}) = \sum_{k=1}^{N_j} x_{i_j}^{(k)}$, where $N_j \geq N_m$.
  Hence $u_{\frac{1}{m}}(\alpha_{i_j}) = \sum_{k=1}^{N_m} x_{i_j}^{(k)}$.
  Therefore, 
  \begin{equation*}
   W_p ( \alpha_{i_j},\alpha) \leq  \left \|  W_p\left(u_{\frac{1}{m}}(\alpha_{i_j}),\sum_{k=1}^{N_m} x_{k}\right), W_p(\ell_{\frac{1}{m}}(\alpha_{i_j}),0),
    W_p\left(0,\sum_{k>N_m} x_{k}\right)  \right \|_p .
  \end{equation*}
  We have $W_p(\ell_{\frac{1}{m}}(\alpha_{i_j}),0) < \frac{\eps}{3}$ and
  $W_p(0,\sum_{k>N_m} x_{k}) < \frac{\eps}{3}$.
  For $1 \leq k \leq N_m$,
  $(x_{i_j}^{(k)})_{j=m}^{\infty} \to x_k$.
  So, there exists $M_k$ such that $d(x_{i_j}^{(k)},x_k) < \frac{\eps}{3N_m}$ for all $j \geq M_k$.
  Let $M = \max \{ m, M_1, \ldots, M_{N_m} \}$.
  For $j \geq M$, $W_p(\sum_{k=1}^{N_m} x_{i_j}^{(k)},\sum_{k=1}^{N_m} x_{k}) < \frac{\eps}{3}$.
  Therefore, $W_p(\alpha_{i_j},\alpha) < \eps$
  for all $j \geq M$. Hence $(\alpha_{i_j})_{j=1}^{\infty} \to \alpha$ and thus $S$ is relatively compact.
\end{proof}


\section{Curvature, dimension, and embeddability}

In our final section, we discuss the curvature and topological dimension of spaces of persistence diagrams as well their embeddability into a Hilbert space.

\subsection{Curvature}
      
In this subsection we generalize some results of \cite{turner2014frechet} and \cite{turner2013medians} on curvature.  
Let $(X, d)$ be a geodesic metric space and $k \in \R$. A \emph{geodesic triangle} $\bigtriangleup (x,y,z)$ in $X$ consists of three points $x, y, z \in X$ and three geodesic segments $[x, y], [y, z], [x, z]$ joining them.
Let $M_k$ denote the following spaces: if $k<0$, then $M_k$ is the hyperbolic space $\mathbb{H}^2$ with the distance function multiplied by $\frac{1}{\sqrt{-k}}$; if $k=0$, then $M_k$ is the Euclidean plane $\mathbb{E}^2$; if $k>0$, then $M_k$ is the unit sphere $\mathbb{S}^2$ with the distance function multiplied by $\frac{1}{\sqrt{k}}$.
A triangle  $\overline{\bigtriangleup} (\overline{x}, \overline{y}, \overline{z})$ in $M_k$ is a \emph{comparison triangle} for $\bigtriangleup (x,y,z)$
if $d(\overline{x}, \overline{y})=d(x, y)$, $d(\overline{y}, \overline{z})=d(y, z)$ and $d(\overline{x}, \overline{z})=d(x, z)$.
Geodesic triangle $\bigtriangleup (x,y,z)$ in $X$ is said to satisfy the \emph{$\CAT(k)$ inequality} if there is a comparison triangle $\overline{\bigtriangleup} (\overline{x},\overline{y},\overline{z})$ in $M_k$ such that the distance between points on $\bigtriangleup (x,y,z)$ are less or equal to the distance between the corresponding points on $\overline{\bigtriangleup} (\overline{x}, \overline{y}, \overline{z})$.
If $k \leq 0$, then $X$ is a \emph{$\CAT(k)$ space} if all its geodesic triangles satisfy the $\CAT(k)$ inequality.
If $k > 0$, then $X$ is a \emph{$\CAT(k)$ space}  if all its geodesic triangles of perimeter less than $\frac{2\pi}{\sqrt{k}}$ satisfy the $\CAT(k)$ inequality.
	
\begin{proposition}[{\cite[Proposition II 1.4(1)]{Bridson:1999}}]
  Let $X$ be a non-extended, separated $\CAT(k)$ space. There is a unique geodesic segment joining each pair of points $x, y \in X$ (provided $d(x, y) < \frac{\pi}{\sqrt{k}}$ if $k >0$), and its geodesic segment varies continuously with its endpoints.
\end{proposition}
	
The next result follows from  Proposition \ref{prop:unique_geodesic}.
	
\begin{corollary} \label{cor:cat}
  Let $(X, d, A)$ be a metric pair, where $(X,d)$ is separated and geodesic and $A$ is distance minimizing.
  Let $p \in [1,\infty]$ and $k \in \R$.
  Assume that there exist  $\alpha, \beta \in (\overline{D}_p(X,A),W_p)$ with
  $W_p(\alpha,\beta) < \infty$ such that $\alpha$ and $\beta$ have distinct optimal matchings and $W_p(\alpha, \beta) < \frac{\pi}{\sqrt{k}}$ if $k >0$.
  Then $(\overline{D}_p(X,A),W_p)$ is not a $\CAT (k)$ space.
\end{corollary}

Let $(X, d_X)$ be a geodesic space and let $k \in \R$.
Then $(X, d_X)$ is an \emph{Alexandrov space with curvature bounded from below by $k$}~\cite{bgp:1992} if
for every $x,y \in X$ with $d(x,y) < \infty$ (and $d(x,y) < \frac{\pi}{\sqrt{k}}$ if $k > 0$) and every constant speed geodesic $\gamma$ from $x$ to $y$ and every $z \in X$ with $d(x,z) < \infty$ (and $\perimeter(\bigtriangleup(x,y,z)) < \frac{2\pi}{\sqrt{k}}$ if $k > 0$),
and every $t \in [0,1]$,
$d_X(z,\gamma(t)) \geq d_{M_k}(\overline{z},\overline{\gamma}(t))$,
where $\overline{\gamma}$ is the unique constant speed geodesic from $\overline{x}$ to $\overline{y}$ in $\overline{\bigtriangleup}(\overline{x},\overline{y},\overline{z})$.
For $k=0$, we have the equivalent condition~\cite{ohta2012barycenters},
\begin{equation} \label{eq:alexandrov}
  d(z, \gamma(t))^2 \geq t d(z,y)^2+(1-t)d(z, x)^2-t(1-t)d(x, y)^2.
\end{equation}

Following \cite{turner2014frechet} we get the next result.
	
\begin{proposition}\label{prop:curvature}
  Let $(X, d, A)$ be a metric pair where $(X, d)$ is a geodesic  Alexandrov space with curvature bounded from below by zero and $A$ is distance minimizing. Then $(\overline{D}_2(X,A),W_2)$ is a geodesic Alexandrov space 
with curvature bounded from below by zero. 
\end{proposition}

\begin{proof}
  By Theorem \ref{thm:geodesic}, $(\overline{D}_2(X,A),W_2)$ is a geodesic space.
    Let $\alpha, \beta, \xi \in \overline{D}_2(X, A)$ with $W_2(\alpha, \beta) < \infty$ and $W_2(\alpha,\xi) < \infty$.
    Let $\gamma$ be a constant speed geodesic from $\alpha$ to $\beta$.
    Let $t \in [0,1]$.

    Let $\sigma$ be an optimal matching of $\alpha$ and $\gamma(t)$,
    let $\tau$ be an optimal matching of $\gamma(t)$ and $\beta$, and
    let $\rho$ be an optimal matching of $\xi$ and $\gamma(t)$.
    By adding elements of $A$ as needed we may write
    these matchings
    using a common indexing set. Write
    $\sigma = \sum_{i \in I}(x_i,w_i)$,
    $\tau = \sum_{i \in I}(w_i,y_i)$, and
    $\rho = \sum_{i \in I}(z_i,w_i)$.
    Note that $\alpha = \sum_{i \in I} x_i$, $\beta = \sum_{i \in I} y_i$,
    $\xi = \sum_{i \in I} z_i$ and $\gamma(t) = \sum_{i \in I} w_i$.
    Since $\gamma$ is a geodesic, we claim that it follows that $\sum_{i \in I}(x_i,y_i)$ is an optimal matching of $\alpha$ and $\beta$.
    Indeed,
    \begin{align*}
      \norm{(d(x_i,y_i))_{i \in I}}_2
      &\leq \norm{(d(x_i,w_i) + d(w_i,y_i))_{i \in I}}_2\\
      &\leq \norm{(d(x_i,w_i))_{i \in I}}_2 + \norm{(d(w_i,y_i))_{i \in I}}_2\\
      &= W_2(\alpha,\gamma(t)) + W_2(\gamma(t),\beta) = W_2(\alpha,\beta).
    \end{align*}
    Therefore, $W_2(\alpha,\beta)^2 = \sum_{i \in I}d(x_i,y_i)^2$.

    For all $i \in I$,
    let $\gamma'_i$ be a constant speed geodesic in $(X,d)$ from $x_i$ to $w_i$, and 
    let $\gamma''_i$ be a constant speed geodesic in $(X,d)$ from $w_i$ to $y_i$.
    Let $\tilde{\gamma}_i$ be the concatenation of $\gamma'_i$ and $\gamma''_i$.
    Then $\tilde{\gamma}_i$ is a constant speed geodesic in $(X,d)$ from $x_i$ to $y_i$ with $\tilde{\gamma}_i(t) = w_i$.
    Therefore, by \eqref{eq:alexandrov},
    \begin{equation*}
      d(z_i,w_i)^2  = d(z_i,\tilde{\gamma}_i(t))^2 \geq td(z_i,y_i)^2 + (1-t)d(z_i,x_i)^2 - t(1-t)d(x_i,y_i)^2.
    \end{equation*}
  Hence we have,
  \begin{align*}
    W_2(\xi,\gamma(t))^2
    &= \sum_{i \in I} d(z_i,w_i)^2\\
    &\geq t\sum_{i \in I}d(z_i,y_i)^2 + (1-t)\sum_{i \in I}d(z_i,x_i)^2 - t(1-t)\sum_{i \in I}d(x_i,y_i)^2\\
    &\geq t W_2(\xi,\beta)^2 + (1-t) W_2(\xi,\alpha)^2 - t(1-t) W_2(\alpha,\beta)^2. \qedhere
  \end{align*}

\end{proof}

From this we get the following.

\begin{corollary} \label{cor:infinite-hausdorff-dimension}
 Let $(X,d,A)$ be a metric pair, where $(X,d)$ is a geodesic Alexandrov space with non-negative curvature and $A$ is distance minimizing and not isolated.
  Then $(\overline{D}_2(X,A),W_2)$ has infinite Hausdorff dimension.
\end{corollary}

\begin{proof}
 By \cref{prop:curvature,thm:not_loc_compactness_D_p} $(\overline{D}_2(X,A),W_2)$
  is an Alexandrov space with non-negative curvature that is not locally compact.
  By~\cite[Theorem 10.8.1]{burago2001course}, 
$(\overline{D}_2(X,A),W_2)$
  has infinite Hausdorff dimension.
\end{proof}

\subsection{Topological dimension}\label{subsection:top_dimension}

A \emph{refinement} of an open cover $\{U_i\}$ is an open
cover $\{V_j\}$ such that for any $V_j$ there is $U_i$ so that
$V_j \subset U_i$. A collection of subsets of $X$ has \emph{order $n$} if
there is a point in $X$ contained in $n$ of the subsets, but no point
of $X$ is contained in $n+1$ of the subsets. The \emph{Lebesgue covering dimension} of a topological space $X$,
denoted $\dim(X)$,
is the smallest $n$ such that every open cover of $X$ has a refinement with
order $n+1$. 
	
	Below are some facts of the dimension theory (they are all described in \cite{bell2008asymptotic}).
	
	\begin{lemma}\label{lem:dim_of_discrete_set}
		If a topological space $X$ is discrete then $\dim(X)=0$.
	\end{lemma}
	
	\begin{proposition} \label{prop:dim_of_subset}
		If $(X, d)$ is a non-extended, separated metric space and $A \subset X$ then $\dim A \leq \dim X$.
	\end{proposition}
	
	\begin{theorem} \label{thm:dim_of_product}
		For a compact, finite-dimensional, non-extended, separated metric space $X$ either $\dim X^n=n \dim X $ or $ \dim X^n=n \dim X -n +1$.
	\end{theorem}
	
	\begin{theorem}[Hurewicz Mapping Theorem]\label{thm:Hurewicz}
		Let $f: X \to Y$ be a map between compact spaces. Then $\dim X \leq \dim Y + \dim f$, where $\dim f=\sup \big\{\dim f^{-1}(y) \ | \ y \in Y \big\}$.
	\end{theorem}
	

\begin{theorem}\label{thm:top_dimension}
  Let $p  \in [1, \infty]$. Let $(X, d, A)$ be a metric pair,
  where $X$ is non-extended and separated. Assume there is $x \in X$ and $\eps >0$ such that $\overline{B}_{\eps}(x)$ is compact,  $\dim \overline{B}_{\eps}(x) \geq 2$ and $\overline{B}_{\eps}(x) \cap A^{3\eps}=\emptyset$. Then $\dim (D^n(X, A), W_p) \geq n$ for all $n$ and hence  $\dim(D(X,A), W_p) = \infty$.
\end{theorem}
	
\begin{proof}
  Denote $\overline{B}_{\eps}(x)$ by $Y$.
  For any $y, y' \in Y$, $d(y, y')\leq 2\eps$, $d(y, A) \geq 3\eps$, $d(y', A) \geq 3\eps$ and hence $d(y, y') < d(y, A) + d(y', A)$.
  Therefore, $d_p(y, y')=d(y, y')$.
  Hence the inclusion map $i: (Y, d|_Y) \to (D(X,A), W_p)$
  is an isometric embedding.
  Let $n \in \N$.
  Similarly the inclusion $g: (Y^n/S_n, \overline{d_p^n}) \to (D^n(X, A), W_p)$ that sends $[y_1, \dots, y_n]$ to $\sum_{i=1}^n y_i$, where $\overline{d_p^n}$ is defined in Section \ref{sec:basic-metric} is an isometric embedding.
  Suppose $\dim Y=\infty$. By Proposition \ref{prop:dim_of_subset}, $\dim (D^n(X,A), W_p) \geq \dim i(Y) = \dim Y  =\infty$. Now suppose $\dim Y \neq \infty$. Since $Y$ is compact then, by Theorem \ref{thm:dim_of_product}, either $\dim Y^n=n \dim Y \geq 2n$ or $ \dim Y^n=n \dim Y -n +1 \geq n+1$. In both cases $\dim Y^n \geq n$. Now consider the quotient map $q: Y^n \to Y^n/S_n$. Note that $Y^n/S_n$ is compact and \[\dim q=\sup \{\dim q^{-1} ([y_1, \dots, y_n]) \ | \ [y_1, \dots, y_n] \in Y^n/S_{n} \}.\] Since $q^{-1}([y_1, \dots, y_n])=\big\{(y_{\sigma(1)}, \dots, y_{\sigma(n)}) \ \big| \ \sigma \in S_{n} \big\}$; which is a discrete set, by Lemma \ref{lem:dim_of_discrete_set}, $\dim q=0$. Therefore, by Theorem \ref{thm:Hurewicz}, $\dim (Y^n/S_{n}) \geq \dim Y^n - \dim q \geq n$. By Proposition \ref{prop:dim_of_subset}, $\dim (D^n(X, A), W_p)  \geq \dim (Y^n/S_n, \overline{d_p^n})  \geq n$. By Proposition \ref{prop:dim_of_subset}, $\dim (D(X, A), W_p) =\infty$.
	\end{proof}

\subsection{Embeddability into Hilbert space}

In this section we study the embeddability of spaces of persistence diagrams into Hilbert space.
In Section \ref{sec:unique-geodesics} we proved that whenever we have more than one optimal matching, we have non-unique geodesics and
hence the corresponding space of persistence diagrams cannot be isometrically embedded into Hilbert space.
The space of countable classical persistence diagrams cannot be coarsely embedded into a Hilbert space for $p \in (2, \infty]$ \cite{bubenik2020embeddings, wagner2019nonembeddability}.
However, the space of classical persistence diagrams on $n$ points can be coarsely embedded into Hilbert space~\cite{mitra2019space}.
We generalized this result to metric pairs.

Let $(X, d_X)$ and $(Y, d_Y)$ be metric spaces. A function $f: X \to Y$ is a \emph{coarse embedding} if there exist non-decreasing functions $\rho_1, \rho_2:[0,\infty] \to [0, \infty]$ satisfying
\begin{enumerate}
	\item $\rho_1(d_X(x,y)) \leq d_Y(f(x), f(y)) \leq \rho_2(d_X(x,y))$ for all $x, y \in X$,
	\item $\lim_{t \to \infty} \rho_1(t)=+\infty.$
\end{enumerate}
It is easy to check that the composition of two coarse embeddings is a coarse embedding.

\begin{lemma}
   Let $(X, d_X)$ and $(Y, d_Y)$ be metric spaces. If $\infty \in \im(d_X)$ and $\infty \notin \im(d_Y)$ then $(X, d_X)$ does not coarsely embed into $(Y, d_Y)$.
\end{lemma}	 

\begin{proof}
  Assume $f:(X, d_X) \to (Y, d_Y)$ is a coarse embedding.
  Let $x, x' \in X$ such that $d_X(x, x')=\infty$. Then for any non-decreasing function $\rho_1:[0, \infty] \to [1, \infty]$ such that $\lim_{t \to \infty} \rho_1(t)=+\infty$ we have  $\rho_1(\infty)=\infty$. By definition, $\infty=\rho_1(d_X(x, x')) \leq d_Y(f(x), f(x')) < \infty$. We have reached a contradiction.
\end{proof}

\begin{corollary}
  If $(X, d_X)$ is a metric space and $\infty \in \im(d_X)$ then $(X, d_X)$ does not coarsely embed into a Hilbert space.
\end{corollary}

Therefore, we may restrict to non-extended metric spaces.


Two metric spaces $X$ and $Y$ are \emph{coarsely equivalent} if there exists coarse embeddings $f:X \to Y$ and $g:Y \to X$ and a constant $K$ such that for all $x \in X$, $d_X(x,g(f(x))) \leq K$ and for all $x \in Y$, $d_Y(y,f(g(y))) \leq K$.

\begin{example}Let $(X, d)$ be a metric space. Define the equivalence relation $x \sim y$ if $d(x,y)=0$. Let $X/{\sim}$ be the quotient metric space, which is a separated metric space called the \emph{Kolmogorov quotient}. Then $X$ and $X/{\sim}$ are coarsely equivalent.
Similarly, for $x_0 \in X$, $n \geq 0$, and $p \in [1,\infty]$, $(D^n(X,x_0),W_p)$ and $(D^n(X/{\sim},x_0),W_p)$ are coarsely equivalent.
\end{example}



  If $X$ and $Y$ are coarsely equivalent then there is a coarse embedding of $X$ in $Z$ if and only if there is a coarse embedding of $Y$ in $Z$.
It follows that we may assume that our metric spaces are separated.

We will show that spaces of persistence diagram of cardinality at most $n$ embed into Hilbert space because they have finite asymptotic dimension, which we now define.

Let $X$ be a metric space and $n \in \Z_{+}$. A family $\{U_i\}$ of subsets of $X$ is \emph{uniformly bounded} if $\sup_{i} \diam(U_i) < \infty$. The \emph{asymptotic dimension} of $X$ does not exceed n ($\asdim X \leq n$) if for every uniformly bounded open cover $\{V_j\}$ of $X$ there is a uniformly bounded open cover $\{U_i\}$ of $X$ of order $\leq n + 1$ so that $\{V_j\}$ refines $\{U_i\}$. And $\asdim X = n$ if $\asdim X \leq n$ and $\asdim X \nleq n-1$.
\begin{example}
 $\asdim \R^n=n$ for any integer $n \geq 1$.
\end{example}
\begin{example}
	If $X$ is a bounded metric space then $\asdim X =0$.
\end{example}
	\begin{proposition}[{\cite{bell2008asymptotic}\label{prop:asdim_subset}}]
		Let $X$ be a metric space and $Y \subset X$. Then $\asdim Y \leq \asdim X$.
	\end{proposition}
	\begin{theorem}[{\cite{bell2008asymptotic}\label{thm:asdim_product}}]
		Let $X$ and $Y$ be metric spaces. Then $\asdim X\times Y \leq \asdim X + \asdim Y$.
	\end{theorem}
	
	
	\begin{theorem}[{\cite{kasprowski2017asymptotic}\label{thm:asdim_X/F}}]
		Let $X$ be a proper metric space and $F$ be a finite group acting on $X$ by isometries. Then $X/F$ has the same asymptotic dimension as that of $X$.
	\end{theorem}
		
	\begin{proposition}[{\cite{bell2008asymptotic}\label{prop:asdim_coarse_eq}}]
		If $f:X \to Y$ is a coarse embedding then $\asdim X \leq \asdim Y$. Moreover, if $f$ is a coarse equivalence then  $\asdim X = \asdim Y$.
	\end{proposition}

\begin{theorem}[{\cite{roe2003lectures}\label{thm:finite_asdim_embedding}}]
	A non-extended metric space of finite asymptotic dimension coarsely embeds in a Hilbert space.
\end{theorem}

\begin{proposition}\label{prop:coarse_emb_D^n_pointed_case} 
	Let $p \in [1, \infty]$. Let $(X, d, x_0)$ be a pointed metric space where $(X,d)$ is non-extended proper and $\asdim X < \infty$. Then for every $n \in \N$ $(D^n(X, x_0), W_p)$ can be coarsely embedded into Hilbert space.
\end{proposition}

\begin{proof}
By Theorem \ref{thm:asdim_product}, $\asdim X^{2n} < \infty$. Since $X$ is proper, so is $X^{2n}$. Since the symmetric group $S_{2n}$ acts by isometries on $X^{2n}$, by Theorem \ref{thm:asdim_X/F}, $\asdim X^{2n}/S_{2n} < \infty$. Furthermore, by Lemma \ref{lem:D^n_isom_embedding} and Proposition \ref{prop:asdim_coarse_eq}, $\asdim D^n(X, x_0) < \infty$. Finally, by Theorem \ref{thm:finite_asdim_embedding}, $(D^n(X, x_0), W_p)$ coarsely embeds into Hilbert space.
\end{proof}


In the remainder of this section we prove that under mild hypotheses, the space of finite persistence diagrams has infinite asymptotic dimension, so we may not use \cref{thm:finite_asdim_embedding} to obtain a coarse embedding into Hilbert space.

Let $p \in [1, \infty]$. Let $\R_+$ denote the set of non-negative reals. For any $r, r' \in \R_+$, the function $\rho(r, r')=|r-r'|$ defines a metric on $\R_+$. Then $(\R_+, \rho, 0)$ is a pointed metric space.
Recall from \cref{sec:basic-metric} that the product $\R_+^n$ as the $p$-product metric $\rho_p^n$ and the quotient $R_+^n/S_n$ has the quotient metric $\overline{\rho_p^n}$.

\begin{lemma}\label{lem:R_+^n/S_n_isometry}
		For each $n \in \N$, $(\R_+^n/S_n, \overline{\rho^n_p})$ and $(D^n(\R_+, 0), W_p)$  are isometric.
\end{lemma}

\begin{proof}
  Define a map $f: (\R_+^n/S_n, \overline{\rho^n_p}) \to (D^n(\R_+, 0), W_p)$ by $f([(r_1, \dots, r_n)])=r_1+\dots+r_n$. It is easy to check that this map is well defined and is a bijection.
  Let $(r_1,\dots,r_n), (s_1,\dots,s_n) \in \R_+^n$.
  Since for any $r, r' \in \R_+ \setminus \{0\}$,
  $\rho(r, r') < \max \{\rho(r, 0),\rho(0, r')\} $,
	\[	\overline{\rho^n_p}((r_1,\dots,r_n), (s_1,\dots,s_n))	=\min_{\sigma \in S_n} \norm{(\rho(r_i, s_{\sigma(i)}))_{i=1}^n}_p= W_p(r_1+\dots+r_n, s_1+\dots s_n) .\] Hence $f$ is an isometry.
	\end{proof}
	
	\begin{lemma}[{\cite{kuchaiev2008coarse}\label{lem:Kuchaiev}}]
		Every unbounded proper geodesic metric space has a geodesic ray (i.e. an isometric copy of $\R_+$).
	\end{lemma}

\begin{lemma}\label{lem:D^n(R_+,0)_isometry}
  Let $(X, d, x_0)$ be a pointed metric space where $(X, d)$ is unbounded proper geodesic. Then for every $n \in \N$, $(D^n(\R_+, 0), W_p)$ coarsely embeds into $(D^n(X, x_0), W_p)$.
  Furthermore, $\asdim (D^n(X,x_0), W_p) \geq n$.
\end{lemma}

\begin{proof} By Lemma \ref{lem:Kuchaiev}, there is an isometric embedding $\phi:(\R_+, \rho) \to (X, d)$. We can assume that $\phi(0)=x_0$. For any $r, r' \in \R_+$, $d(\phi(r), \phi(r'))=\rho(r, r')$ and hence for any $r, r' \in \R_+ \setminus \{0\}$,  $d(\phi(r), \phi(r')) < d(\phi(r), x_0)+ d(x_0, \phi(r'))$. Now define a map $\tilde{\phi}: (D^n(\R_+, 0), W_p) \to (D^n(X, x_0), W_p)$ by $\tilde{\phi}(r_1+\dots+r_n)=\phi(r_1)+\dots+\phi(r_n)$. It is easy to check that map $\tilde{\phi}$ is well-defined. Let $r_1+\dots+r_n, s_1+\dots, s_n \in D^n(\R_+, 0)$. Then
\begin{align*}
W_p(r_1+\dots+r_n &,  s_1+\dots+s_n) =\min_{\sigma \in S_n} \norm{(\rho(r_i, s_{\sigma(i)}))_{i=1}^n}_p \\
	&=\min_{\sigma \in S_n} \norm{(d(\phi(r_i), \phi(s_{\sigma(i)})))_{i=1}^n}_p=W_p(\phi(r_1)+\dots+\phi(r_n), \phi(s_1)+\dots+\phi(s_n)).
\end{align*}
Hence $\tilde{\phi}$ is an isometric embedding.

By Lemma \ref{lem:R_+^n/S_n_isometry},
$\asdim (D^n(\R_+, 0), W_p) = n$.
Therefore, by Lemma \ref{lem:D^n(R_+,0)_isometry} and Proposition \ref{prop:asdim_coarse_eq}, we have that $\asdim (D^n(X, x_0), W_p ) \geq n$.
\end{proof}

\begin{proposition}\label{prop:asdim_is_infty}
Let $(X,d,A)$ be a metric pair where $X$ is geodesic, $(X/A, d_{1})$ is unbounded and proper, and $A$ is distance minimizing. Then \[\asdim (D(X,A),W_p)=\infty.\]
\end{proposition}

\begin{proof}
  Let $n \in \N$.
  By Proposition \ref{prop:X/A_geodesic}, $(X/A, d_{1})$ is a geodesic space.
From Lemma \ref{lem:Kuchaiev} and Proposition \ref{prop:asdim_coarse_eq} it follows that $\asdim (X/A, d_{1}) \geq 1$.
By Lemma \ref{lem:D^n(R_+,0)_isometry}, $\asdim (D^n(X/A, W_p[d_1])) \geq n$.
Moreover, by Lemma \ref{lem:isomorphism}, $W_p[d_1]=W_p[d]$ and $\asdim(D^n(X, A), W_p) \geq n$. Finally, since $n$ was arbitrary $\asdim (D(X, A), W_p) =\infty$.
\end{proof}



        

\subsection*{Acknowledgments}

This research was partially supported by the Southeast Center for Mathematics and Biology, an NSF-Simons Research Center for Mathematics of Complex Biological Systems, under National Science Foundation Grant No. DMS- 1764406 and Simons Foundation Grant No. 594594. This material is based upon work supported by, or in part by, the Army Research Laboratory and the Army Research Office under contract/grant number W911NF-18-1-0307.
The authors would like to thank Alex Elchesen and Alexander Dranishnikov for helpful discussions and comments.
The authors would like to thank the referees whose careful reading led to many improvements in our paper.

On behalf of all authors, the corresponding author states that there is no conflict of interest.
	

\end{document}